\documentclass[a4paper]{article}

\usepackage[utf8]{inputenc} 
\usepackage[T1]{fontenc} 
\usepackage{amsmath,amsthm,latexsym} 
\usepackage[fixamsmath]{mathtools}
\mathtoolsset{showmanualtags,mathic,centercolon} 
\usepackage{amssymb,amsfonts} 
\usepackage{booktabs} 
\usepackage[UKenglish]{babel} 
\usepackage[svgnames]{xcolor} 
\usepackage{algpseudocode} 
\usepackage{algorithm} 
\usepackage{url}

\usepackage[margin=0.9in]{geometry} 
\usepackage[labelfont=sl,textfont=sl]{subcaption} 
\usepackage[font={small,sl},labelsep=period]{caption} 

\numberwithin{equation}{section} 
\theoremstyle{plain}
\newtheorem{theorem}{Theorem}[section] 
\newtheorem{lemma}[theorem]{Lemma} 
\newtheorem{corollary}[theorem]{Corollary}

\theoremstyle{definition}
\newtheorem{definition}[theorem]{Definition}

\newtheorem{assumption}[theorem]{Assumption}

\theoremstyle{remark}
\newtheorem{remark}[theorem]{Remark}

\newcommand{\appref}[1]{Appendix~\ref{#1}}
\newcommand{\secref}[1]{Section~\ref{#1}}
\newcommand{\defref}[1]{Definition~\ref{#1}}
\newcommand{\thmref}[1]{Theorem~\ref{#1}}
\newcommand{\lemref}[1]{Lemma~\ref{#1}}

\renewcommand{\algref}[1]{Algorithm~\ref{#1}}
\newcommand{\assref}[1]{Assumption~\ref{#1}}
\newcommand{\figref}[1]{Figure~\ref{#1}}

\newcommand{\R}{\mathbb{R}} 
\newcommand{\N}{\mathbb{N}} 
\newcommand{\bigO}{\mathcal{O}} 
\DeclareMathOperator*{\argmin}{arg\,min} 
\DeclareMathOperator*{\argmax}{arg\,max} 
\newcommand{\defeq}{:=} 
\newcommand{\grad}{\nabla} 

\def\be{\begin{equation}}
\def\ee{\end{equation}}

\renewcommand{\b}[1]{\mathbf{#1}} 
\renewcommand{\t}[1]{\widetilde{#1}} 

\newcommand{\bx}{\b{x}}
\newcommand{\by}{\b{y}}
\newcommand{\bs}{\b{s}}
\newcommand{\br}{\b{r}}
\newcommand{\bee}{\b{e}}
\newcommand{\bg}{\b{g}}
\newcommand{\bem}{\b{m}}

\algrenewcommand\algorithmicrequire{\textbf{Input:}}
\algrenewcommand\algorithmicensure{\textbf{Output:}}

\begin{document}
\title{Improving the Flexibility and Robustness of Model-Based Derivative-Free Optimization Solvers}
\author{
	Coralia Cartis \thanks{Mathematical Institute, University of Oxford, Radcliffe Observatory Quarter, Woodstock Road, Oxford, OX2 6GG, United Kingdom (\texttt{cartis@maths.ox.ac.uk}).}
	\and
	Jan Fiala \thanks{Numerical Algorithms Group, Wilkinson House, Jordan Hill Road, Oxford, OX2 8DR, United Kingdom (\texttt{jan.fiala@nag.co.uk}).}
	\and
	Benjamin Marteau \thanks{Numerical Algorithms Group, Wilkinson House, Jordan Hill Road, Oxford, OX2 8DR, United Kingdom (\texttt{benjamin.marteau@nag.co.uk}).}
	\and
	Lindon Roberts \thanks{Mathematical Institute, University of Oxford, Radcliffe Observatory Quarter, Woodstock Road, Oxford, OX2 6GG, United Kingdom (\texttt{robertsl@maths.ox.ac.uk}). This work was supported by the EPSRC Centre For Doctoral Training in Industrially Focused Mathematical Modelling (EP/L015803/1) in collaboration with the Numerical Algorithms Group Ltd.}
}
\date{\today}
\maketitle

\begin{abstract}
We present DFO-LS, a software package for derivative-free optimization (DFO) for nonlinear Least-Squares (LS) problems, with optional bound constraints. 
Inspired by the Gauss-Newton method, DFO-LS  constructs simplified linear regression models for the residuals. DFO-LS allows flexible initialization for expensive problems,
whereby it can begin making progress from as few as two objective evaluations.
Numerical results show DFO-LS can gain reasonable progress on some medium-scale problems with fewer objective evaluations than is needed for one gradient evaluation. 
DFO-LS has  improved robustness to noise, allowing sample averaging, the construction of regression-based models, and multiple restart strategies together with an auto-detection mechanism. Our extensive numerical experimentation shows that restarting the solver when stagnation is detected is a cheap and effective mechanism for achieving robustness, with superior performance over both sampling and regression techniques. We also present our package Py-BOBYQA, a Python implementation of BOBYQA (Powell, 2009), which also implements robustness to noise strategies. Our numerical experiments show that Py-BOBYQA is comparable to or better than existing general DFO solvers 
for noisy problems. In our comparisons, we introduce a new adaptive measure of accuracy for the data profiles of noisy functions that strikes a balance between measuring the true and the noisy objective improvement.

\end{abstract}

\\



\section{Introduction}

The ability to solve optimization problems in the absence of derivative information --- known as derivative-free optimization (DFO) --- is an important goal for optimization software.
The need for DFO software particularly arises when function evaluations are expensive (so finite differencing is too costly), or when evaluations are noisy (so the accurate evaluation of derivatives is impossible).
A state-of-the-art  category of DFO algorithms are the so-called `model-based' methods.
These methods are similar to classical trust-region methods, which require the iterative minimization of local models for the objective over a trust-region ball, except the local models are constructed by interpolation instead of using derivative information.
Model-based DFO solvers are known to capture curvature in the objective well \cite{Custodio2017}, and have good practical performance \cite{More2009}.

In this paper, we focus on improving the flexibility and robustness of model-based DFO solvers for two regimes: 
\begin{description}
	\item[\normalfont\textit{Expensive:}] objectives which may be noiseless but expensive to evaluate. Here, the goal is to make reasonable progress, not necessarily reaching high accuracy in the
solution, using very few evaluations; and,
	\item[\normalfont\textit{Noisy:}] objectives which are cheap(er) to evaluate but whose evaluation may contain noise. We aim to improve the robustness of the solver --- maximizing the amount of progress the solver can make, and hence, the number of problems that can be solved despite the difficulties associated with inaccurate local models and objective evaluations.
\end{description}
Clearly, the two regimes may overlap, in which case we still aim and show that we can make reasonable progress in our proposed algorithms.
We are particularly interested in solving unconstrained (or possibly bound-constrained) nonlinear least-squares problems, but also consider general nonlinear objectives.

Regarding the `expensive' regime, model-based DFO solvers typically require at least $n+1$ objective evaluations (for an $n$-dimensional problem)   before they can begin the first iteration;
this evaluation cost represents the cost of setting up the first local model, from scratch, while subsequent iterations commonly
only update the interpolation set and the local model at a much lower evaluation cost.
However, in the `expensive' regime, this start-up cost may be prohibitive, and the user may wish to see decreases in the objective much sooner.
Direct search DFO solvers, such as BFO by Porcelli and Toint \cite{Porcelli2017}, can make progress with very few objective evaluations, but this flexibility is not generally found in model-based DFO
methods.

For the `noisy' regime, model-based DFO solvers can generally  make some progress on a problem, but often stagnate at incorrect solutions, without even  using  the full computational budget provided by the user; see \figref{fig_restarts_motivation}, for instance.
Two main methods have been suggested for robustly handling noisy objectives in a model-based DFO context.
Sample averaging is the most common approach for handling noisy evaluations; see  \cite{Deng2006,Deng2009,Shashaani2016,Chen2016}.
For theoretical convergence guarantees to hold, one must compute $\bigO(\Delta_k^{-4})$ samples of the objective at each point (e.g.~\cite{Chen2016}), where $\Delta_k$ is the trust region radius at iteration $k$.
However, this requirement rapidly becomes infeasible, so $\bigO(\Delta_k^{-1})$ samples is a more sensible choice in practice  \cite{Chen2016}.
The other main approach is to build regression models (i.e.~having more interpolation points than degrees of freedom in the model) rather than interpolation models \cite{Conn2008,Billups2013,Chen2016}.
We note in particular the STORM algorithm from Chen, Menickelly and Scheinberg \cite{Chen2016}, which uses $\bigO(\Delta_k^{-1})$ interpolation points at each iteration $k$, and determines whether a step gives sufficient objective decrease by averaging over $\bigO(\Delta_k^{-1})$ samples.
In both cases,  there is a tradeoff between robustness of the solver and performance in early phases, where the latter is very slow as sampling and regression require a large amount of problem information to accumulate before starting to generate substantial objective improvement. 


An alternative approach for the `noisy' regime is used in SNOWPAC by Augustin and Marzouk \cite{Augustin2017}.
This solver extends a previous model-based DFO code for constrained nonlinear programs by the same authors, NOWPAC \cite{Augustin2014}, by constructing a Gaussian Process surrogate model for the noisy objective, from previously-seen objective values and standard errors.
This approach avoids the performance loss in early phases, however it requires the user to provide standard error estimates for each objective evaluation, and introduces potentially expensive surrogate model construction steps, especially when using a large set of observations.
Here, we are particularly interested in nonlinear least-squares problems, where we build local models for each residual separately.
In this context especially, building surrogate models may prohibitively expensive.

\paragraph{Algorithm development and software contributions}
In this paper we introduce a new model-based DFO package in Python for nonlinear least-squares problems with optional bound constraints, which we call DFO-LS (Derivative-Free Optimization for Least-Squares). It builds on our previous code for nonlinear least-squares, DFO-GN \cite{Cartis2017a}, in that it continues to use linear local models for each residual function (rather than quadratic), which reduces the computational cost of the interpolation step.
DFO-LS has a wide variety of additional default and optional features,  that can be used on their own or in combination, with defaults selected based on extensive testing. These features,
apart from averaging and regression sampling, are novel for model-based DFO solvers. 
The most notable of these features are:
\begin{description}
	\item[\normalfont\textit{Reduced Initialization Cost:}] The ability to begin the main iteration after as few as 2 objective evaluations (as opposed to at least $n+1$ for an $n$-dimensional problem in other solvers);
	\item[\normalfont\textit{Multiple Default Parameter Choices:}] The modification of some algorithm parameters (such as trust-region parameters, termination criteria) to more appropriate values if the objective function is noisy;
	\item[\normalfont\textit{Sample Averaging \& Regression:}] The optional use of sample averaging (allowing an extensive range of sampling methodologies) and/or regression-based model construction; and,
	\item[\normalfont\textit{Multiple Restarts:}] The use of multiple restarts to allow greater exploration of the search space for noisy objectives. Although this feature is novel in the model-based DFO setting, similar techniques have been commonly used in numerical analysis, such as  multiple restarts of nonlinear conjugate gradient methods \cite[Chapter 5]{Nocedal2006} and GMRES \cite[Chapter 6]{Demmel1997},  multi-starting local solvers  in global optimization \cite{Locatelli2013}, as well as for  robustness improvement of the Nelder-Mead algorithm \cite{Kelley1999}. In our results, we find that multiple restarts greatly enhance the performance of DFO-LS for noisy problems, yielding superior performance even compared to DFO-LS with a high level of sample averaging. In particular, we note that the multiple restarts approach avoids the early loss of performance typical of sample averaging and regression, 
does not require extra user input common to  surrogate model approaches,	and is cheap to implement.	
\end{description}
The `reduced initialization cost' feature is designed for the `expensive' regime; the others are designed for the `noisy' regime.
We additionally demonstrate that these regimes are not mutually exclusive: using a reduced initialization cost works similarly well for noisy problems (as for noiseless problems), and multiple restarts can sometimes improve performance, including escaping local minima, for noiseless problems.

Some of the above features of DFO-LS are not closely tied to the least-squares problem structure.
Hence, in this paper we also introduce a package for general objective problems with optional bound constraints, Py-BOBYQA, so named as it is a Python implementation of Powell's BOBYQA \cite{Powell2009}.
In particular, Py-BOBYQA implements \textit{multiple default parameter choices}, \textit{sample averaging}, and \textit{multiple restarts}.

\paragraph{Testing Framework Contribution}
We also propose an improvement to the measurement standards of solver performance for noisy problems.
As detailed in \cite{More2009}, data and performance profiles are useful measures for comparing DFO solvers on a standard given test set, which measure the number of objective evaluations required to reach an objective value below a problem- and accuracy-dependent threshold. We assume that  a collection of deterministic test problems is used --- such as Mor\'e \& Wild or CUTEst ---
and that noisy variants of each problem are created by perturbing the objective or residual functions by multiplicative or additive stochastic noise. In this context, 
one can check decrease using either the value of the true (noiseless) objective, or the actual (noisy) objective seen by the solver; these two approaches are used, for instance, in \cite{Chen2016} and \cite{Billups2013} respectively.
In this paper, we show that these two measures produce similar results until a problem- and noise-specific accuracy level is reached; beyond this cut-off level, measured performance is better when the `noisy objective'  is used due, most commonly, to successful sampling (rather than optimization). 
As a result, we propose  showing profile results using an adaptive accuracy level; namely,
at the desired accuracy level whenever the latter is larger than the per-problem accuracy cut-off, and at the cut-off accuracy level, otherwise.
We illustrate that this approach is  a fairer approach for testing which focuses on genuine objective reductions rather than `lucky' sampling errors.



\paragraph{Comparisons to Related Software}
In our numerical results, we compare DFO-LS to DFO-GN \cite{Cartis2017a} and DFBOLS \cite{Zhang2010}, also designed for nonlinear least-squares problems\footnote{There is only 
one other nonlinear least squares DFO solver that we are aware of, namely, POUNDERS \cite{Wild2017}. We have already compared it against DFO-GN
and DFBOLS in \cite{Cartis2017a}.}.
We find that using different default parameters for noisy problems, coupled with multiple restarts, makes DFO-LS have substantially improved robustness to noise over both DFO-GN and DFBOLS, without the early loss of performance associated with sample averaging and regression models.
We also find that using a reduced initialization cost for medium-scale problems ($n\approx 100$ dimensions) allows DFO-LS to make reasonable progress on some problems with fewer than $n$ objective evaluations, but with a slight performance penalty for medium-sized budgets.

As mentioned above, the general-objective solver Py-BOBYQA is based on the original package by Powell \cite{Powell2009,Zhang_URL}.
In our testing, we compare Py-BOBYQA with the original BOBYQA, together with (S)NOWPAC \cite{Augustin2014,Augustin2017}, and our own implementation\footnote{\:There are several versions of the STORM algorithm given for different noise settings. We use the version of STORM designed for unbiased noise, which builds regression models from independent samples at every iteration, because it showed better performance than other variants.} of STORM \cite{Chen2016}.
In our testing for noisy problems, we find that the different default parameters and multiple restarts in Py-BOBYQA means it substantially outperforms BOBYQA.
It achieves a similar or better level of robustness than SNOWPAC and STORM, but with a mechanism which is cheap to implement and does not penalize performance in early phases.


\paragraph{Software Availability}
The two Python packages in this paper, DFO-LS and Py-BOBYQA, are available on Github\footnote{\:See \url{https://github.com/numericalalgorithmsgroup/dfols} and \url{https://github.com/numericalalgorithmsgroup/pybobyqa} respectively. Versions 1.0.1 of both packages were used for all the testing below.}.
They are released under the GNU General Public License.

\paragraph{Paper Structure}
In \secref{sec_algo_framework}, we introduce the general DFO-LS algorithm.
Details about the new features of DFO-LS are given in \secref{sec_implementation}.
We summarize the testing framework, including the modification of testing criteria for noisy problems, in \secref{sec_testing_framework}.
Then, in \secref{sec_new_features_results}, we provide a collection of different studies, showing numerical results for the key new features in DFO-LS.
The final section on DFO-LS is \secref{sec_dfols_benchmarking}, where we compare its performance against other derivative-free nonlinear least-squares solvers.
Lastly, in \secref{sec_pybobyqa}, we introduce the Py-BOBYQA solver by summarizing the model construction process for general objective minimization and detailing the features from DFO-LS which Py-BOBYQA inherits.
This section also shows numerical results comparing Py-BOBYQA against other model-based derivative-free solvers for general objective problems.
We summarize our results and conclude in \secref{sec_conclusion}.

\section{General Algorithmic Framework} \label{sec_algo_framework}
The DFO-LS software is designed to solve the nonlinear least-squares problem\footnote{\:Note that in line with the implementation of DFO-LS, we do not have a constant $1/2$ factor in \eqref{eq_ls_definition}.}
\be \min_{\bx\in\R^n}\: f(\bx) \defeq \|\br(\bx)\|^2 = \sum_{i=1}^{m}r_i(\bx)^2, \label{eq_ls_definition} \ee
where $\br(\bx) \defeq [r_1(\bx) \: \cdots \: r_m(\bx)]^{\top}$ is a continuously differentiable function from $\R^n$ to $\R^m$, but its Jacobian matrix of first derivatives is unavailable. Both the case when $m\geq n$ (least-squares) and $m\leq n$ (inverse problems) are allowed.
Lastly, we use $\|\cdot\|$ for the 2-norm of vectors and matrices (i.e.~Euclidean norm and largest singular value respectively) unless otherwise specified, and for $\bx\in\R^n$ and $\Delta>0$, we define $B(\bx,\Delta):=\{\by\in\R^n : \|\by-\bx\|\leq\Delta\}$.

\subsection{Regression Interpolation Models} \label{sec_regression_models}
DFO-LS constructs a linear model for $\br(\bx)$ in a neighbourhood of the current iterate $\bx_k$ at every iteration.
To achieve this in a derivative-free way, we maintain a set of $p+1$ points $Y_k=\{\bx_k, \by_1, \ldots, \by_p\}\subset\R^n$, where we let $\by_0\defeq \bx_k$ for notational convenience.
The usual regime has $p\geq n$, but if needed, we also allow $p<n$ in early iterations in order to reduce the initial evaluation cost of DFO-LS; both constructions are described here. The default option in DFO-LS is to initialize with a full set of $n+1$ points (so $p=n$).

When $p\geq n$, we build a model
\be \br(\bx_k+\bs) \approx \bem_k(\bs) \defeq \br_k + J_k\bs, \label{eq_linear_models} \ee
by solving the regression problem
\be \min_{\br_k,J_k} \: \sum_{t=0}^{p} \|\bem_k(\by_t-\bx_k) - \br(\by_t)\|^2. \label{eq_interp_conditions} \ee
This corresponds to finding the least-squares solutions to the overdetermined linear systems
\be W_k \begin{bmatrix}r_{k,i} \\ \b{j}_{k,i}\end{bmatrix} \defeq \begin{bmatrix}1 & (\by_0-\bx_k)^{\top} \\ \vdots & \vdots \\ 1 & (\by_p-\bx_k)^{\top}\end{bmatrix}\begin{bmatrix}r_{k,i} \\ \b{j}_{k,i}\end{bmatrix} = \begin{bmatrix}r_i(\by_0) \\ \vdots \\ r_i(\by_p)\end{bmatrix}, \label{eq_linear_interp_system} \ee
for all $i=1,\ldots,m$, where $r_{k,i}$ and $\b{j}_{k,i}^{\top}$ are the $i$-th entry of $\br_k$ and row of $J_k$ respectively.
The matrix $W_k$ has full column rank whenever the set $\{\by_1-\bx_k, \ldots, \by_p-\bx_k\}$ spans $\R^n$; we ensure this in DFO-LS by calling procedures to improve the geometry of $Y_k$ (in a specific sense discussed below).
However, as the algorithm progresses, the points $\by_t$ get progressively closer to $\bx_k$, so $W_k$ becomes ill-conditioned.
To avoid this issue, we precondition \eqref{eq_linear_interp_system} by scaling the second through last columns of $W_k$ by $\alpha_k^{-1}$, where $\alpha_k \defeq \max_{t=1,\ldots,p} \|\by_t-\bx_k\|$.

Once we have built the vector model $\bem_k$ \eqref{eq_linear_models}, we construct a quadratic model $m_k$ for the full objective $f(\bx)$ in the obvious way, by defining
\be f(\bx_k+\bs) \approx m_k(\bs) \defeq \|\bem_k(\bs)\|^2 = \|\br_k\|^2 + \bg_k^{\top}\bs + \frac{1}{2}\bs^{\top}H_k\bs, \label{eq_gn_full_model_dfo} \ee
where $\bg_k\defeq 2 J_k^{\top}\br_k$ and $H_k\defeq 2 J_k^{\top}J_k$.

This approach is similar to the DFO-GN algorithm, but our slightly different formulation of the interpolation problem \eqref{eq_interp_conditions} is designed to allow improved robustness for noisy problems, and reduce the initialization cost of the algorithm.

\begin{remark}
	An alternative interpolation framework which we considered, inspired by a comment in \cite[Chapter 4]{Conn2009}, designed to balance accuracy of interpolation against large changes in the model between iterations, was to replace \eqref{eq_interp_conditions} with
	\be \min_{\br_k,J_k} \: \|J_k-J_{k-1}\|_F^2 + \lambda_k \sum_{t=0}^{p} \|\bem_k(\by_t-\bx_k) - \br(\by_t)\|^2, \ee
	where $\lambda_k>0$ is an algorithm parameter.
	This idea of allowing inexact interpolation was motivated by the case of noisy objective evaluation.
	However, our extensive testing showed that the best results for this framework, even for noisy objectives, required setting $\lambda_k$ very large (at least $10^{10}$), which means that we are essentially solving \eqref{eq_interp_conditions}.
\end{remark}

\paragraph{Reduced Initialization Cost for Expensive Objectives}
The interpolation problem \eqref{eq_interp_conditions} requires $p\geq n$, so that the system \eqref{eq_linear_interp_system} is square or overdetermined.
This means that before the first model can be constructed, we must evaluate the objective at $p+1$ points --- this is common in model-based DFO algorithms.
Although these evaluations may be parallelized, a user may not have the ability to do this, and the cost of these evaluations may be prohibitive.
In such settings, DFO-LS can proceed with a reduced initialization cost, constructing the model \eqref{eq_linear_models} using as few as 2 objective evaluations.

Suppose we have evaluated the objective at $p+1$ affinely-independent points $\{\by_0,\ldots,\by_p\}$ with $\by_0\defeq\bx_k$, where we now assume $1\leq p<n$.
We construct $\bem_k$ by solving the same interpolation system \eqref{eq_linear_interp_system}, which is now underdetermined, and for which we select the minimal (Euclidean) norm solution.
The resulting $\br_k$ and $J_k$ are solutions to\footnote{\:Note that because in this phase of the algorithm we never remove points from $Y_k$, provided $\{\by_t-\bx_k : t=1,\ldots,p\}$ is linearly independent, \eqref{eq_growing_min_norm} is equivalent to minimizing the change in the model, $\min_{\br_k,J_k}\|\br_k-\br_{k-1}\|^2 + \alpha_k \|J_k-J_{k-1}\|_F^2$.}
\be \min_{\br_k, J_k} \|\br_k\|^2 + \alpha_k \|J_k\|_F^2 \qquad \text{s.t.} \qquad \bem_k(\by_t-\bx_k) = \br(\by_t), \quad \forall t=0,\ldots,p, \label{eq_growing_min_norm} \ee
where $\alpha_k$, defined above, is the column scaling used to precondition \eqref{eq_linear_interp_system}.

However, the construction \eqref{eq_growing_min_norm} is not ideal, because, as proven in \lemref{lem_rank_deficient_model} below, the resulting $J_k$ is not full rank, so the models $\bem_k$ and $m_k$ are not full-dimensional; that is, there are directions along which these are constant, regardless of the objective.

\begin{lemma} \label{lem_rank_deficient_model}
	Suppose $\bem_k$ \eqref{eq_linear_models} is constructed  using \eqref{eq_growing_min_norm}
	 with $p<n$, and where $\{\by_0,\ldots,\by_p\}$ are affinely independent.
	Then $J_k$ has column rank $p$.
\end{lemma}
\begin{proof}
The solution of \eqref{eq_growing_min_norm} is the minimal norm solution for system \eqref{eq_linear_interp_system}.
	Using $\by_0=\bx_k$, we write $W_k$ in \eqref{eq_linear_interp_system} as
	\be W_k = \begin{bmatrix}1 & \b{0}^{\top} \\ \bee & L_k\end{bmatrix}, \qquad \text{where} \qquad L_k \defeq \begin{bmatrix}(\by_1-\bx_k) & \cdots & (\by_p-\bx_k)\end{bmatrix}^{\top} \in \R^{p\times n}, \ee
	and $\bee\in\R^{p}$ is the vector of ones.
	Since the interpolation points are affinely independent, $L_k^{\top}$ has full column rank $p$, so we have the QR factorization $L_k^{\top} = \widehat{Q}\widehat{R}$, where $\widehat{Q}\in\R^{n\times p}$ has columns which are an orthonormal basis for $\operatorname{col}(L_k^{\top})$ and $\widehat{R}\in\R^{p\times p}$ is invertible and upper triangular.
	Then the minimal-norm solution to \eqref{eq_linear_interp_system} is given by
	\be \begin{bmatrix}r_{k,i} \\ \b{j}_{k,i}\end{bmatrix} = \begin{bmatrix}1 & \b{0}^{\top} \\ \b{0} & \widehat{Q}\end{bmatrix}\begin{bmatrix}1 & \b{0}^{\top} \\ \bee & \widehat{R}^{\top}\end{bmatrix}^{-1}\begin{bmatrix}r_i(\by_0) \\ \vdots \\ r_i(\by_p)\end{bmatrix}. \ee
	That is, $\b{j}_{k,i}\in\operatorname{col}(\widehat{Q})$, so $J_k$ can be written as $J_k=\widehat{J}_k \widehat{Q}^{\top}$ for some $\widehat{J}_k\in\R^{m\times p}$, and thus $J_k$ has column rank $p$.
\end{proof}

Thus by using this model, we will not in general be able to find a solution.
A simple way to address this issue would be to, at each iteration, replace one point with the new iterate $\bx_{k+1}$ using standard methods, then add another point to the interpolation set, chosen to increase the dimension of the model (until a full-dimensional model is achieved).
However, this would require two objective evaluations per iteration, which may be wasteful when evaluations are expensive.

{\bf Making the Jacobian full rank in the expensive regime}
Instead, after calculating the rank-deficient $J_k$, DFO-LS makes it full rank --- and hence makes the model full-dimensional --- by increasing its $n-p$ smallest singular values $\sigma_{p+1}=\cdots=\sigma_n=0$ to the level of the smallest nonzero singular value $\sigma_p>0$; this requires the calculation of the SVD of $J_k$.

{\bf Alternative mechanism for expanding the search space}
DFO-LS has another optional mechanism for increasing the dimension of the model, instead of perturbing the singular values of $J_k$.
In this approach, after finding the trust region step $\bs_k$, we replace the new candidate point $\bx_k+\bs_k$ with the perturbed point $\bx_k+\bs_k+\b{d}_k$, where $\b{d}_k$ is a random direction orthogonal to our current set of search directions (with length a constant multiple of $\Delta_k$).
We compare these two approaches in \secref{sec_growing_testing}, and conclude that the SVD-based variant  has similar performance to the random direction extension for small budgets, but better performance for longer budgets.
Thus the SVD approach is chosen as the default in DFO-LS when an initialization with less than $n$ interpolation points is used.

\subsection{Core Algorithmic Framework} \label{sec_main_algo}
\paragraph{Trust Region Framework}
The general algorithmic framework of DFO-LS is that of trust region methods \cite{Conn2000}.
In these methods we maintain a radius parameter $\Delta_k>0$, and say that we expect $m_k$ \eqref{eq_gn_full_model_dfo} to be a good approximation for $f$ in $B(\bx_k,\Delta_k)$, the so-called `trust region'.

At each step in the algorithm, we construct $m_k$ and calculate a step by solving the trust region subproblem
\be \bs_k \approx \argmin_{\|\bs\|\leq\Delta_k} m_k(\bs). \label{eq_tr_subproblem} \ee
Efficient algorithms exist for solving \eqref{eq_tr_subproblem} approximately (e.g.~\cite{Conn2000,Powell2009}).
Having calculated a step $\bs_k$, we evaluate $f(\bx_k+\bs_k)$.
If this step produces a sufficient decrease in the objective, in the sense that
\be r_k \defeq \frac{f(\bx_k) - f(\bx_k+\bs_k)}{m_k(\b{0}) - m_k(\bs_k)}, \label{eq_tr_ratio} \ee
is sufficiently large, then we accept the step (i.e.~set $\bx_{k+1}=\bx_k+\bs_k$) and increase $\Delta_k$.
If the step does not produce sufficient decrease, then we reject the step (i.e.~$\bx_{k+1}=\bx_k$) and decrease $\Delta_k$.
In our derivative-free setting, we then need to update $Y_k$ to ensure it includes the (possibly new) point $\bx_{k+1}$, and where necessary, move points in $Y_k$ to improve its geometry.

\paragraph{Geometric Considerations}
When developing model-based DFO methods, it is well-known (e.g.~\cite{Scheinberg2010}) that one needs to take steps to keep the geometry of $Y_k$ `good', and prevent degeneracy.

For the linear regression models in DFO-LS, the notion of `good' was defined by Conn, Scheinberg and Vicente \cite{Conn2008}.
First, we define the regression Lagrange polynomials of $Y_k$, as the linear functions $\{\Lambda_0(\by), \ldots, \Lambda_p(\by)\}$ given by
\be \Lambda_t(\by) \defeq c_t + \bg_t^{\top}(\by-\bx_k), \quad \text{where $c_t$ and $\bg_t$ solve} \quad \min_{c_t,\bg_t} \: \sum_{s=0}^{p} \left(\Lambda_t(\by_s) - \delta_{s,t}\right)^2. \ee
These polynomials exist and are unique whenever $W_k$ \eqref{eq_linear_interp_system} has full column rank \cite{Conn2008}.
Given these Lagrange polynomials, the measure of the quality of $Y_k$ is given by the following definition.

\begin{definition}[$\Lambda$-poised, regression sense] \label{def_poised}
	For $B\subset\R^n$ and $\Lambda>0$, the set $Y_k$ with $|Y_k|=p+1$ is $\Lambda$-poised in $B$ in the regression sense if $p\geq n$ and
	\be \max_{t=0,\ldots,p}\: \max_{\by\in B} |\Lambda_t(\by)| \leq \Lambda, \ee
	for all $\by\in B$, where $\{\Lambda_0(\by), \ldots, \Lambda_p(\by)\}$ are the regression Lagrange polynomials for $Y_k$.
\end{definition}

A similar definition holds for exact (i.e.~non-regression) interpolation when $p=n$; see \cite{Cartis2017a} for details.
As in that case, a small value of $\Lambda$ indicates that the geometry of $Y_k$ is `good'.
The steps we use in DFO-LS to improve the $\Lambda$-poisedness of $Y_k$ are outlined in \secref{sec_implementation_general}, and the details of how $\Lambda$-poisedness leads to good regression models are given in \appref{sec_convergence}.

No geometry-improving steps are allowed in the early iterations of the expensive regime, while $p<n$.

\subsection{DFO-LS Algorithm}
\begin{algorithm}
	\small{
	\begin{algorithmic}[1]
		\Require Starting point $\bx_0\in\R^n$, initial trust region radius $\Delta_0>0$ and integers $p_{init}$ and $p$, the sizes of the initial and final interpolation sets, respectively,
where $1\leq p_{init}\leq p$ and   $p\geq n$. 
		\vspace{0.2em}
		\Statex Parameters: maximum trust region radius $\Delta_{max}\geq\Delta_0$, minimum trust region radius $0<\rho_{end}<\Delta_0$, trust region radius scalings $0<\gamma_{dec}<1<\gamma_{inc}\leq\overline{\gamma}_{inc}$ and $0<\alpha_1<\alpha_2<1$, acceptance thresholds $0 < \eta_1 \leq \eta_2 < 1$, safety reduction factor $0 < \omega_S < 1$, safety step threshold $0 < \gamma_S < 1$, and Boolean flag \texttt{NOISY} for the presence of noise in the objective.
		\vspace{0.5em}
		\State Build an initial interpolation set $Y_0\subset B(\bx_0,\Delta_0)$ of size $p_{init}+1$, with $\bx_0\in Y_0$. Set $\rho_0=\Delta_0$.
		\For{$k=0,1,2,\ldots$}
			\If{\texttt{NOISY} \textbf{and} \textit{all values $\{f(\by) : \by \in Y_k\}$ are within noise level of $f(\bx_k)$}}
				\State Call restart (set $\Delta_{k+1}=\rho_{k+1}=\Delta_0$ and build $Y_{k+1}$ as per \secref{sec_restarts_description}) and \textbf{goto} line \ref{ln_loop}.
			\EndIf
			\State Given $\bx_k$ and $Y_k$, construct the model $\bem_k(\bs)$ \eqref{eq_linear_models}
by solving the interpolation problem \eqref{eq_growing_min_norm} if $ |Y_k|<p+1$, otherwise \eqref{eq_linear_interp_system}.\label{ln_loop}
\State Form the full model $m_k$ \eqref{eq_gn_full_model_dfo}, and  
approximately solve the trust region subproblem \eqref{eq_tr_subproblem} to get a step $\bs_k$.\label{ln_trs}
			\If{$\|\bs_k\| < \gamma_S \rho_k$}
				\If{$|Y_k|<p+1$}
					\State \underline{Safety Phase (Growing)}: Form $Y_{k+1}=Y_k\cup\{\bx_k+\bs\}$ for some $\bs$ orthogonal to $\{\by-\bx_k : \by\in Y_k\}$ with $\|\bs\|=\Delta_k$.
					\State Set $(\rho_{k+1}, \Delta_{k+1}) = (\rho_k, \Delta_k)$.
				\Else
					\State \underline{Safety Phase}: Set $\bx_{k+1}=\bx_k$ and $\Delta_{k+1} = \max(\rho_k, \omega_S \Delta_k)$, and form $Y_{k+1}$ by improving the geometry of $Y_k$.
					\State If $\Delta_{k+1} = \rho_k$, set $(\rho_{k+1}, \Delta_{k+1}) = (\alpha_1\rho_k, \alpha_2\rho_k)$, otherwise set $\rho_{k+1}=\rho_k$.
					\State If $\rho_{k+1}\leq \rho_{end}$: call restart if \texttt{NOISY}, else \textbf{terminate}.
				\EndIf
				\State \textbf{goto} line \ref{ln_loop}.
			\EndIf
			\State Evaluate $\br(\bx_k+\bs_k)$ and calculate ratio $r_k$ \eqref{eq_tr_ratio}.
			\State Accept/reject step and update trust region radius: set
			\be \bx_{k+1} = \begin{cases}\bx_k + \bs_k, & r_k \geq \eta_1, \\ \bx_k, & r_k < \eta_1, \end{cases} \quad \text{and} \quad \Delta_{k+1} = \begin{cases}\min(\max(\gamma_{inc}\Delta_k, \overline{\gamma}_{inc}\|\bs_k\|), \Delta_{max}), & r_k \geq \eta_2, \\ \max(\gamma_{dec}\Delta_k, \|\bs_k\|, \rho_k), & \eta_1 \leq r_k < \eta_2, \\ \max(\min(\gamma_{dec}\Delta_k, \|\bs_k\|), \rho_k), & r_k < \eta_1. \end{cases} \ee
			\If{$|Y_k| < p+1$}
				\State \underline{Growing Phase}: Form $Y_{k+1}=Y_k\cup\{\bx_k+\bs_k\}$ and set $\rho_{k+1} = \rho_k$.
			\ElsIf{$r_k \geq \eta_1$}
				\State \underline{Successful Phase}: Form $Y_{k+1}=Y_k\cup\{\bx_{k+1}\}\setminus\{\by\}$ for some $\by\in Y_k$ and set $\rho_{k+1}=\rho_k$.
				\State If objective decrease is too slow: call  restart if \texttt{NOISY}, else \textbf{terminate}.
			\ElsIf{\texttt{NOISY} \textbf{and} \textit{restart auto-detected}}
				\State \underline{Restart Auto-Detection}: Call restart.
			\ElsIf{\textit{geometry of $Y_k$ is not good}}
				\State \underline{Model Improvement Phase}: Improve the geometry of $Y_{k+1}$ and set $\rho_{k+1}=\rho_k$.
			\Else
				\State \underline{Unsuccessful Phase}: Set $Y_{k+1}=Y_k$, and if $\Delta_{k+1} = \rho_k$, set $(\rho_{k+1}, \Delta_{k+1}) = (\alpha_1\rho_k, \alpha_2\rho_k)$, otherwise set $\rho_{k+1}=\rho_k$. \label{ln_rho_redn}
				\State If $\rho_{k+1}\leq \rho_{end}$: call restart if \texttt{NOISY}, else \textbf{terminate}.
			\EndIf \label{ln_loop_end}
		\EndFor
	\end{algorithmic}
	} 
	\caption{DFO-LS: Derivative-Free Optimization for Least-Squares.}
	\label{alg_dfols}
\end{algorithm}

A full statement of the DFO-LS algorithm is given in \algref{alg_dfols}. 
The overall structure of DFO-LS builds upon that of DFO-GN \cite{Cartis2017a}, with a key difference being the ability to use regression models, which allows us the flexibility to have a reduced initialization cost when evaluations are expensive, and to implement regression models when the problem is noisy.
For the latter regime, the most efficient contribution of DFO-LS is the multiple restarts feature described in \secref{sec_restarts_description}. 
Other key features of DFO-LS are described in \secref{sec_implementation_general}, and optional features for noisy problems (such as regression and sampling) are described in \secref{sec_other_noisy_features}.

We note that DFO-LS uses a standard trust region framework, but maintains two different measures of trust region radius: the usual $\Delta_k$ as described in \secref{sec_main_algo}, and a lower bound $\rho_k\leq\Delta_k$.
Originally a feature from Powell \cite{Powell2003}, this is used to ensure that we do not decrease $\Delta_k$ too much until we are confident that the geometry of $Y_k$ is sufficiently good; i.e.~that unsuccessful steps (where $r_k<\eta_1$) are not because of a poor quality model, but because the nature of the objective near $\bx_k$ requires a small trust region in order to make good progress.

A summary of the convergence guarantees of DFO-LS is given in \appref{sec_convergence}.

\section{New Algorithmic Features} \label{sec_implementation}
In this section, we describe  briefly the general features of DFO-LS 
 and, in greater detail, several new features for handling noisy objectives.

\subsection{General Features} \label{sec_implementation_general}
We summarize how some of the steps in \algref{alg_dfols} are performed in practice. 
The majority of these general features are inherited from DFO-GN \cite{Cartis2017a}, but DFO-LS includes variable scaling, two new termination criteria, different default trust region parameters for noisy problems, and a slightly different approach for determining the initial set $Y_0$.
A full outline of these general features is given in \appref{sec_general_features_appendix}.
The most important are:

\begin{description}
	\item[\normalfont\textit{Geometry-Improving Steps:}] Similar to BOBYQA \cite{Powell2009}, geometry-improving steps in DFO-LS do not guarantee the $\Lambda$-poisedness of $Y_k$ (as required for convergence theory; see \appref{sec_convergence}), but instead perform a simplified procedure: take the point furthest from $\bx_k$, and replace it with
	\be \by^+ = \argmax_{\by\in B(\bx_k,\Delta_k)} |\Lambda_t(\by)|, \label{eq_geom_improvement} \ee
	\item[\normalfont\textit{Model Updating:}] Unlike \algref{alg_dfols}, the new point $\bx_k+\bs_k$ is added to the interpolation set at all iterations, so the newest information about the objective is always used;
	\item[\normalfont\textit{Inclusion of Bound Constraints and Variable Scaling:}] Like DFO-GN, DFO-LS can solve \eqref{eq_ls_definition} with optional bound constraints $\b{a} \leq \bx \leq \b{b}$. To improve problem conditioning, users can optionally allow the variables to be internally shifted and scaled to be between $[0,1]$.
	\item[\normalfont\textit{Termination Criteria:}] In addition to the criteria from DFO-GN, in DFO-LS, termination can occur when the rate of objective decrease is slow (similar to the ``$f_i^{*\prime}$ test'' of Larson and Wild \cite{Larson2013}), or if all interpolation values $f(\by_t)$ are within some user-specified noise level of $f(\bx_k)$; and,
	\item[\normalfont\textit{Default Parameters for Noisy Problems:}] DFO-LS allows user to specify, via an input flag, if their objective is noisy. If so, DFO-LS uses different, and more appropriate, default values for $\gamma_{dec}$, $\alpha_1$ and $\alpha_2$ than the noiseless case.
\end{description}

\subsection{Multiple Restarts for Noisy Objectives} \label{sec_restarts_description}
To improve robustness to noisy objectives, DFO-LS uses a multiple restarts mechanism.
As motivation, note that, in  DFO trust-region methods, $\Delta_k$ tends to zero as a measure of convergence; see for example, \cite[Lemma 3.11]{Cartis2017a} and Theorem \ref{thm_lim} for the deterministic case. However, when the function is noisy, as $\Delta_k$ gets small, 
the interpolation points get very close together and the corresponding objective values are all within noise level.
As a result, $\Delta_k$ no longer reflects convergence and the solver can stagnate in a suboptimal region.

\begin{figure}[t]
	\centering
	\begin{subfigure}[b]{0.48\textwidth}
		\includegraphics[width=\textwidth]{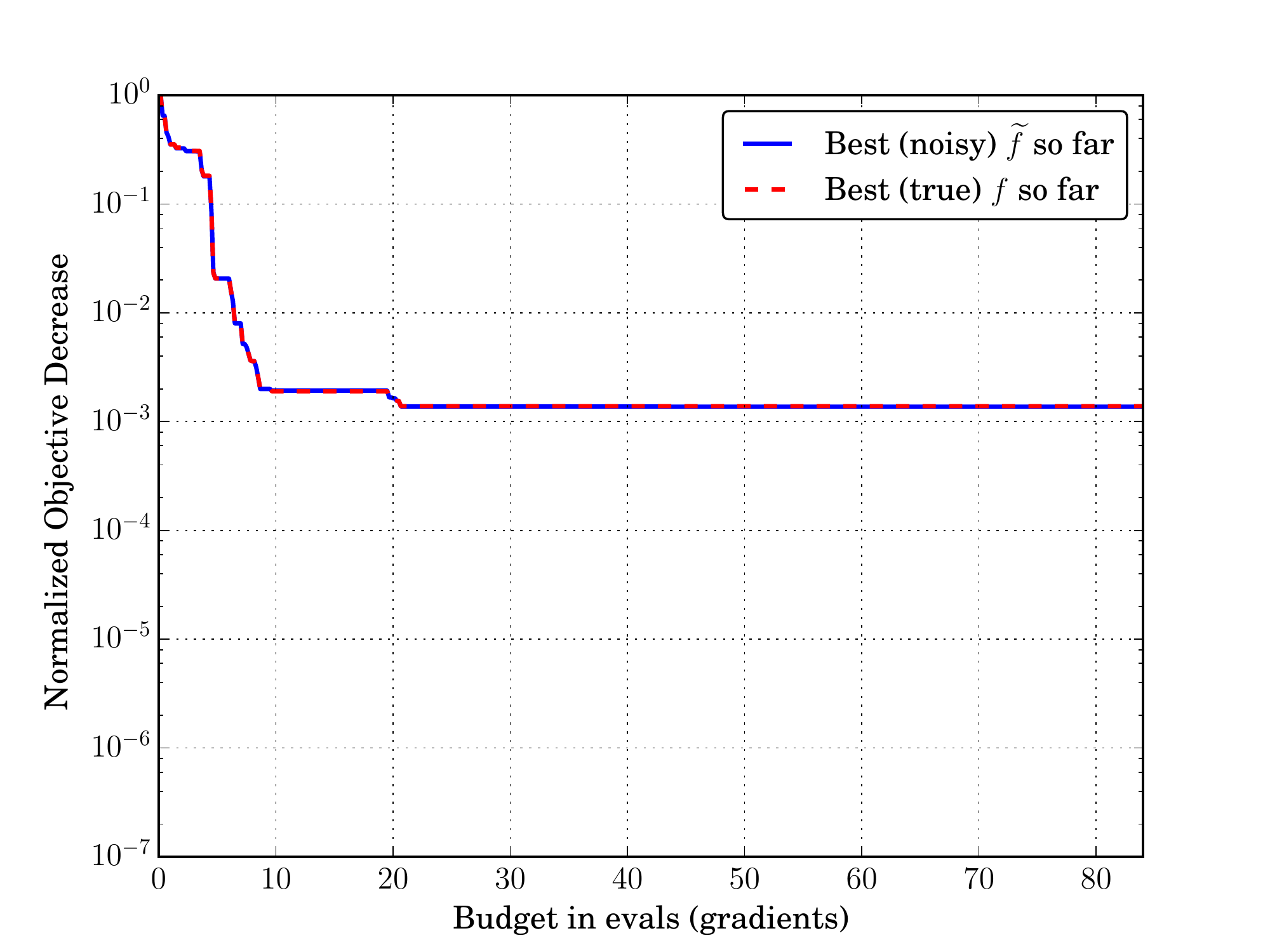}
		\caption{Normalized objective decrease (no restarts)}
	\end{subfigure}
	~
	\begin{subfigure}[b]{0.48\textwidth}
		\includegraphics[width=\textwidth]{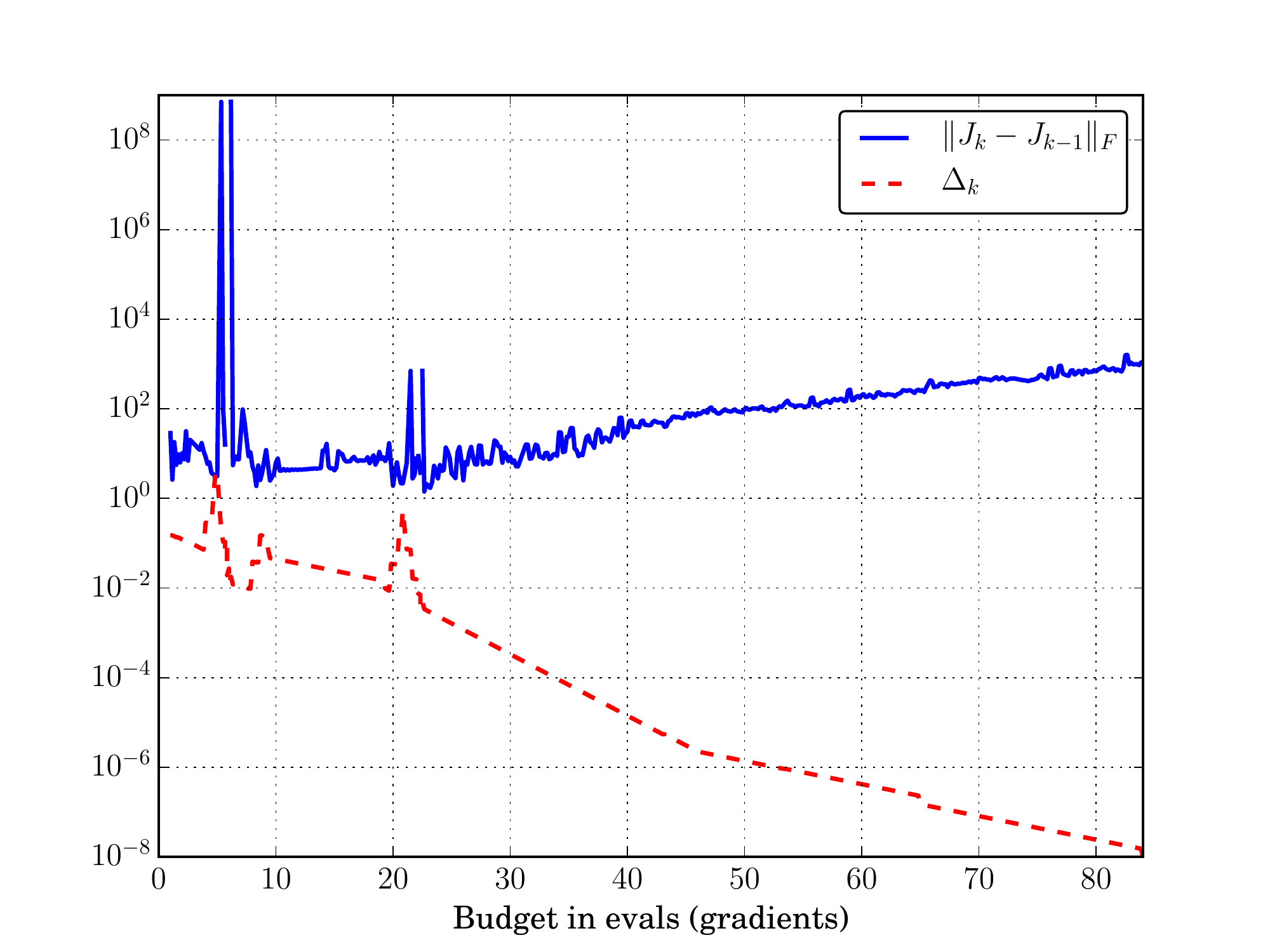}
		\caption{Convergence details (no restarts)}
	\end{subfigure}
	\\
	\begin{subfigure}[b]{0.48\textwidth}
		\includegraphics[width=\textwidth]{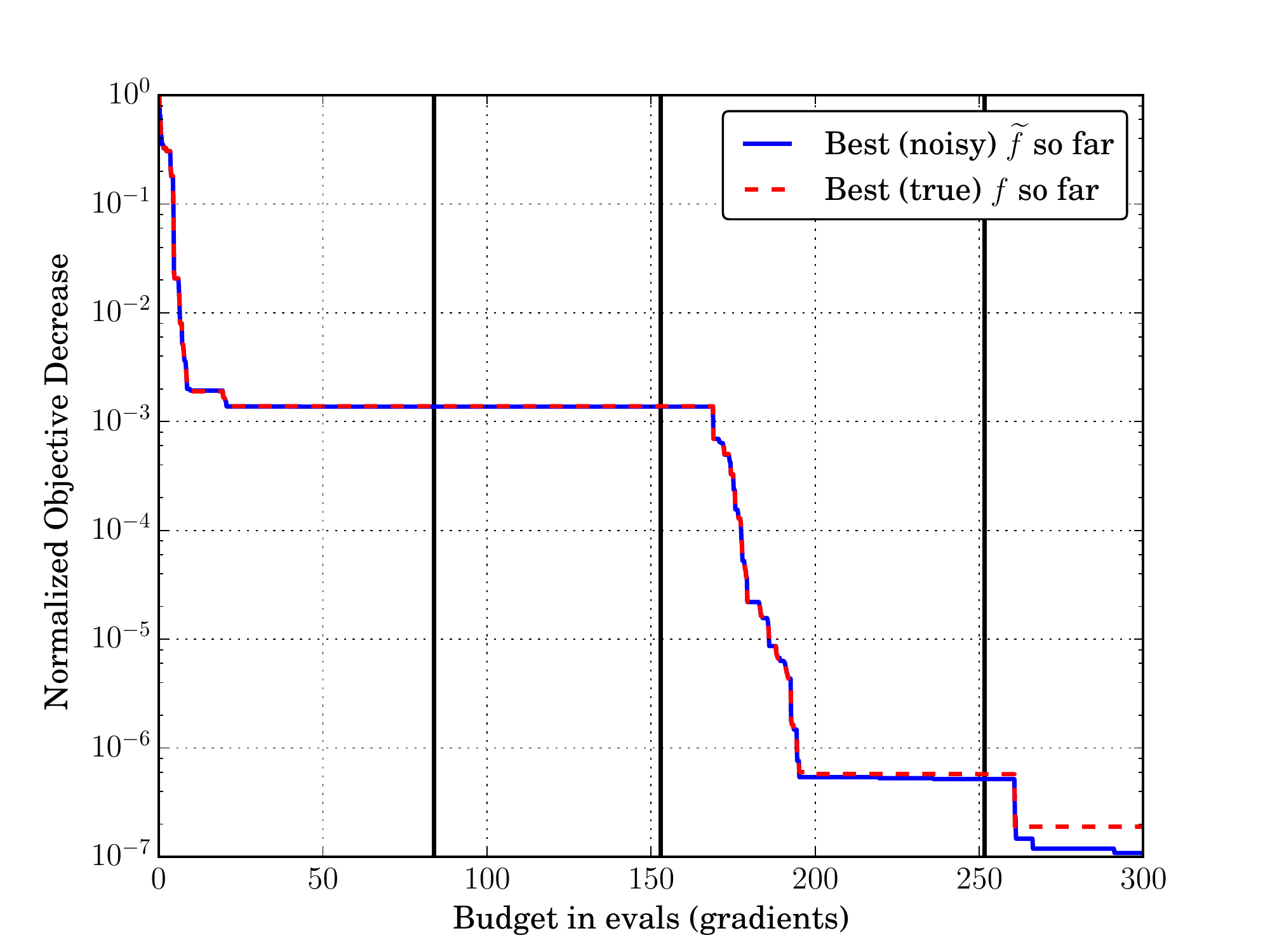}
		\caption{Normalized objective decrease (with restarts)}
		\label{fig_restarts_motivation_with_restarts_obj_redn}
	\end{subfigure}
	~
	\begin{subfigure}[b]{0.48\textwidth}
		\includegraphics[width=\textwidth]{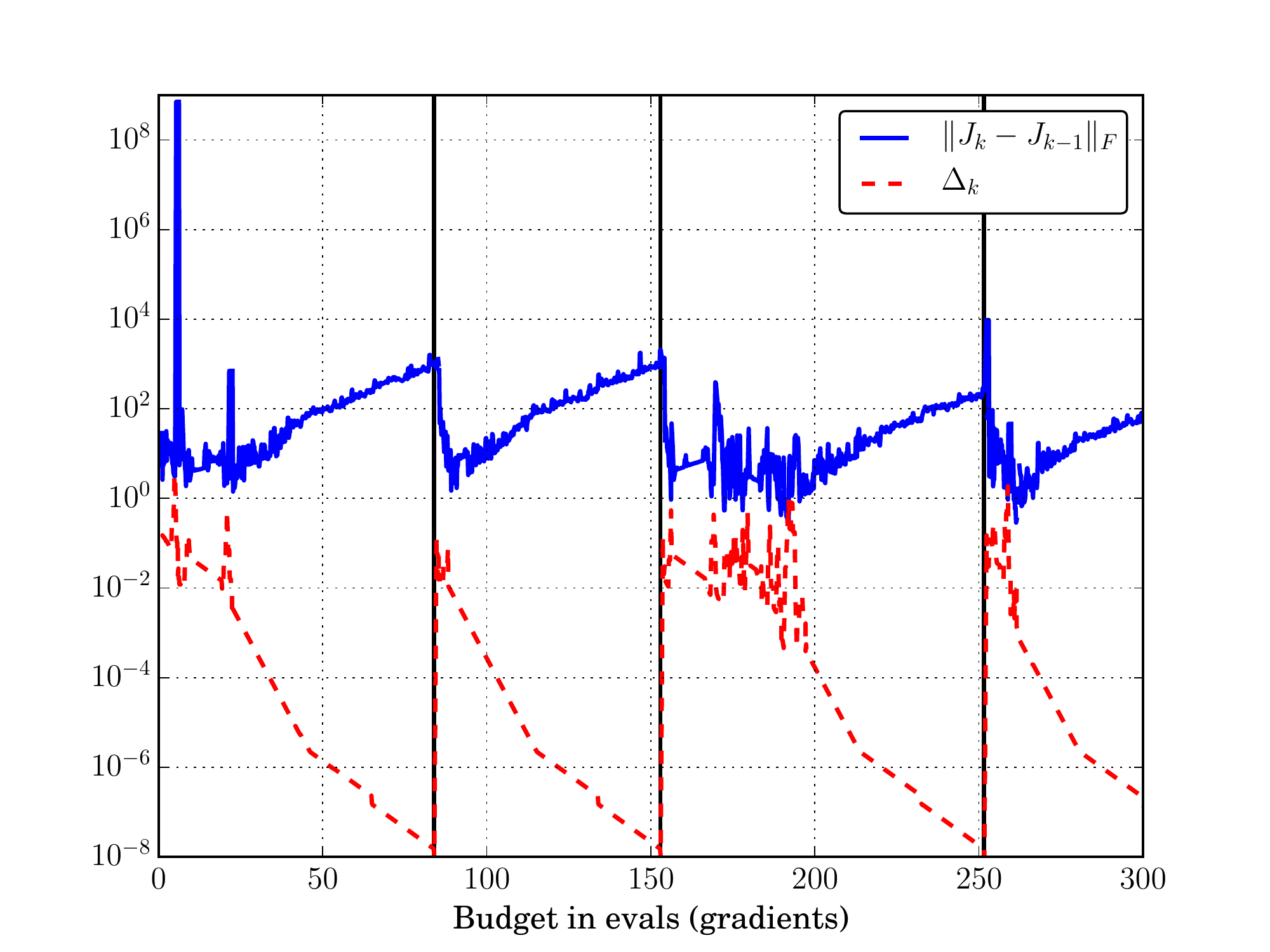}
		\caption{Convergence details (with restarts)}
	\end{subfigure}
	\caption{Normalized objective decrease achieved by DFO-LS, measured in both the noisy and true objective, and convergence information for test problem `Osborne 1' with $(n,m)=(5,33)$ and unbiased multiplicative Gaussian noise of size $\sigma=10^{-2}$. 
	The vertical black lines indicate when restarts occurred. Restart type was `soft (moving $\bx_k$)', and the budget was $300(n+1)$ evaluations;  the (a) and (b) runs terminated early on small trust-region radius.
	}
	\label{fig_restarts_motivation}
\end{figure}

{\it An illustrative example.} We may see this effect by considering a test problem.
In \figref{fig_restarts_motivation}, we compare two runs of DFO-LS --- with and without multiple restarts --- for the Osborne 1 test problem \cite[Problem 36]{More2009}, where we have added unbiased multiplicative Gaussian noise with $\sigma=10^{-2}$.
After making some initial progress, the run without restarts has many unsuccessful steps, and $\Delta_k$ shrinks as the solver attempts to find a descent direction.
When this happens, the interpolated Jacobian $J_k$ begins to change substantially at each iteration.
This indicates that the noise in the interpolation problem is dominating the true descent information.

When we introduce restarts, the stagnation can eventually be overcome. 
When a restart occurs (and we increase $\Delta_k$ to its original level), the changes in $J_k$ reduce quickly, and so the interpolation is more likely to capture genuine information about changes in the objective.
As a result, the solver is able to progress, and ultimately finds a much higher accuracy solution. $\Box$

In DFO-LS, a restart may be triggered by all the termination criteria, except for small objective and maximum computational budget.
At its simplest, a restart involves increasing $\Delta_k$ to a much larger value, and possibly moving some of the points in $Y_k$.
There are two main types of restart which DFO-LS can perform:
\begin{description}
	\item[\normalfont\textit{Hard restart:}] Reset the trust region radius to $\Delta_k=\rho_k=\Delta_0$, and rebuild $Y_k$ in the new (larger) trust region $B(\bx_k,\Delta_k)$ from scratch using the same mechanism as how $Y_0$ was originally constructed in the case $p=n$ (see \secref{sec_implementation_general}); and,
	\item[\normalfont\textit{Soft restart (moving $\bx_k$):}] Reset the trust region radius to $\Delta_k=\rho_k=\Delta_0$, and save the current best point $\bx_k$ separately. Then, move $\bx_k$ to a geometry-improving point in the new trust region $B(\bx_k,\Delta_k)$, shifting the trust region to this new point. Finally, move the $N-1<p$ points in $Y_k$ which were closest to the old value of $\bx_k$ to geometry-improving points in the new (larger \& shifted) trust region $B(\bx_k,\Delta_k)$ as per \eqref{eq_geom_improvement}. The iteration then continues from whichever of these $N$ new points has the least objective value, which may be worse than the value from the end of the previous iteration. The final solution returned by the solver takes the optimal value seen so far, including the saved endpoints from previous restarts.
\end{description}
The soft restart approach with $N=\min(3,p)$ is the default approach in DFO-LS.

We see that soft restarts require the objective to be evaluated $N<p$ times, whereas hard restarts require a full $p$ objective evaluations (as we have changed all interpolation points except $\bx_k$).
We also note that the soft restart mechanism is intrinsically linked to the model-based DFO framework, and there is not a clear derivative-based equivalent of this procedure.

In DFO-LS, when restarts are used, we add an extra termination criterion: we terminate if the last $M$ consecutive restarts have not achieved any objective reduction (default $M=10$).

\paragraph{Auto-detection of restarts} 
As discussed above, we saw in \figref{fig_restarts_motivation} that the need for a restart can be determined by a series of unsuccessful iterations, coupled with large changes in $J_k$, with this change increasing rapidly with $k$.
By contrast, selecting an a priori value of $\rho_{end}$ which provides a timely trigger for restarts is not straightforward.
Thus, DFO-LS uses previous iteration information to auto-detect when a restart is needed.
A restart is triggered if, in the last $N$ iterations:
\begin{itemize}
	\item The trust region radius $\Delta_k$ has never been increased, and it has been decreased on at least twice as many iterations as it has been kept constant; and,
	\item The slope and correlation coefficient of a linear fit through the points $\{(k, \log\|J_k-J_{k-1}\|_F)\}$ exceed given thresholds; that is, $\|J_k-J_{k-1}\|_F$ is consistently increasing at a given exponential rate\footnote{\:This condition is similar to the noise detection mechanism from \cite{Augustin2014}, which we became aware of after the condition had been chosen.}.
\end{itemize}

{\it An illustrative example, revisited.}
In \figref{fig_restarts_autodetect_motivation}, we see the same results as in \figref{fig_restarts_motivation}, but with auto-detection of when to restart.
Because of the auto-detection, restarts are triggered much earlier (avoiding the iterations during which no progress was made), and we achieve accuracy $\tau=10^{-6}$ after approximately $15(n+1)$ evaluations, rather than approximately $80(n+1)$ evaluations without auto-detection. $\Box$

\begin{figure}
	\centering
	\begin{subfigure}[b]{0.48\textwidth}
		\includegraphics[width=\textwidth]{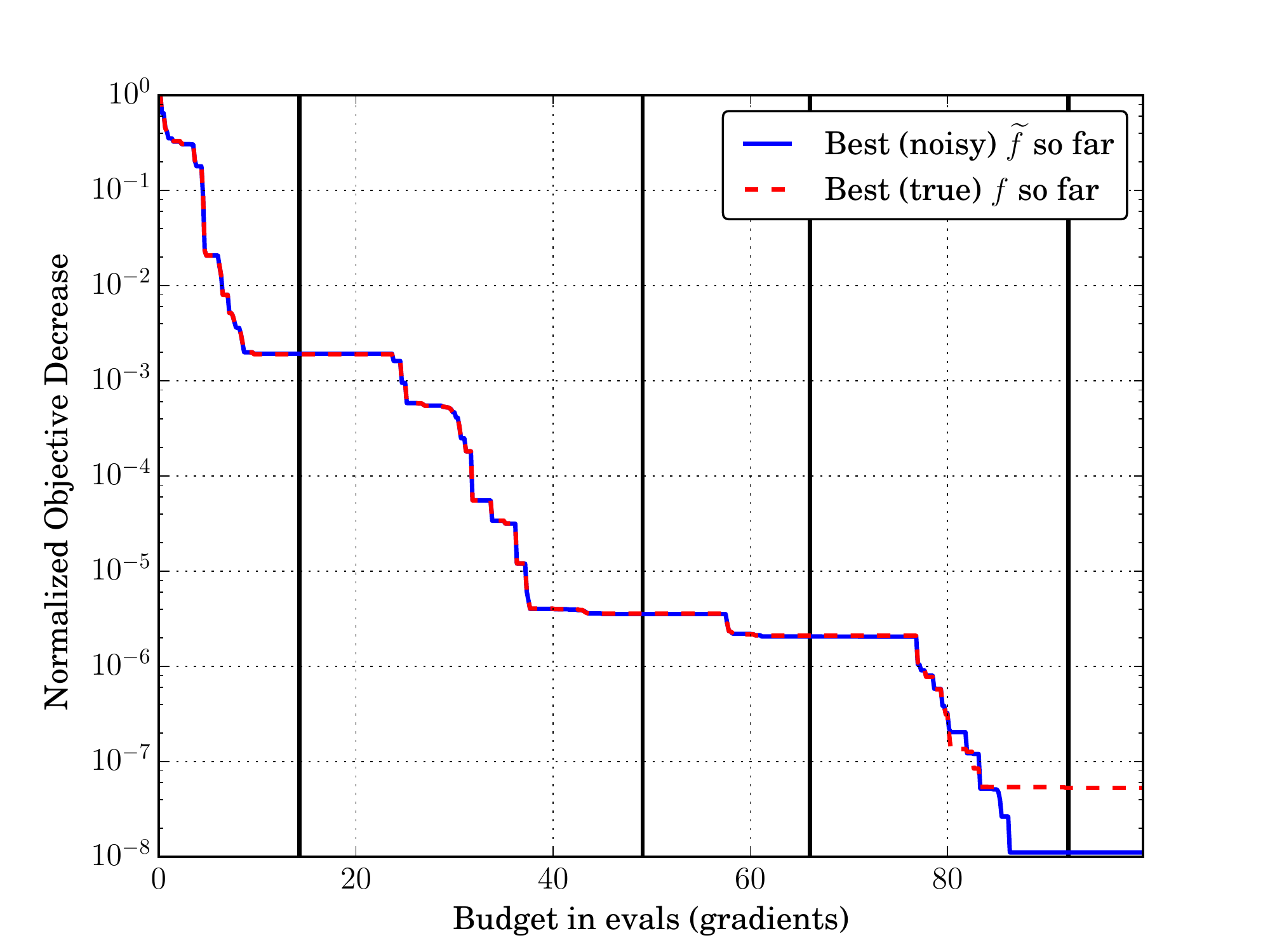}
		\caption{Normalized objective decrease}
		\label{fig_restarts_autodetect_motivation_obj_redn}
	\end{subfigure}
	~
	\begin{subfigure}[b]{0.48\textwidth}
		\includegraphics[width=\textwidth]{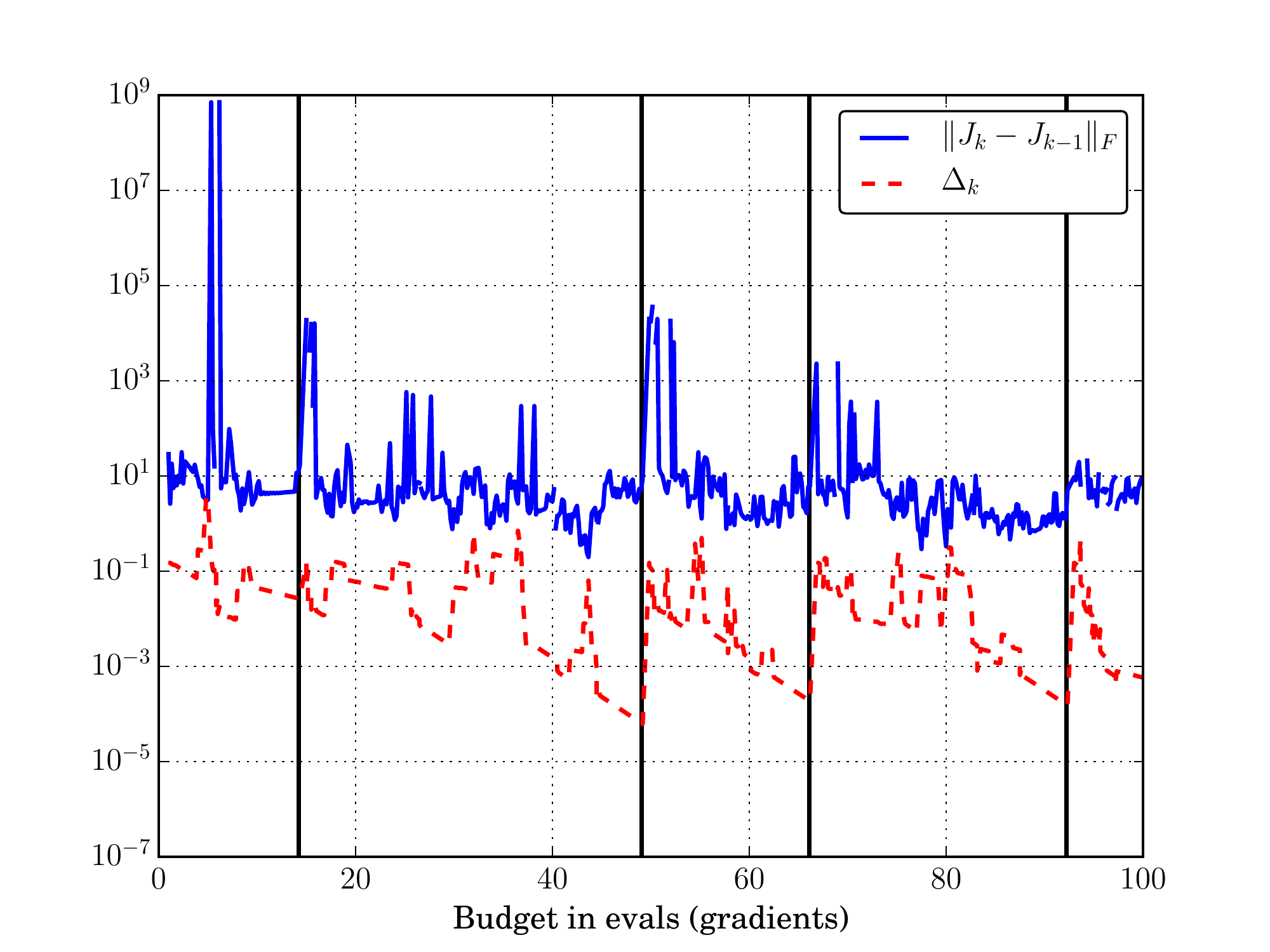}
		\caption{Convergence details}
	\end{subfigure}
	\caption{Normalized objective decrease achieved by DFO-LS, measured in both the noisy and true objective, and convergence information as per \figref{fig_restarts_motivation}, but allowing auto-detection of restarts. The budget was $100(n+1)$ objective evaluations. The difference between the `noisy' and `true' objective reduction measures is discussed in \secref{sec_testing_methodology}.}
	\label{fig_restarts_autodetect_motivation}
\end{figure}

\paragraph{Alternative Restart Mechanism}
DFO-LS has another approach available for performing soft restarts.
\begin{description}
	\item[\normalfont\textit{Soft restart (fixed $\bx_k$):}] Reset the trust region radius to $\Delta_k=\rho_k=\Delta_0$, and move the $N<p$ points $\by_t\neq\bx_k$ in $Y_k$ which are closest to $\bx_k$ to geometry-improving points in the new trust region $B(\bx_k,\Delta_k)$ as per \eqref{eq_geom_improvement}.  
	The iteration then continues from whichever of these $N$ new points has the least objective value.
\end{description}
As we will see in \secref{sec_new_features_results}, the soft restart (fixed $\bx_k$), with the default value $N=\min(3,p)$, performs noticeably worse than the `moving $\bx_k$' version of soft restarts.

\subsection{Optional Features for Noisy Objectives} \label{sec_other_noisy_features}
Aside from multiple restarts, DFO-LS also implements the two most common approaches for handling noisy objectives in model-based DFO: sample averaging and regression models.
In \secref{sec_new_features_results}, we show numerically that using multiple restarts gives better performance than averaging and regression. Therefore, the latter features are not used in DFO-LS by default.

\subsubsection{Sample Averaging} \label{sec_averaging_description}
Sample averaging replaces evaluations of the noisy objective $f(\bx)$  with an average of $N$ samples. In the least-squares
case, we replace an evaluation of the noisy objective $\overline{\br}$ with the sample average
\be \overline{\br}_N(\bx) \defeq \frac{1}{N}\sum_{i=1}^{N}\br(\bx, \xi_i), \ee
where $\xi_1, \ldots, \xi_N$ are different realizations of the random variable $\xi$ defining the noise.

For theoretical convergence guarantees to hold, one must choose $N=\bigO(\Delta_k^{-4})$ (e.g.~\cite{Chen2016}).
However, as $\Delta_k$ can easily be of size $10^{-2}$ or $10^{-3}$, this amount of averaging rapidly becomes impractical, so $N=\bigO(\Delta_k^{-1})$, for instance, is a more sensible choice in practice (e.g.~Algorithm TR-SAA in \cite{Chen2016}).

In the implementation of DFO-LS, the use of sample averaging is governed by a user-specified function which allows for a wide range of sample averaging techniques:
\be N = \mathrm{nsamples}(\rho_k, \Delta_k, k, n_{\mathrm{restarts}}), \label{eq_nsamples} \ee
where $n_{\mathrm{restarts}}\in\{0,1,2,\ldots\}$ is the number of restarts that the solver has performed (see \secref{sec_restarts_description} for details) and $k\in\{0,1,2,\ldots\}$ is the iteration number since the most recent restart.
The default option for $\mathrm{nsamples}$ gives $N\equiv 1$ always; i.e.~no sample averaging.

\subsubsection{Regression Models for Noisy Objectives}

Building regression models rather than interpolation models requires having more interpolation points than degrees of freedom in the model \cite{Conn2008,Billups2013,Chen2016}. 
Our formulation of the DFO-LS model construction problem \eqref{eq_interp_conditions} allows regression models, 
when $p>n$. It remains to consider 
how to evolve the set $Y_k$ at each iteration.


DFO-LS has three mechanisms for moving multiple points on successful iterations, which can be used alongside regression models:
\begin{description}
	\item[\normalfont\textit{Nothing:}] Replace one point in $Y_k$ with $\bx_{k+1}$, and nothing else;
	\item[\normalfont\textit{Geometry-based:}] Replace one point in $Y_k$ with $\bx_{k+1}$, and then move the $N$ points in $Y_k$ which are furthest from $Y_{k+1}$ to geometry-improving locations in $B(\bx_{k+1},\Delta_{k+1})$, given by \eqref{eq_geom_improvement}; and
	\item[\normalfont\textit{Momentum-based:}] Replace one point in $Y_k$ with $\bx_{k+1}$, and then move the $N$ points in $Y_k$ which are furthest from $Y_{k+1}$ to $\bx_{k+1}+\Delta_{k+1}\b{d}$, where $\b{d}$ is a random unit vector with\footnote{\:When using bound constraints, if $\Delta_{k+1}\b{d}$ gives a point outside the bounds, we instead take $\alpha\b{d}$ for some $\alpha\in(0,\Delta_{k+1})$ such that the bounds are satisfied. If this requires $\alpha<10^{-3}$, we replace $\b{d}$ with $-\b{d}$, sacrificing the requirement that $\b{d}^{\top}\bs_k>0$.} $\b{d}^{\top}\bs_k>0$.
\end{description}
The last two mechanisms above, that replace multiple points,
 try to mimic/match the situation arising in sample averaging, where moving one point affects $c$ function values, where $c$ is the sampling rate.

The first mechanism (`nothing') is the default in DFO-LS, when regression models are used; they are compared with each other, and against sample averaging, in \secref{sec_regression_results}.

\section{Testing Framework} \label{sec_testing_framework}

We outline the framework we used for testing and comparing DFO-LS to other solvers, which is similar to  Mor\'e and Wild \cite{More2009} and to \cite{Cartis2017a}, but designed to capture the two regimes of interest - expensive and/or noisy. We also introduce a new approach for measuring solver performance for noisy problems.

\subsection{Testing Methodology} \label{sec_testing_methodology}
When running each solver, we choose the maximum allowed budget in units of simplex gradients (i.e.~in multiples of $n+1$) to provide fair comparisons across problems of different dimensions.
The measure of solver performance  is the number of evaluations required to achieve a specified reduction in the objective. 

Suppose we have an underlying smooth objective $f$, but only see evaluations of the noisy objective $\t{f}\approx f$, where 
\be \t{f}(\bx) \defeq \alpha f(\bx) + \beta + \sigma(\bx)\epsilon, \label{eq_noise_model_assumption} \ee
for constants $\alpha>0$ and $\beta\in\R$, and where $\sigma(\bx)$ is the standard deviation of the noise.
The stochastic noise $\epsilon$ has zero mean and unit variance, and so $|\epsilon|\sim\bigO(1)$.
Under this assumption --- which holds for all the noise models we consider in \secref{sec_test_problems} --- minimizing $\mathbb{E}[\t{f}(\bx)]$ yields a minimizer of the true objective $f$.
Then, for a solver $\mathcal{S}$, problem $p$, and accuracy level $\tau\in(0,1)$, we define the number of evaluations required to solve the (smooth or noisy) problem as:
\begin{align}
	\left\{\begin{array}{c}N_p(\mathcal{S}; \tau) \\ \t{N}_p(\mathcal{S}; \tau)\end{array}\right\} &\defeq \text{\# objective evaluations taken to find a point $\bx$ satisfying} \nonumber \\
	&\qquad \left\{\begin{array}{rcl} f(\bx) &\leq & f(\bx^*) + \tau(f(\bx_0) - f(\bx^*))] \\ \t{f}(\bx) &\leq & \mathbb{E}[\t{f}(\bx^*) + \tau(\t{f}(\bx_0) - \t{f}(\bx^*))] \end{array}\right\}, \label{eq_solved_threshold}
\end{align}
where $\bx^*$ is an estimate of the true minimizer of the smooth objective $f$\footnote{For our two sets of test problems (Mor\'e \& Wild and CUTEst), values of $f(\bx^*)$ are given in \cite{Cartis2017a}.}.
If the required reduction \eqref{eq_solved_threshold} was never achieved in the maximum allowed budget, we define $N_p(\mathcal{S}; \tau)$ or $\t{N}_p(\mathcal{S}; \tau)=\infty$.

Following \cite{More2009}, we compare different solvers using data profiles.
These measure the proportion of test problems for which the required budget $N_p(\mathcal{S}; \tau)$ or $\t{N}_p(\mathcal{S}; \tau)$ is less than a given value (in units of simplex gradients), namely $\alpha(n_p+1)$, where $n_p$ is the dimension of problem $p$.
The data profiles are defined, for the solver $\mathcal{S}$ and set of test problems $\mathcal{P}$, to be the curves
\be \left\{\begin{array}{c} d_{\mathcal{S}}(\alpha) \\ \t{d}_{\mathcal{S}}(\alpha) \end{array}\right\} \defeq \frac{1}{|\mathcal{P}|} \cdot \left|\left\{p\in\mathcal{P} :  \left\{\begin{array}{c} N_p(\mathcal{S}; \tau_p) \\ \t{N}_p(\mathcal{S}; \tau_p)\end{array}\right\} \leq \alpha(n_p+1)\right\}\right|, \qquad \alpha\geq0, \label{eq_data_profile} \ee
where $n_p$ is the dimension of problem $p$. $N_p(\mathcal{S};\tau)$ measures genuine progress towards the minimizer, excluding any objective reductions from sampling errors, but $\t{N}_p(\mathcal{S};\tau)$ measures progress using information actually available to the solver. The two different performance measures \eqref{eq_solved_threshold} have each been used previously (e.g.~$N_p(\mathcal{S};\tau)$  in \cite{Chen2016} and  $\t{N}_p(\mathcal{S};\tau)$, in \cite{Billups2013}), but we are not aware of work where the two measures have been compared and combined. We propose to do so, in the accuracy level we choose.

We use throughout
a problem-specific accuracy level $\tau_p$ rather than a constant value for all problems, which we set based on the accuracy which is reasonable for a solver to attain given the noise level in the problem. By adapting $\tau_p$ to each problem, we can measure progress using $\t{N}_p$ but gain the benefit of $N_p$, thus capturing genuine progress in the objective.

Note that throughout, for noisy problems, we performed 10 runs of each solver and below we always show average data profiles.
We also show average profiles over 10 runs for DFO-LS and Py-BOBYQA in the case of noiseless objectives, because the generation of the initial set $Y_0$ uses random orthogonal directions.

\paragraph{Choice of adaptive accuracy level $\tau_p$}
Given \eqref{eq_solved_threshold}, we would expect the two measures $N_p$ and $\t{N}_p$ to be similar, provided that our desired objective reduction was much larger than the noise level.
If $N_p$ and $\t{N}_p$ were similar, we could conclude that reductions in the observable $\t{f}$ correspond to genuine objective reductions, and not sampling error.
Specifically, suppose a solver has reached a point $\bx_k$, which corresponds to an accuracy of (exactly) $\tau$ based on $N_p$ and of $\t{\tau}$ based on $\t{N}_p$, 
\be f(\bx_k) = f(\bx^*) + \tau(f(\bx_0)-f(\bx^*)) \quad \text{and} \quad \t{f}(\bx_k) = \mathbb{E}[\t{f}(\bx^*) + \t{\tau}(\t{f}(\bx_0)-\t{f}(\bx^*))].\label{temp:ex-noise} \ee
Letting $\epsilon_k$ be the realization of $\epsilon$ for the given $\t{f}(\bx_k)$, we combine the expressions in \eqref{temp:ex-noise} using \eqref{eq_noise_model_assumption}, to get
\be \underbrace{\t{\tau}}_{\text{noisy progress}} = \underbrace{\tau}_{\text{true progress}} + \underbrace{\frac{\sigma(\bx_k)}{\mathbb{E}\left[\t{f}(\bx_0)-\t{f}(\bx^*)\right]}\epsilon_k}_{\text{sampling error}}. \ee
Since $|\epsilon_k|\sim\bigO(1)$, we can expect $\t{\tau}$ and $\tau$ to be similar in size whenever
\be \tau \: \text{and/or} \:\t{\tau} \:\gg\: \frac{\sigma(\bx_k)}{\mathbb{E}\left[\t{f}(\bx_0)-\t{f}(\bx^*)\right]}. \label{eq_tau_thresh} \ee
Effectively, \eqref{eq_tau_thresh} provides a limit on the accuracy we can reasonably expect a solver to achieve, given the noise level in a particular problem.
In our numerical results, we approximate \eqref{eq_tau_thresh} in a way that is independent of $\bx_k$, and say that the best accuracy we can expect a solver to achieve is $\tau_{crit}(p)$, where
\be  \tau_{crit}(p) \defeq 10^{\lceil \log_{10}\widehat{\tau}(p)\rceil}, \qquad \text{where} \qquad \widehat{\tau}(p) \defeq \frac{\sigma(\bx^*)}{\mathbb{E}\left[\t{f}(\bx_0)-\t{f}(\bx^*)\right]}. \label{eq_tau_thresh_used} \ee

Finally, to construct our data profiles, we choose a desired level of accuracy $\tau$, usually $10^{-5}$, and set our problem-specific tolerance to be either $\tau$ if we can expect the solver to reach this accuracy, otherwise we choose $\tau_{crit}(p)$. Thus, in our data profiles \eqref{eq_data_profile}, we use the problem-specific accuracy level
\be \tau_p \defeq \min(\tau_{max}, \max(\tau_{crit}(p), \tau)), \label{eq_tau_modification} \ee
where $\tau_{max}\defeq 10^{-1}$ is an upper bound on $\tau_p$.
We note that for noiseless problems we have $\tau_{crit}(p)=0$, so $\tau_p=\tau$ is problem-independent.

\paragraph{Impact of $\tau_p$}
We now illustrate that using the per-problem accuracy level $\tau_p$ for measuring noisy progress $\t{N}_p$ allows us to compare performance based on genuine objective decreases, not sampling errors.
We do this by verifying that the performance measures $N_p$ and $\t{N}_p$, and data profiles $d_{\mathcal{S}}$ and $\t{d}_{\mathcal{S}}$, give similar results.

First, we consider the problem and noise model used in Figures~\ref{fig_restarts_motivation} and \ref{fig_restarts_autodetect_motivation}, where $\tau_{crit}(p)=10^{-7}$.
In two runs\footnote{\:Although the same random seed was used in both runs for reproducibility, since they have a different sequence of restarts, the set of iterates is ultimately different.} of the same problem, shown in Figures~\ref{fig_restarts_motivation_with_restarts_obj_redn} and \ref{fig_restarts_autodetect_motivation_obj_redn}, we see that the decreases achieved under both measures are essentially identical until they reach accuracy level very close to $\tau_{crit}(p)$.

\begin{figure}
	\centering
	\begin{subfigure}[b]{0.48\textwidth}
		\includegraphics[width=\textwidth]{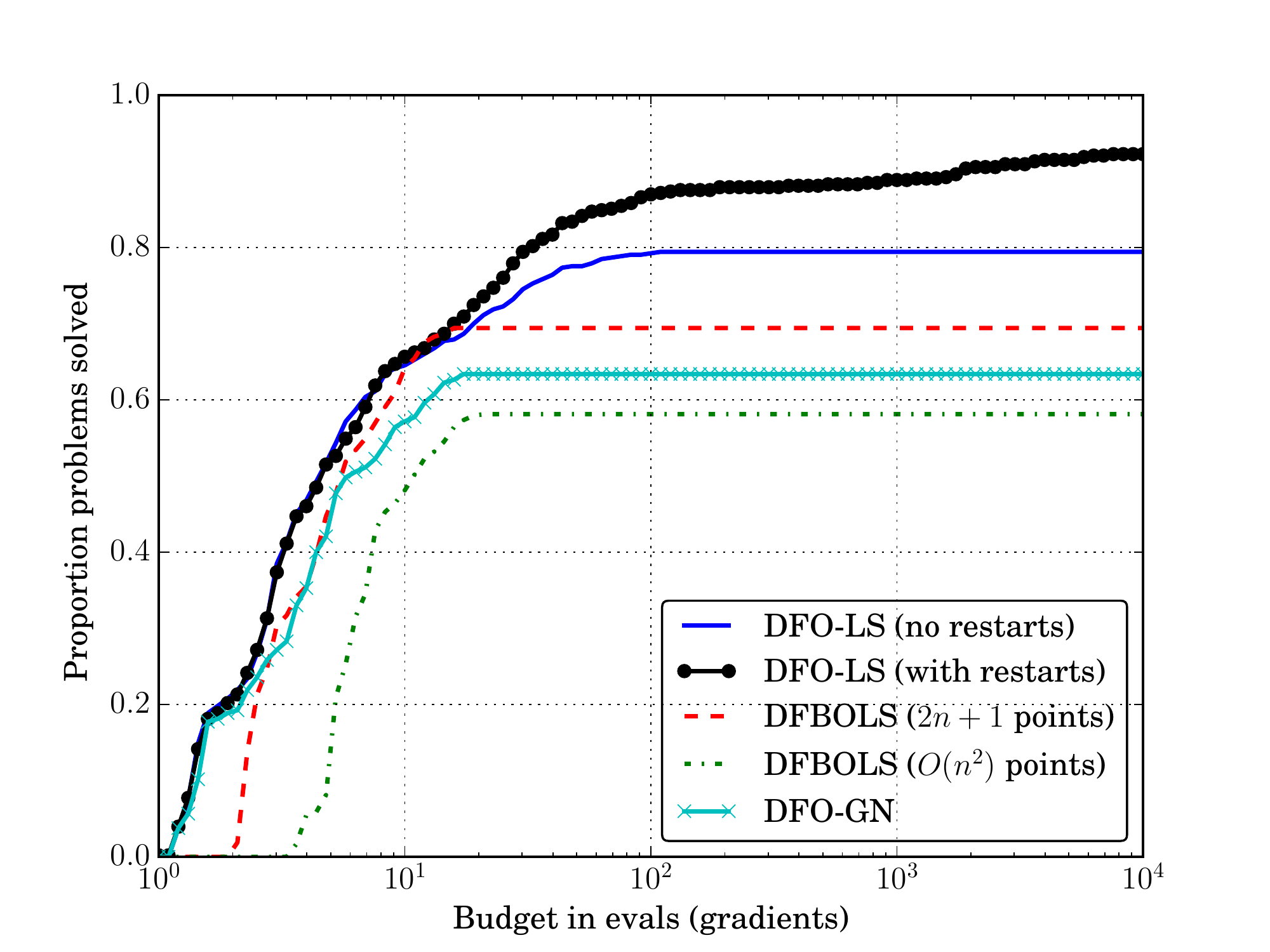}
		\caption{Constant $\tau_p$, noisy $\t{d}_{\mathcal{S}}$ data profile}
		\label{fig_threshold_demonstration_tau_fixed_noisyf}
	\end{subfigure}
	~
	\begin{subfigure}[b]{0.48\textwidth}
		\includegraphics[width=\textwidth]{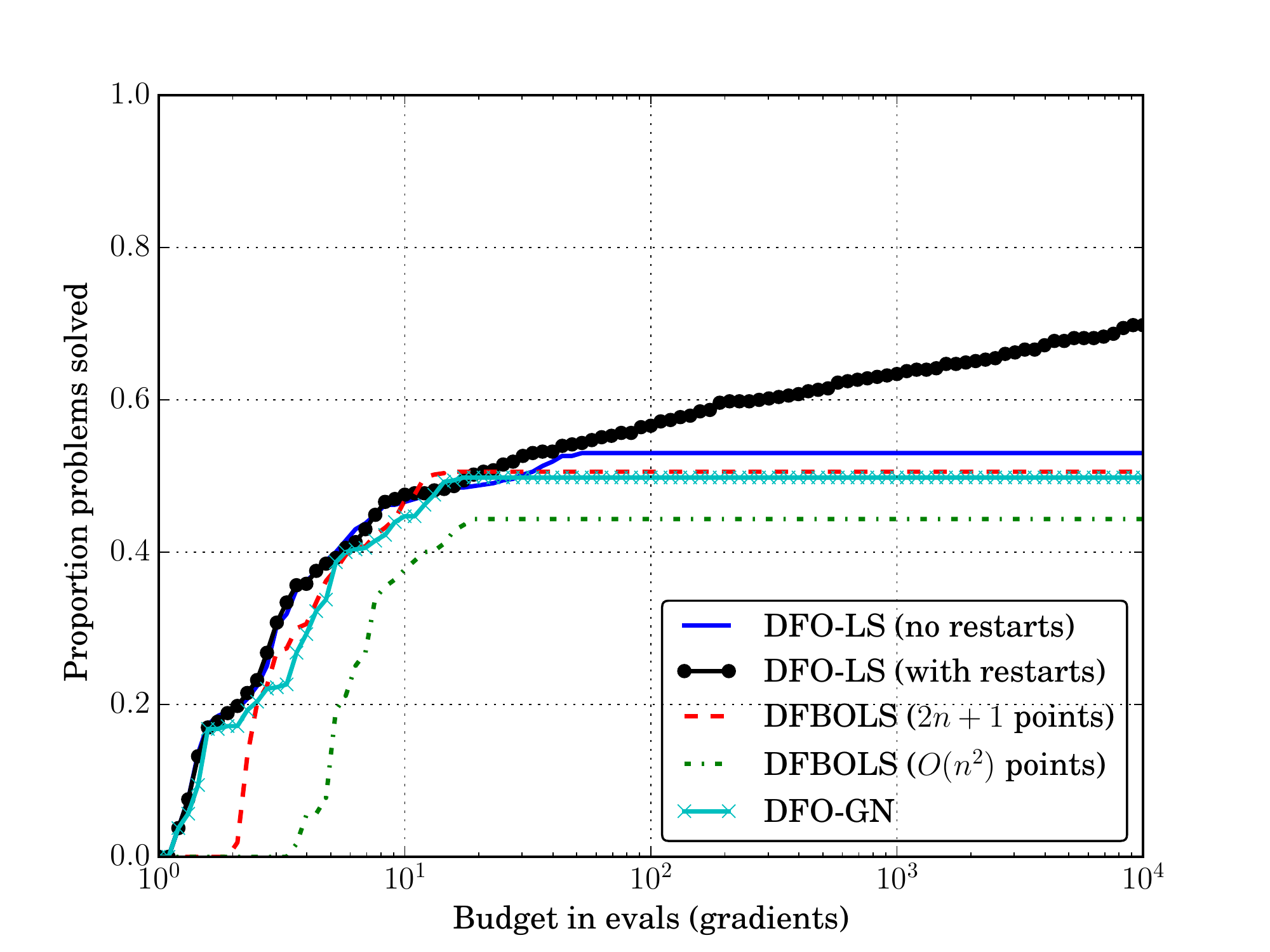}
		\caption{Constant $\tau_p$, true $d_{\mathcal{S}}$ data profile}
		\label{fig_threshold_demonstration_tao_fixed_truef}
	\end{subfigure}
	\\
	\begin{subfigure}[b]{0.48\textwidth}
		\includegraphics[width=\textwidth]{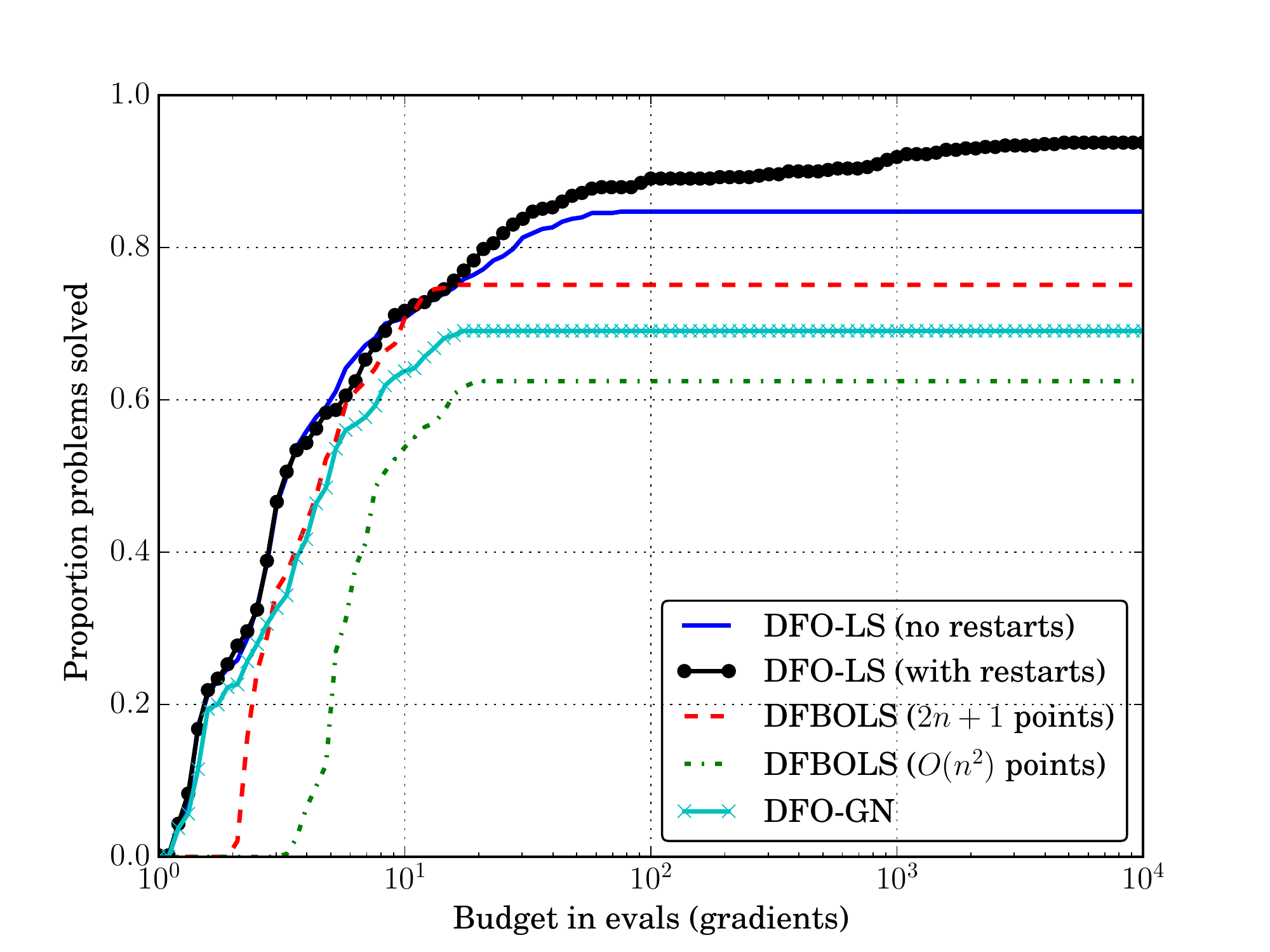}
		\caption{Per-problem $\tau_p$, noisy $\t{d}_{\mathcal{S}}$ data profile}
		\label{fig_threshold_demonstration_tao_prob_adj_noisyf}
	\end{subfigure}
	~
	\begin{subfigure}[b]{0.48\textwidth}
		\includegraphics[width=\textwidth]{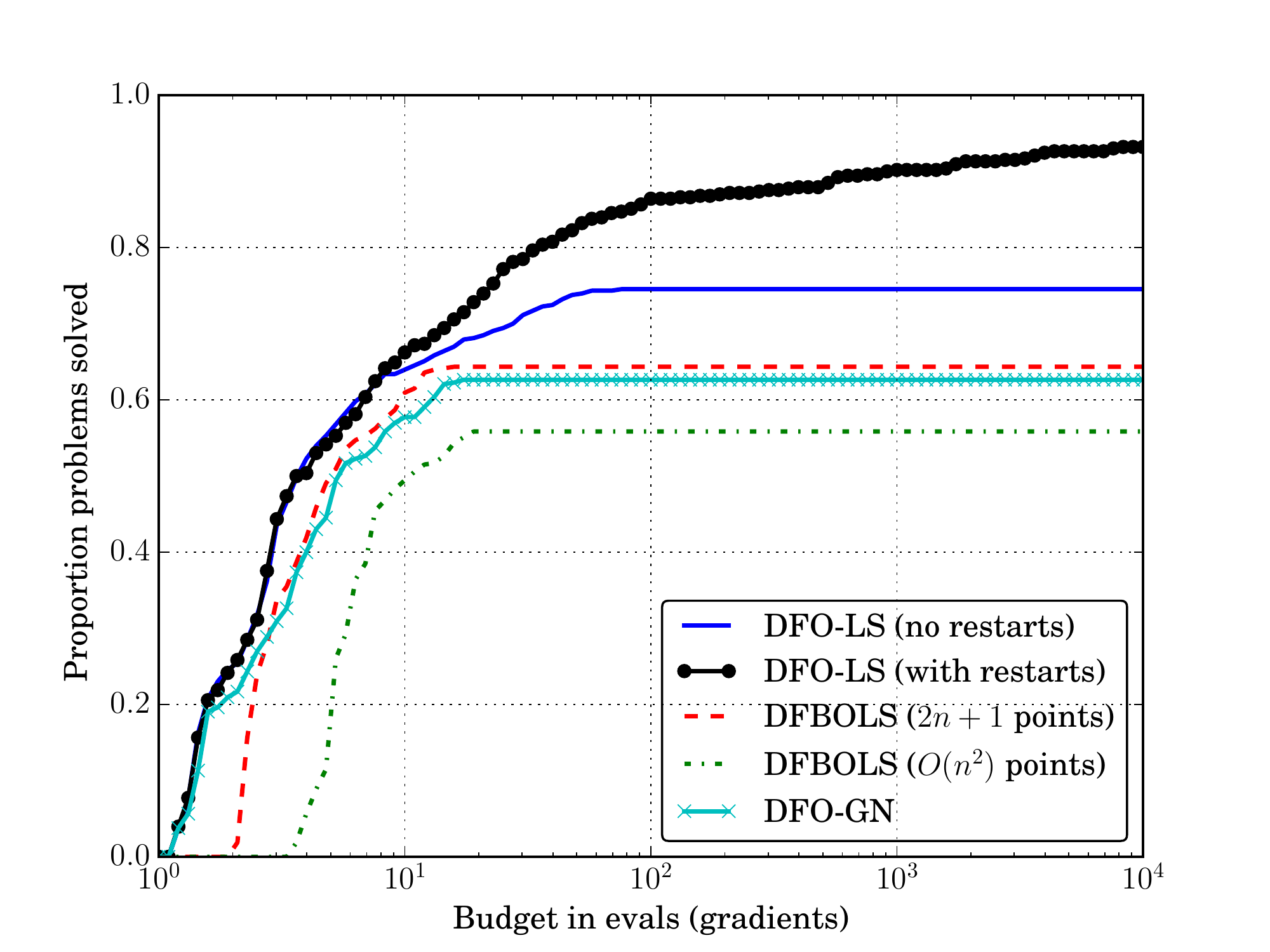}
		\caption{Per-problem $\tau_p$, true $d_{\mathcal{S}}$ data profile}
		\label{fig_threshold_demonstration_tau_prob_adj_truef}
	\end{subfigure}
	\caption{A comparison of the results in \figref{fig_basic_noise2} with additive Gaussian noise, using the data profiles $d_{\mathcal{S}}$ and $\t{d}_{\mathcal{S}}$ \eqref{eq_data_profile}, and either choosing $\tau_p=10^{-5}$ for all problems, or applying the per-problem threshold \eqref{eq_tau_modification}.}
	\label{fig_threshold_demonstration}
\end{figure}

Next, in \figref{fig_threshold_demonstration}, we show a set of data profiles which we will discuss\footnote{\:This figure compares DFO-LS to other solvers for objectives with additive Gaussian noise; see \figref{fig_basic_noise2_addgsn_noisyf}. The same conclusions may be drawn by using the other plots in this paper.} in \secref{sec_dfols_benchmarking}, but we compare the same solvers using both the `noisy $\t{f}$' and `true $f$' measures of objective reduction, and consider either a fixed $\tau=10^{-5}$ for all problems, or using the per-problem value $\tau_p$ \eqref{eq_tau_modification}. 
We see that when a constant value of $\tau_p$ is used, the two data profiles look very different, with the noisy $\t{d}_{\mathcal{S}}$ profile showing more problems solved than the true $d_{\mathcal{S}}$ profile.
When we switch to the per-problem threshold \eqref{eq_tau_modification}, the two profiles look much more consistent both in shape and magnitude.
Most importantly, conclusions about relative solver performance in both low-budget and long-budget regimes based on the $\t{d}_{\mathcal{S}}$ profile would be consistent with those based on the true $d_{\mathcal{S}}$ measure.

Thus, for the remainder of the paper, we present our numerical results  using the problem-adjusted $\tau_p$ with the noisy data profile $\t{d}_{\mathcal{S}}$ (e.g.~\figref{fig_threshold_demonstration_tao_prob_adj_noisyf}).


\subsection{Test Problems and Solver Settings} \label{sec_test_problems}
We test DFO-LS on the two collections used in \cite{Cartis2017a}:
\begin{description}
	\item[\rm \textit{(MW)}] The set of 53 nonlinear least-squares from Mor\'e and Wild \cite{More2009}. The problems are low-dimensional, with $2 \leq n \leq 12$ and $n \leq m \leq 65$, so this collection is used as the main test set for the `noisy' regime (see noise models below);
	\item[\rm \textit{(CR)}] The set of 60 nonlinear least-squares from \cite{Cartis2017a}, available via the CUTEst package \cite{Gould2015}. The problems are medium-sized, with $25 \leq n \leq 120$ and $n\leq m \leq 400$, so this collection is used as the main test set for the `expensive' regime.
\end{description}
Full details of both collections may be found in \cite{Cartis2017a}.

\paragraph{Noise Models}
For results robustness, we allow the evaluation of $\br(\bx)$ to include several types of stochastic noise.
In the following sections, we show results for the following noise models, where $\t{f}(\bx)=\sum_{i=1}^m \t{r}_i(\bx)^2$, 
\begin{itemize}
	\item Smooth (noiseless) evaluations: $\t{r}_i(\bx) = r_i(\bx)$;
	\item Multiplicative Gaussian noise: $\t{r}_i(\bx) = (1+\epsilon)r_i(\bx)$, where $\epsilon\sim N(0,\sigma^2)$;
	\item Additive Gaussian noise: $\t{r}_i(\bx) = r_i(\bx) + \epsilon$, where $\epsilon\sim N(0,\sigma^2)$; and
	\item Additive $\chi^2$ noise: $\t{r}_i(\bx) = \sqrt{r_i(\bx)^2 + \epsilon^2}$, where $\epsilon\sim N(0,\sigma^2)$.
\end{itemize}
In each case, $\epsilon$ is drawn i.i.d.~for each $\bx$ and each $i=1,\ldots,m$; and  $\sigma=10^{-2}$.
For noisy problems, our goal is to minimize $\t{f}(\bx)$ in expectation --- note that $\mathbb{E}[\t{f}(\bx)]$ is an affine transformation of $f(\bx)$ for these noise models, so they have the same minimizer(s).

\paragraph{Solver Settings}
In the below, we compare DFO-LS v1.0.1 against DFO-GN v0.2 \cite{Cartis2017a} and DFBOLS \cite{Zhang2010}.
For DFBOLS, we show results using with $2n+1$ and $(n+1)(n+2)/2$ interpolation points.
For all solvers, we choose trust region settings $\Delta_0=0.1\max(\|\bx_0\|_{\infty}, 1)$ and $\rho_{end}=10^{-8}$, and the default values for all other parameters (unless otherwise specified).

We used a maximum budget of $10^4 (n+1)$ evaluations for the Mor\'e and Wild problems (MW).
Particularly for noisy problems, we are interested in a regime where objective evaluations are cheap, and we are concerned with the robustness of each solver --- how many problems can it solve, if budget were not an issue.
Since this budget is much larger than is often used for testing (e.g.~\cite{Zhang2010,Cartis2017a}), we show data profiles with a log-scale for budget, so we can easily compare solvers both for large budgets (to check robustness) and for realistically small budgets.
For the CUTEst problems (CR), we used a much smaller budget of $50(n+1)$ evaluations, to represent the other regime, where objectives are expensive to evaluate.

\section{Numerical Studies of New DFO-LS Features} \label{sec_new_features_results}
We test the new  features of DFO-LS 
and showcase the successful ones which are chosen as defaults, with the remaining features available as options.
\secref{sec_dfols_benchmarking} then compares DFO-LS with its default settings against state-of-the-art  DFO least-squares solvers.

\subsection{Reduced Initialization Cost} \label{sec_growing_testing}
In \figref{fig_growing_smooth}, we consider the CUTEst problems (CR) with noiseless evaluations.
We compare the basic implementation of DFO-LS against DFO-LS with a reduced initialization cost of 2, $n/4$ and $n/2$ function evaluations (growing the direction space via both mechanisms described in \secref{sec_regression_models} above: modifying $J_k$ using its SVD, and perturbing the trust region step). 
Using our small budget of $50(n+1)$ evaluations, we show data profiles in two settings: for a small budget ($5(n+1)$ evaluations) with low accuracy $\tau=10^{-1}$, and the full budget with high accuracy $\tau=10^{-5}$.

In the short budget, low accuracy plots, we see the benefit of a reduced initialization cost --- we are able to solve a notable fraction of the problems to low accuracy with very few evaluations; less than the number required to perform a single gradient evaluation.
We also see the tradeoff of this benefit, which is a lower performance at small budgets (1--3 gradients).
The long budget, high accuracy plots show that, although the reduced initialization cost suffers a performance loss relative to the basic DFO-LS implementation for small budgets, this does not translate into a loss of performance for longer budgets; the robustness of DFO-LS is maintained even with this very small initialization cost.
The difference between initializing with 2, $n/4$ and $n/2$ points is not substantial.
Comparing the two mechanisms for increasing the model dimensionality, we find that the SVD approach performs similarly to the perturbed trust region approach for small budgets, but better matches DFO-LS with a full initialization set.

\begin{figure}
	\centering
	\begin{subfigure}[b]{0.48\textwidth}
		\includegraphics[width=\textwidth]{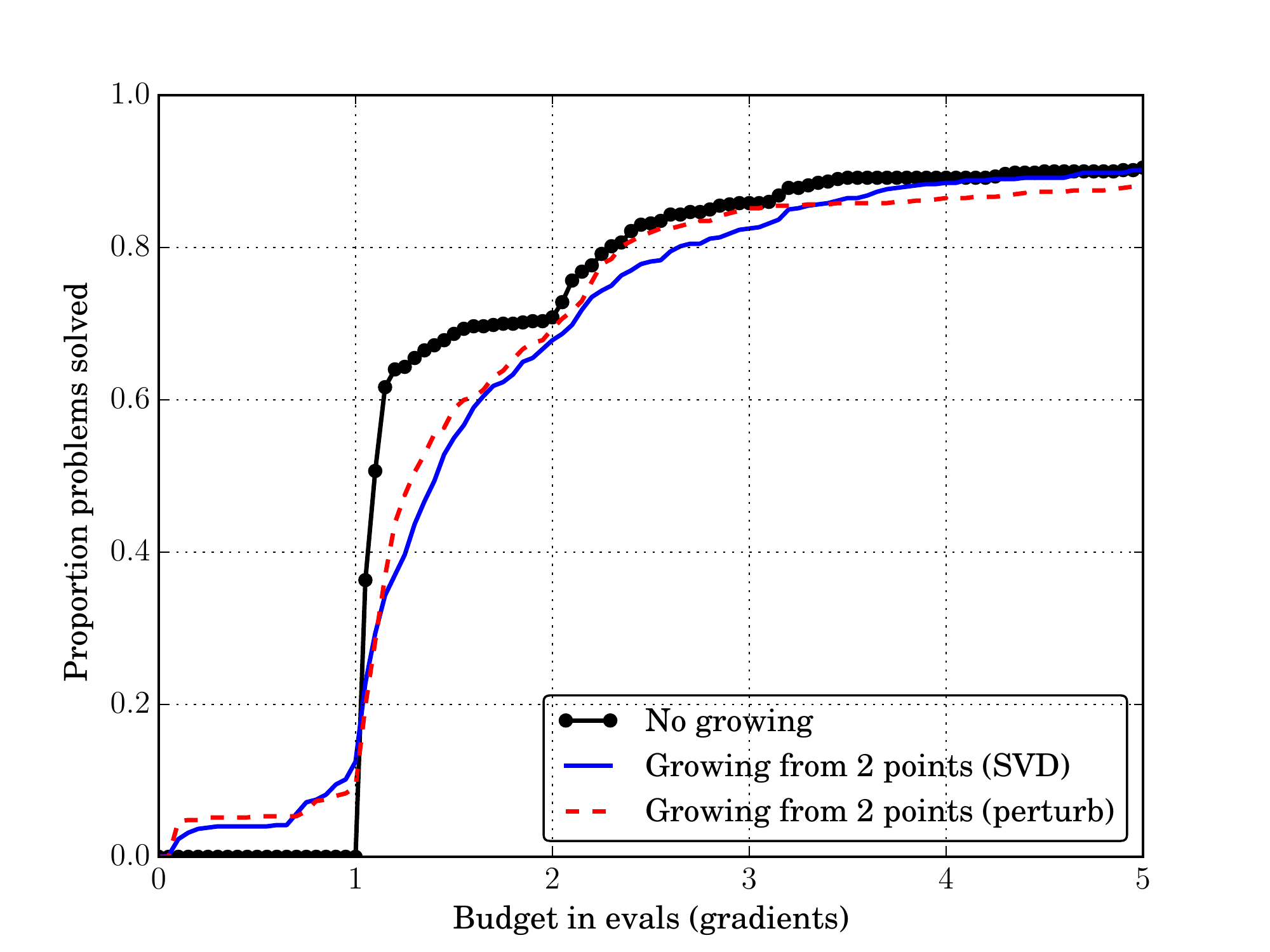}
		\caption{Short budget, 2 starting points, $\tau=10^{-1}$}
		\label{fig_growing_smooth_easy}
	\end{subfigure}
	~
	\begin{subfigure}[b]{0.48\textwidth}
		\includegraphics[width=\textwidth]{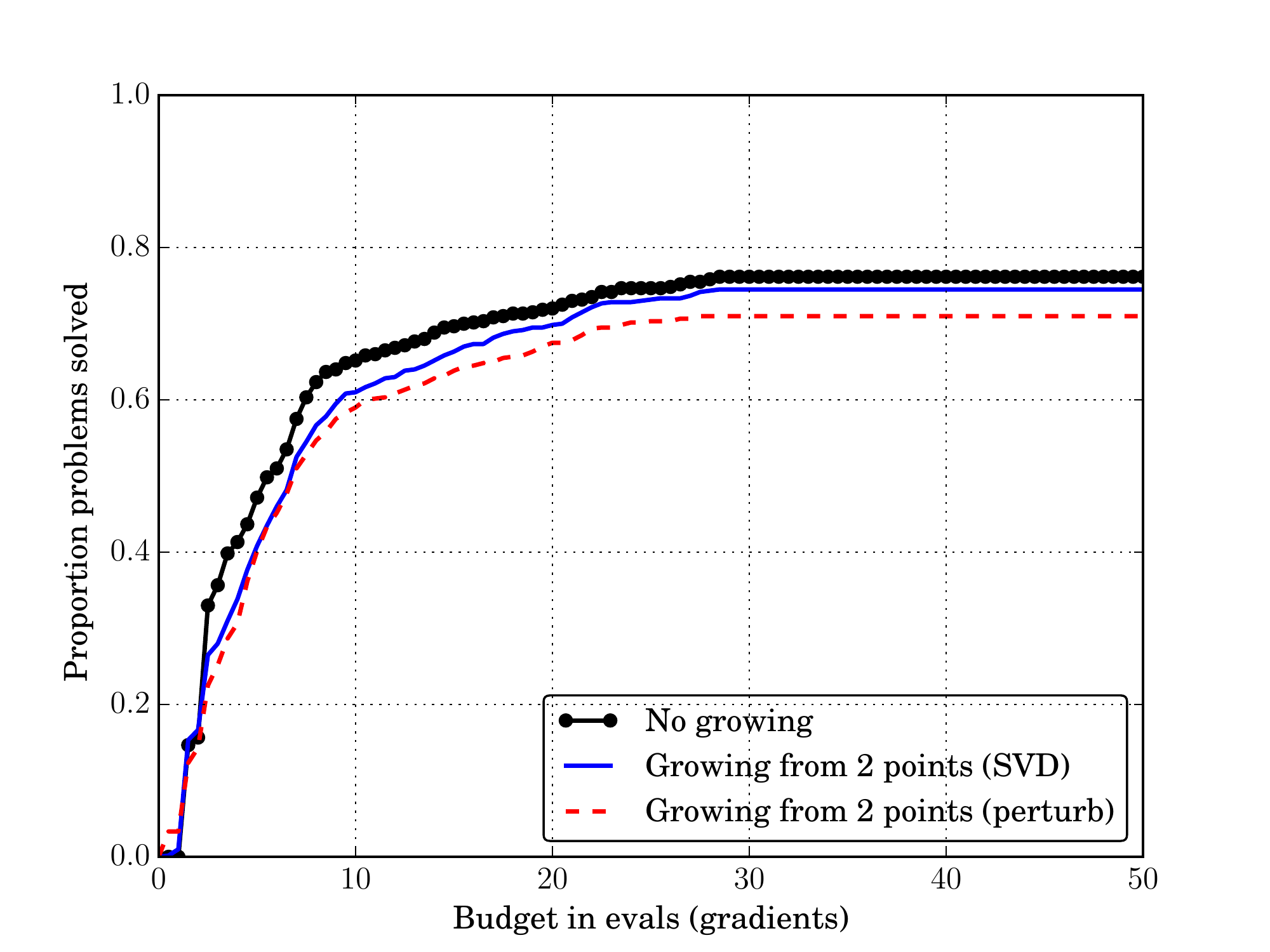}
		\caption{Long budget, 2 starting points, $\tau=10^{-5}$}
		\label{fig_growing_smooth_hard}
	\end{subfigure}
	\\
	\begin{subfigure}[b]{0.48\textwidth}
		\includegraphics[width=\textwidth]{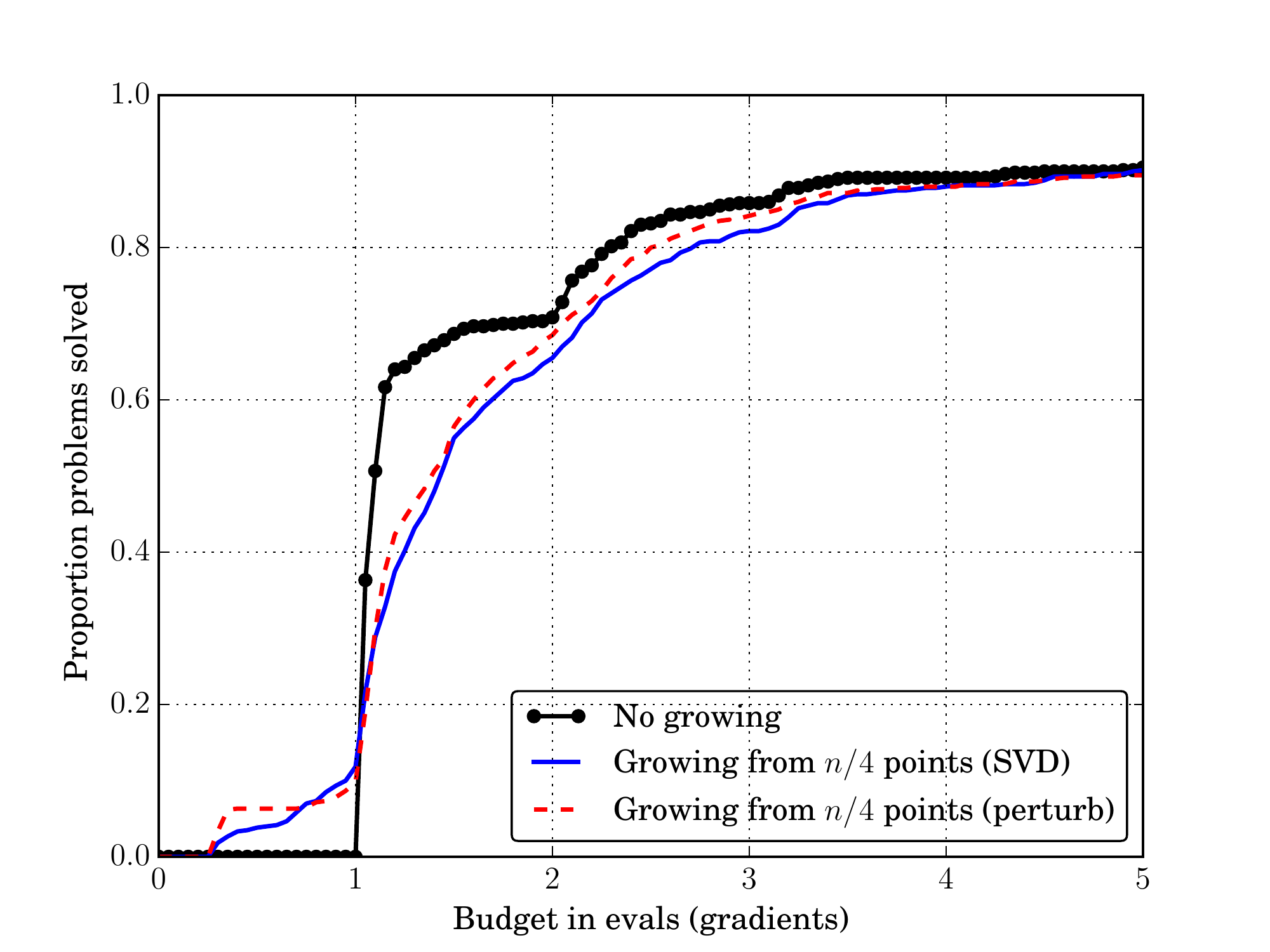}
		\caption{Short budget, $n/4$ starting points, $\tau=10^{-1}$}
		\label{fig_growing_qtrn_smooth_easy}
	\end{subfigure}
	~
	\begin{subfigure}[b]{0.48\textwidth}
		\includegraphics[width=\textwidth]{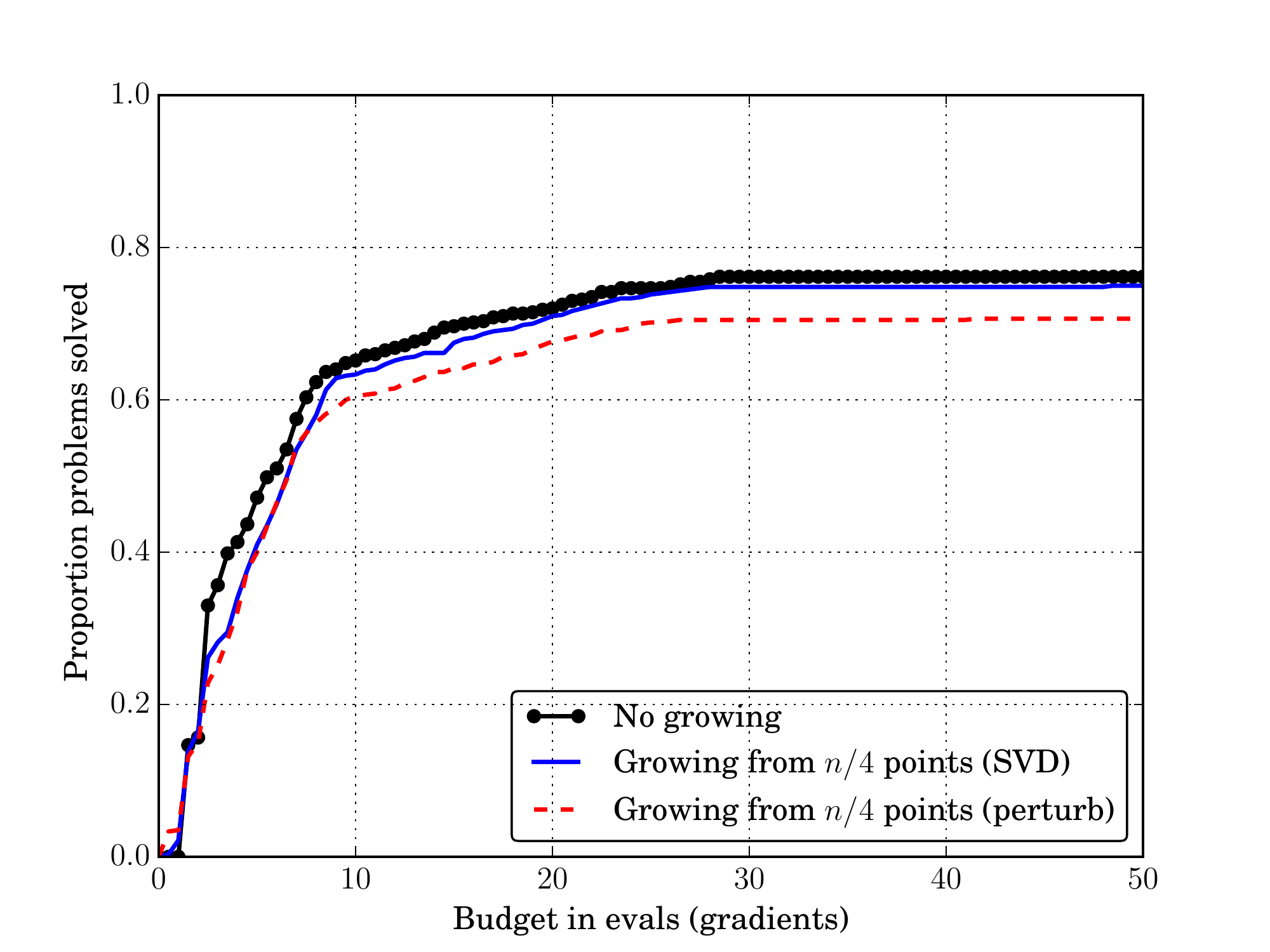}
		\caption{Long budget, $n/4$ starting points, $\tau=10^{-5}$}
		\label{fig_growing_qtrn_smooth_hard}
	\end{subfigure}
	\\
	\begin{subfigure}[b]{0.48\textwidth}
		\includegraphics[width=\textwidth]{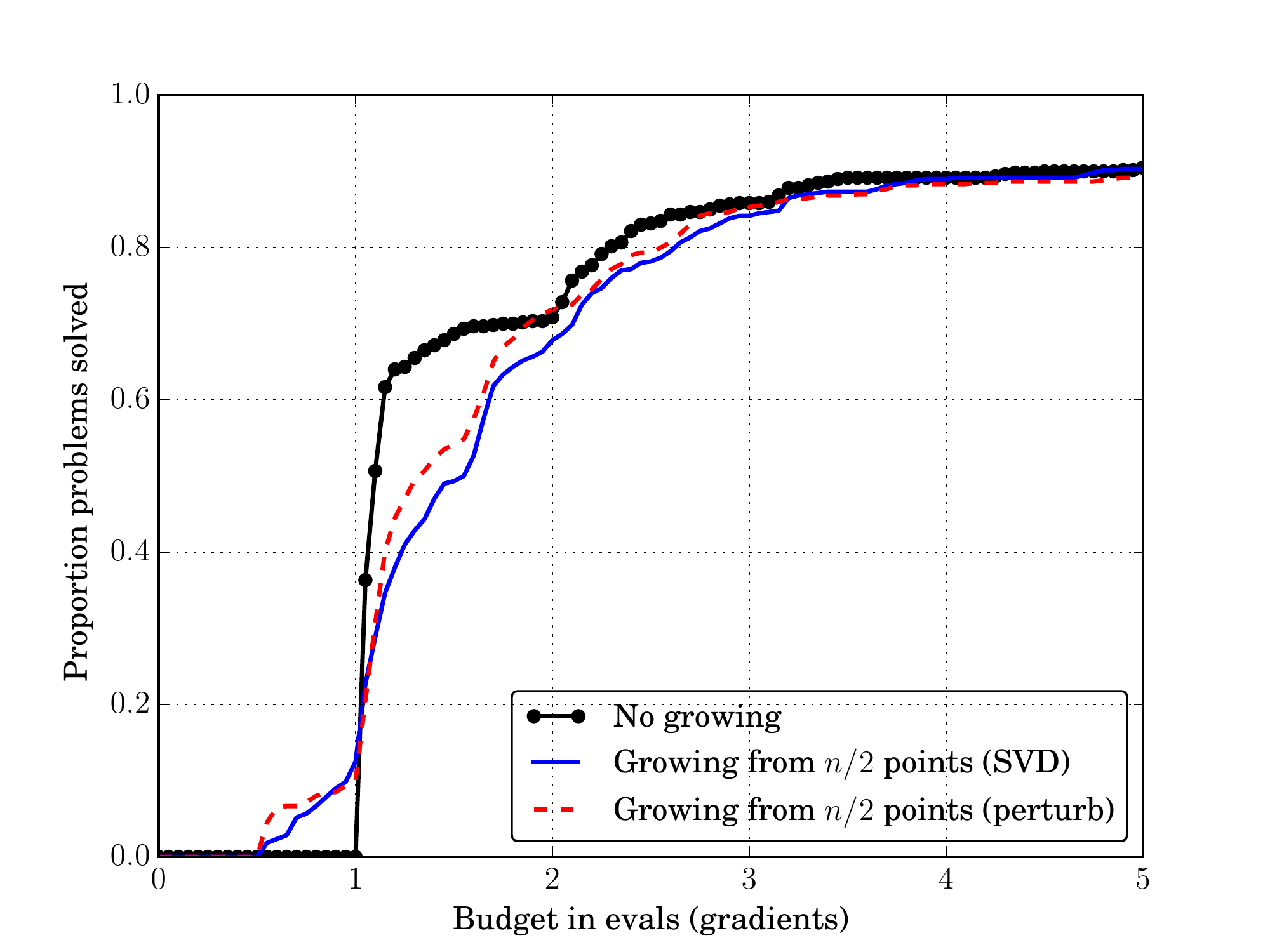}
		\caption{Short budget, $n/2$ starting points, $\tau=10^{-1}$}
		\label{fig_growing_halfn_smooth_easy}
	\end{subfigure}
	~
	\begin{subfigure}[b]{0.48\textwidth}
		\includegraphics[width=\textwidth]{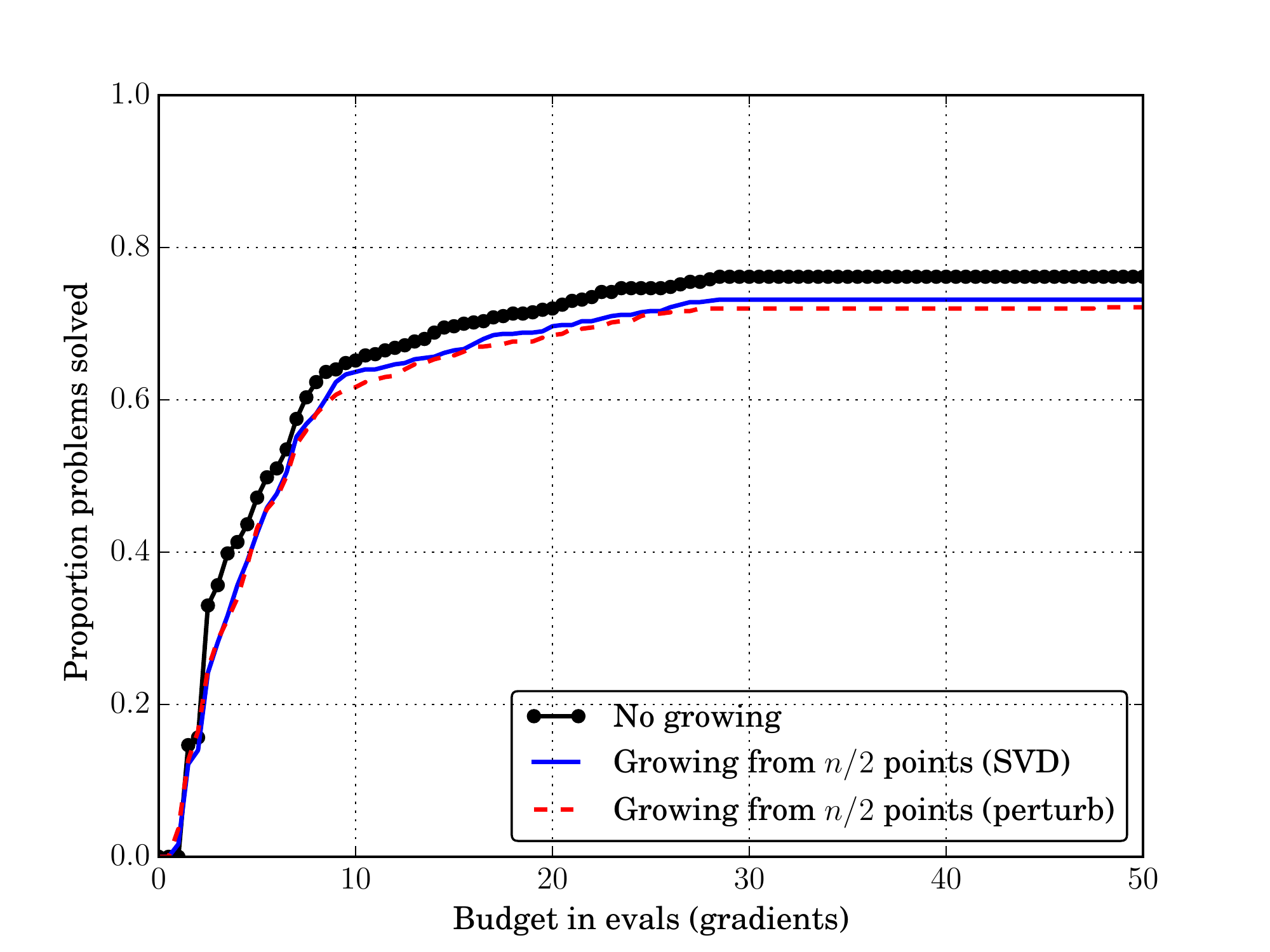}
		\caption{Long budget, $n/2$ starting points, $\tau=10^{-5}$}
		\label{fig_growing_halfn_smooth_hard}
	\end{subfigure}
	\caption{Data profiles showing the impact of the reduced initialization cost of DFO-LS (using $n+1$ interpolation points) against using the full initial set, for smooth objectives. Results an average of 10 runs in each case. The problem collection is (CR).}
	\label{fig_growing_smooth}
\end{figure}

\subsection{Sample Averaging} \label{sec_averaging}
To demonstrate the impact of using sample averaging, \figref{fig_avg_noise2} shows data profiles
 with averaging strategies $N\in\{1, 2, 5, 10, 30, \max(1,\lfloor\Delta_k^{-1}\rfloor)\}$.
Each of the $N$ samples for a given $\bx_k$ are counted towards the maximum computational budget of $10^4 (n+1)$ values.

Unsurprisingly, we see that using a larger number of samples can improve the robustness of DFO-LS.
Of course, to achieve this robustness, a proportionally larger number of evaluations are required, so for small-to-medium budgets (in serial) we lose in performance.
This does not take into account the benefits of parallelization that may be available when sample averaging is used. 
We also notice that using $N=\bigO(\Delta_k^{-1})$ can provide a compromise --- it still makes progress for small budgets, but manages to achieve a reasonable level of robustness overall.

\begin{figure}[t]
	\centering
	\begin{subfigure}[b]{0.48\textwidth}
		\includegraphics[width=\textwidth]{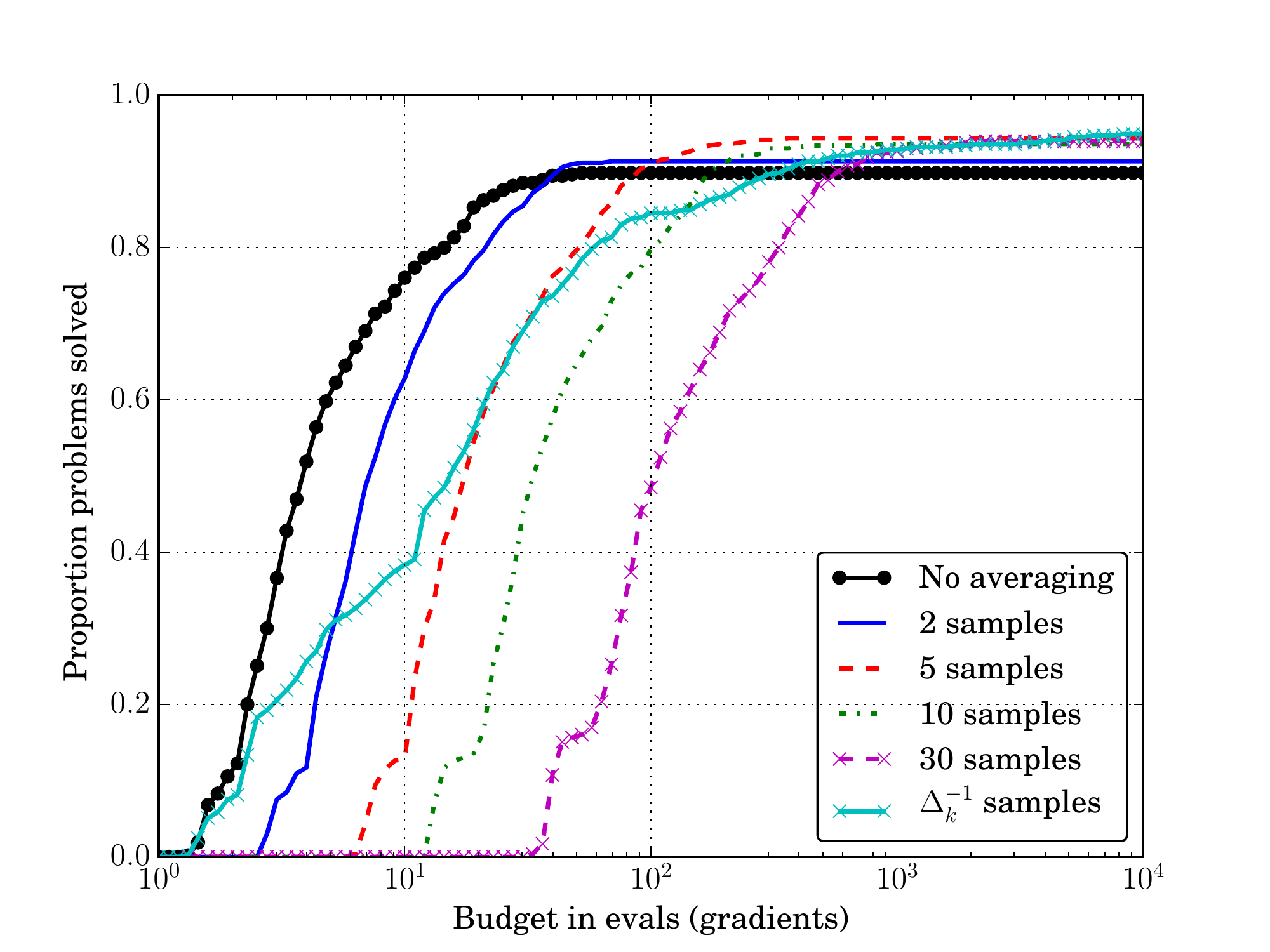}
		\caption{Multiplicative Gaussian noise}
		\label{fig_avg_noise2_ubgsn_noisyf}
	\end{subfigure}
	~
	\begin{subfigure}[b]{0.48\textwidth}
		\includegraphics[width=\textwidth]{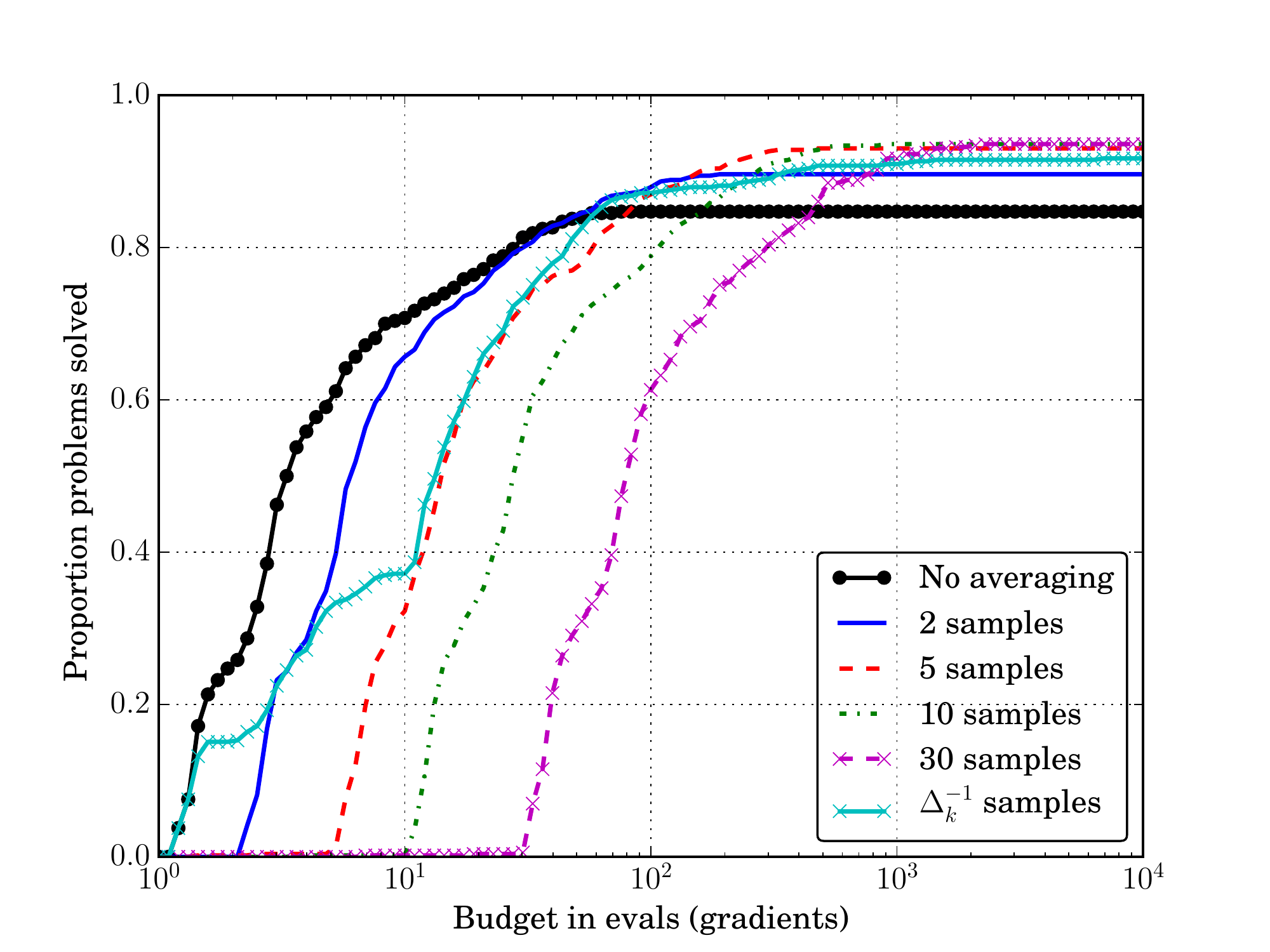}
		\caption{Additive Gaussian noise}
		\label{fig_avg_noise2_addgsn_noisyf}
	\end{subfigure}
	\\
	\begin{subfigure}[b]{0.48\textwidth}
		\includegraphics[width=\textwidth]{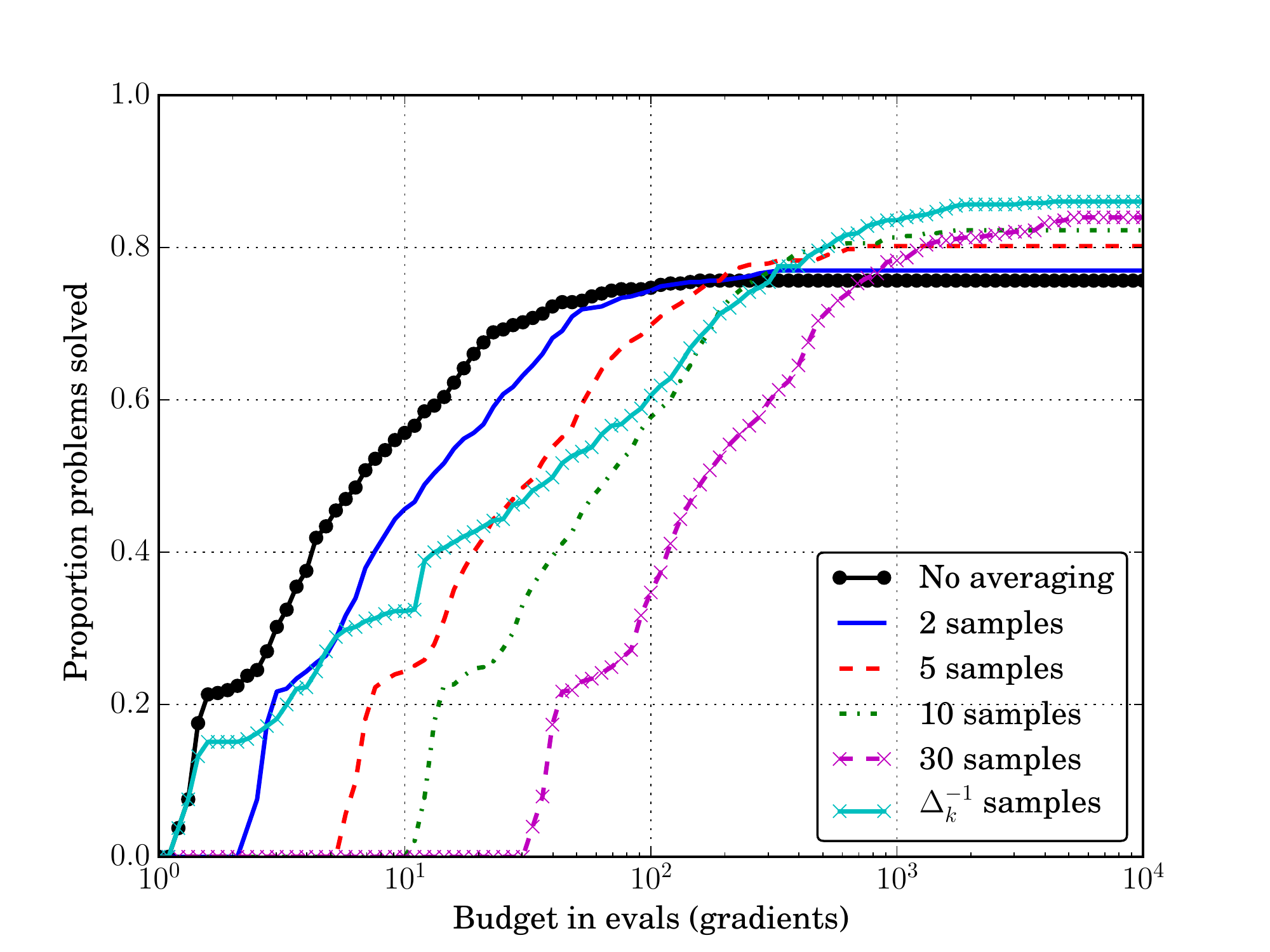}
		\caption{Additive $\chi^2$ noise}
		\label{fig_avg_noise2_addchisq_noisyf}
	\end{subfigure}
	\caption{Comparison of different sample averaging methods for DFO-LS (using $n+1$ interpolation points). We are using noisy objective evaluations with $\sigma=10^{-2}$, high accuracy $\tau=10^{-5}$, and an average of 10 runs for each solver. The problem collection is (MW).}
	\label{fig_avg_noise2}
\end{figure}

\subsection{Regression Models} \label{sec_regression_results}
In practice, we find that the geometry-based moves perform similarly to or slightly better than momentum-based ones, so we do not show results for the latter mechanism. \figref{fig_regression_noise2} compares the remaining two techniques
 with varying numbers of regression points ($|Y_k|=c(n+1)$ for $c\in\{5,30\}$) against interpolation models ($|Y_k|=n+1$).
We see that using a larger sample set improves the robustness of DFO-LS, particularly for additive Gaussian noise, and this improvement (for $|Y_k|=c(n+1)$) is generally comparable to, or slightly worse than, the use of sample averaging (with $c$ samples at each point).
The geometry-based mechanism for moving multiple points makes the algorithm progress more slowly, as indicated by the performance profiles, while at times providing a slight improvement over the `basic' approach (moving one point per iteration).

In \appref{sec_regression_v_avg}, we argue, briefly and in a simplified framework, that regression and sample averaging generate similar model error;
thus, since sample averaging produces a better estimate of objective decrease 
for fixed noise level, we expect that overall, regression will be slightly less robust compared to sample averaging when
considering large computational budgets.


\begin{figure}[t]
	\centering
	\begin{subfigure}[b]{0.48\textwidth}
		\includegraphics[width=\textwidth]{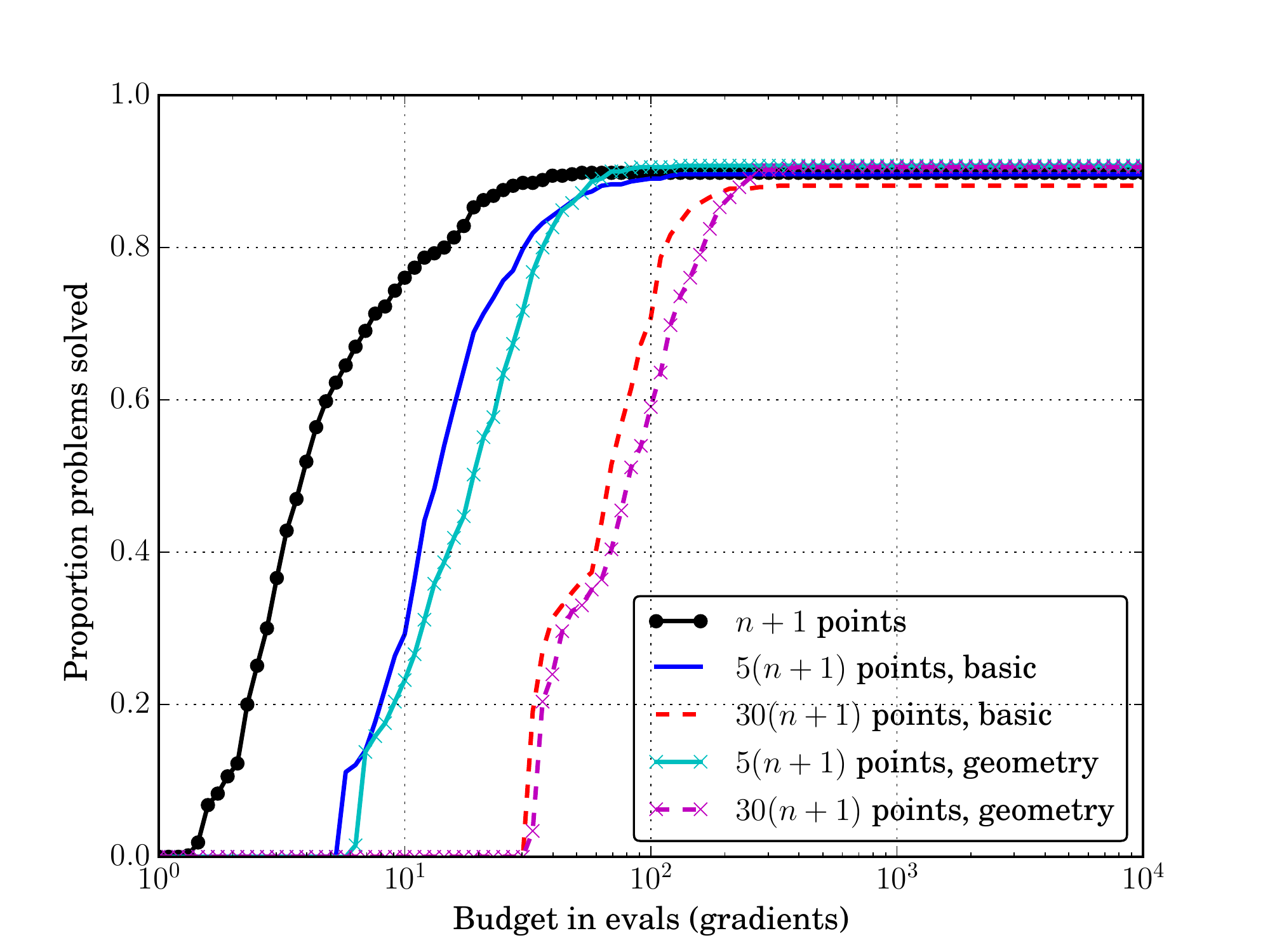}
		\caption{Multiplicative Gaussian noise}
		\label{fig_regression_noise2_ubgsn_noisyf}
	\end{subfigure}
	~
	\begin{subfigure}[b]{0.48\textwidth}
		\includegraphics[width=\textwidth]{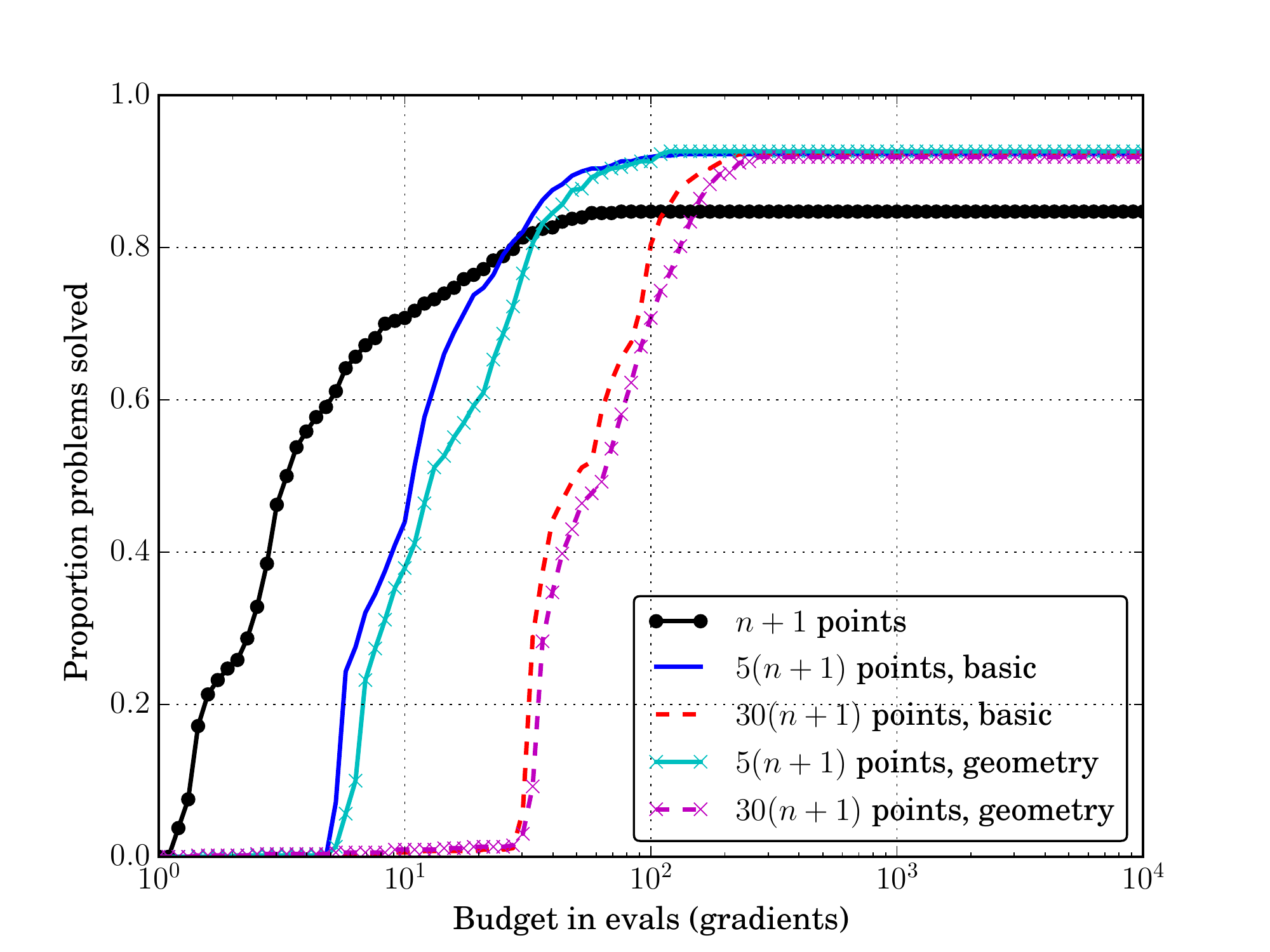}
		\caption{Additive Gaussian noise}
		\label{fig_regression_noise2_addgsn_noisyf}
	\end{subfigure}
	\\
	\begin{subfigure}[b]{0.48\textwidth}
		\includegraphics[width=\textwidth]{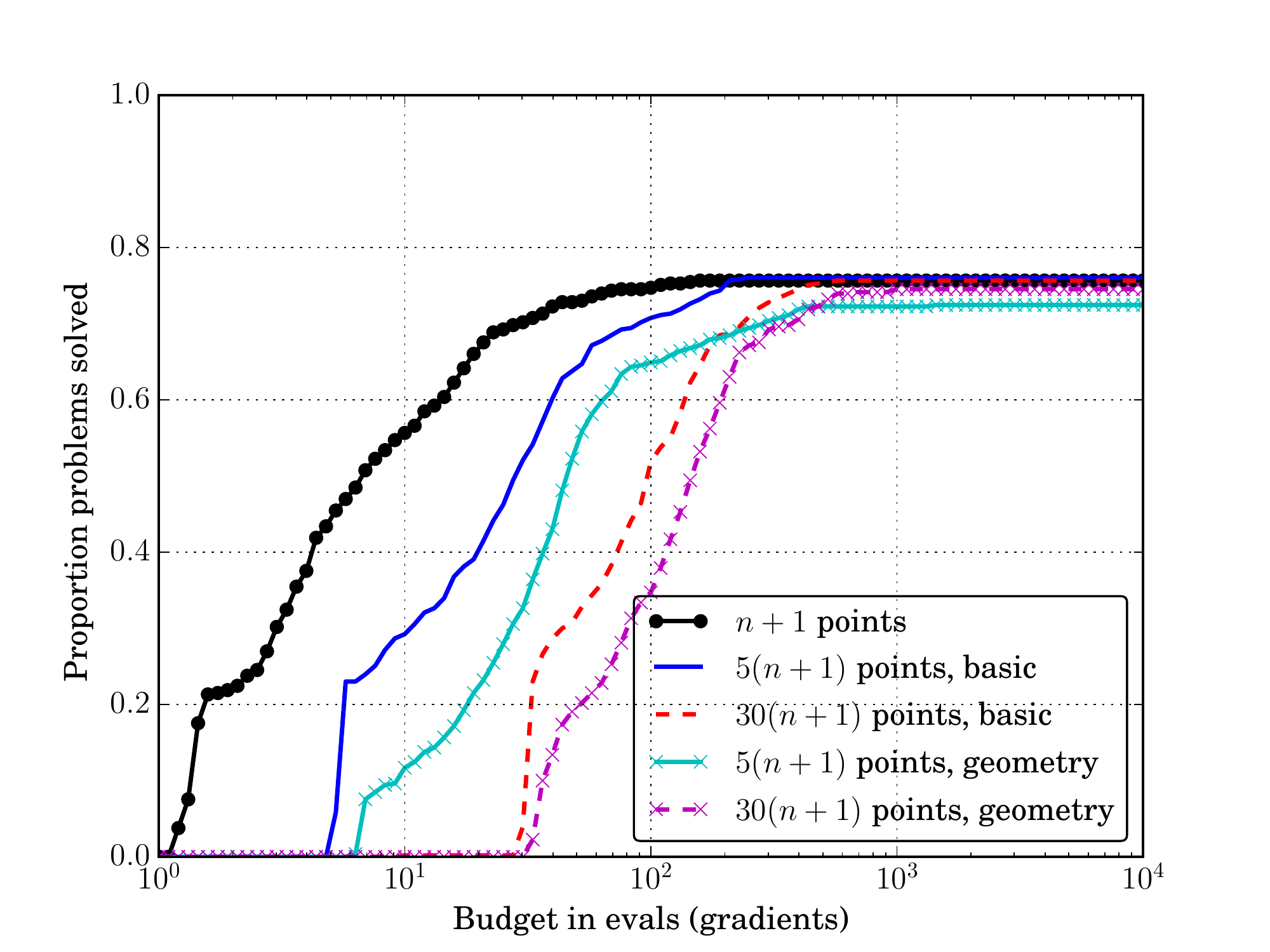}
		\caption{Additive $\chi^2$ noise}
		\label{fig_regression_noise2_addchisq_noisyf}
	\end{subfigure}
	\caption{Impact of regression for DFO-LS. We are using noisy objective evaluations with $\sigma=10^{-2}$, high accuracy $\tau=10^{-5}$, and an average of 10 runs for each solver. The problem collection is (MW).}
	\label{fig_regression_noise2}
\end{figure}

\subsection{Multiple Restarts} \label{sec_restarts}
\figref{fig_restarts_noise2} compares the different restart methods against sample averaging (30 samples at every point).
All runs use auto-detection of restarts and
the optional noise-based termination criterion \eqref{eq_termination_noise}.

We see that soft restarts (moving $\bx_k$) is the most successful restarts mechanism, followed by hard restarts, then soft restarts (fixing $\bx_k$), and that all these mechanisms are better than DFO-LS without any noise-based features.
Compared to the case of using averaging with $\Delta_k^{-1}$ samples at every point, soft restarts (moving $\bx_k$) achieve a similar or better level of robustness with many fewer objective evaluations --- this is most clearly observed in the case of additive $\chi^2$ noise.

The improvements in robustness from using multiple restarts are obvious at the end of the full budget of $10^4$ simplex gradients, but there are still benefits to be found at much smaller budgets (e.g.~$\bigO(100)$ simplex gradients).
As a result of these benefits, the soft restarts (moving $\bx_k$) mechanism is activated by default in DFO-LS for noisy problems.

Next, we consider the impact of using increased levels of sample averaging with every restart.
The reason for this is that after every restart, we hope to be closer to the desired solution, so an increased amount of averaging may help distinguish points near to this optimum.
To achieve this, in \eqref{eq_nsamples} we use
\be N = \mathrm{nsamples}(\rho_k, \Delta_k, k, n_{\mathrm{restarts}}) = \min\{n_{\mathrm{restarts}} + 1, 30\}. \ee
\figref{fig_restarts_avg_noise2} shows that 
 augmenting multiple restarts with sample averaging improves the robustness of hard and soft restarts (fixing $\bx_k$), but not for the default mechanism (soft restarts moving $\bx_k$).
Ultimately, using soft restarts (moving $\bx_k$) is better than the other two restart mechanisms, with or without sample averaging.
Hence, we do not use any sample averaging in DFO-LS by default.

\begin{figure}[t]
	\centering
	\begin{subfigure}[b]{0.48\textwidth}
		\includegraphics[width=\textwidth]{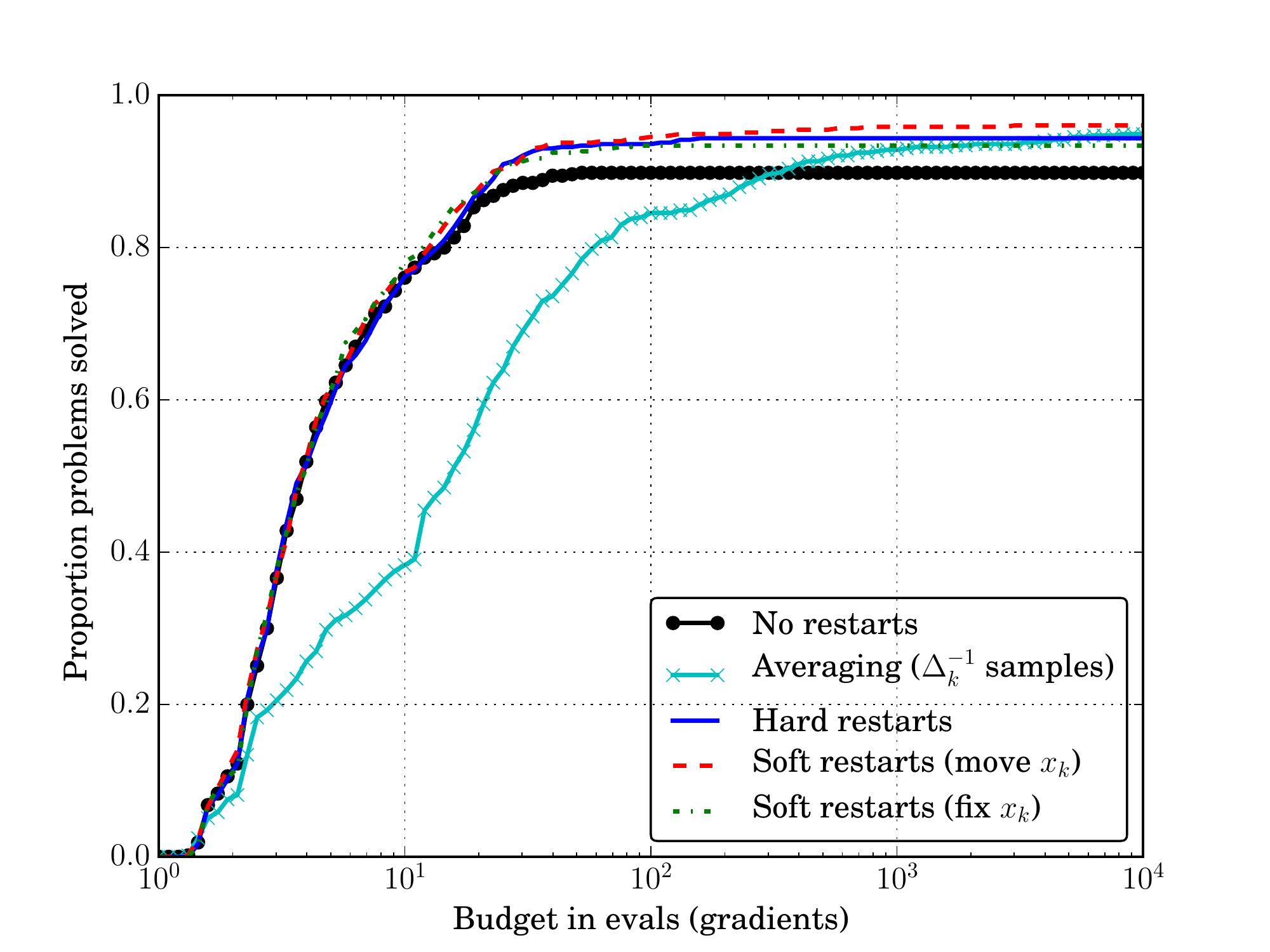}
		\caption{Multiplicative Gaussian noise}
		\label{fig_restarts_noise2_ubgsn_noisyf}
	\end{subfigure}
	~
	\begin{subfigure}[b]{0.48\textwidth}
		\includegraphics[width=\textwidth]{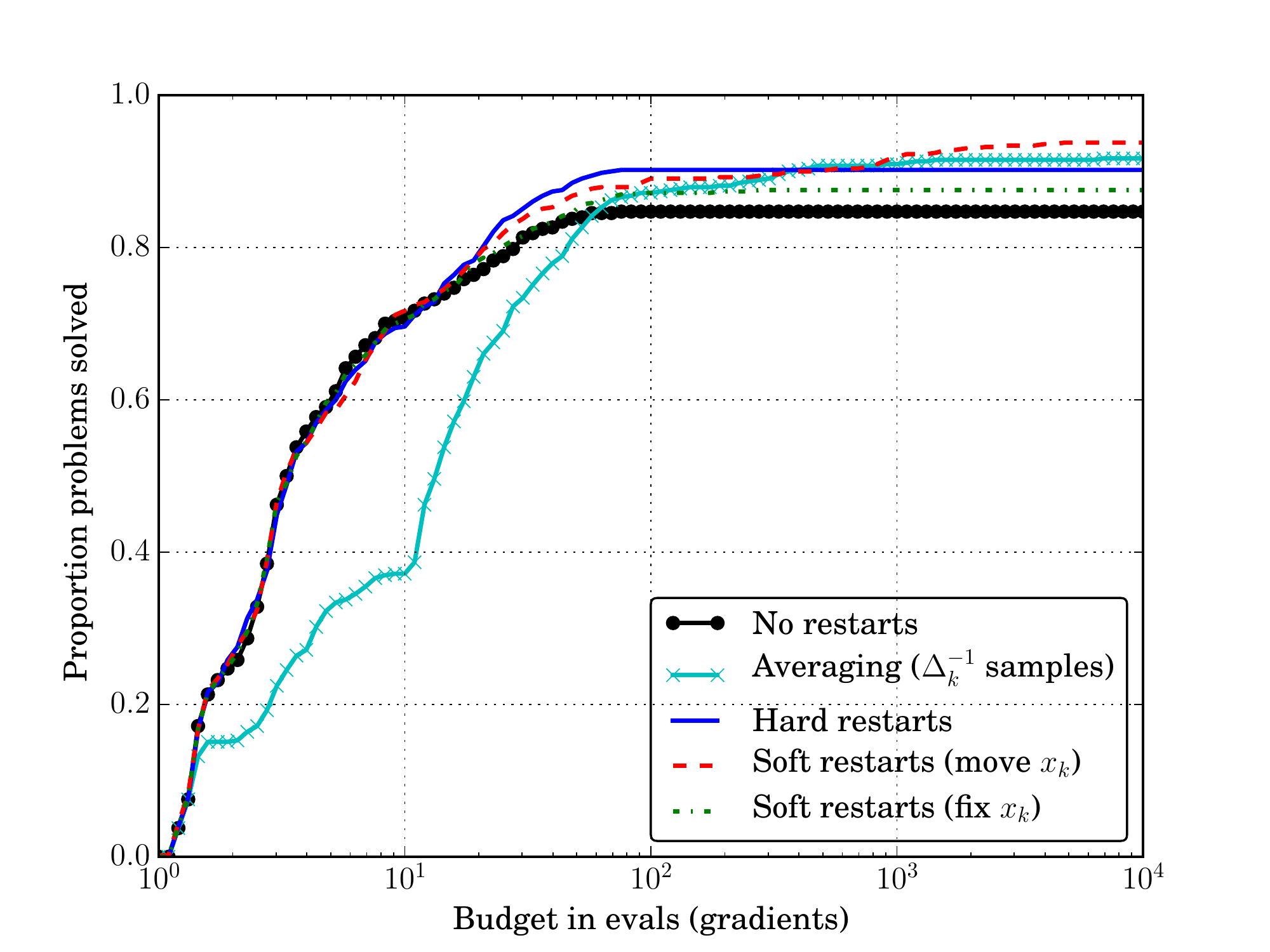}
		\caption{Additive Gaussian noise}
		\label{fig_restarts_noise2_addgsn_noisyf}
	\end{subfigure}
	\\
	\begin{subfigure}[b]{0.48\textwidth}
		\includegraphics[width=\textwidth]{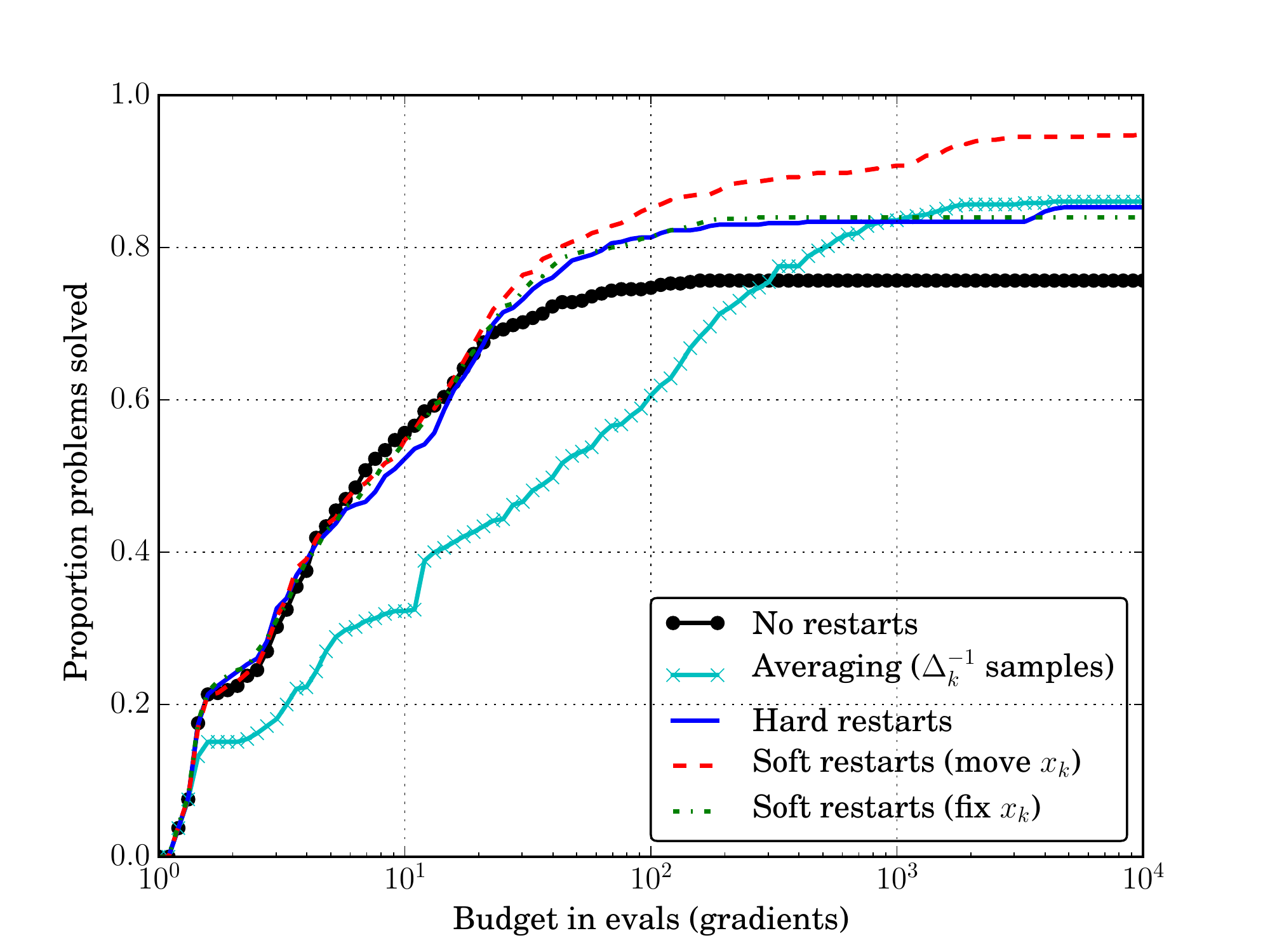}
		\caption{Additive $\chi^2$ noise}
		\label{fig_restarts_noise2_addchisq_noisyf}
	\end{subfigure}
	\caption{Impact of multiple restarts for DFO-LS (using $n+1$ interpolation points). We are using noisy objective evaluations with $\sigma=10^{-2}$, high accuracy $\tau=10^{-5}$, and an average of 10 runs for each solver. The problem collection is (MW).}
	\label{fig_restarts_noise2}
\end{figure}

\begin{figure}
	\centering
	\begin{subfigure}[b]{0.48\textwidth}
		\includegraphics[width=\textwidth]{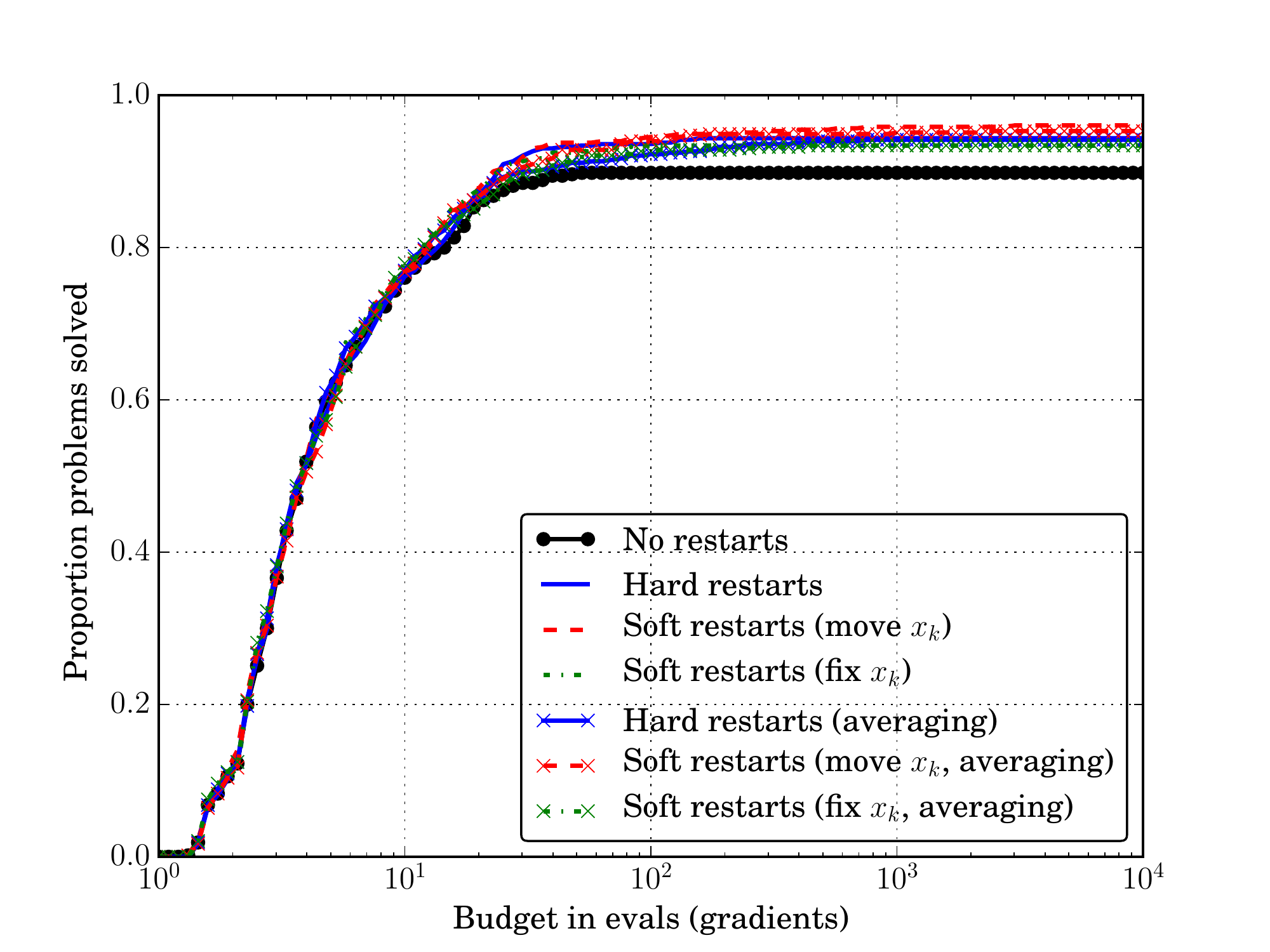}
		\caption{Multiplicative Gaussian noise}
		\label{fig_restarts_avg_noise2_ubgsn_noisyf}
	\end{subfigure}
	~
	\begin{subfigure}[b]{0.48\textwidth}
		\includegraphics[width=\textwidth]{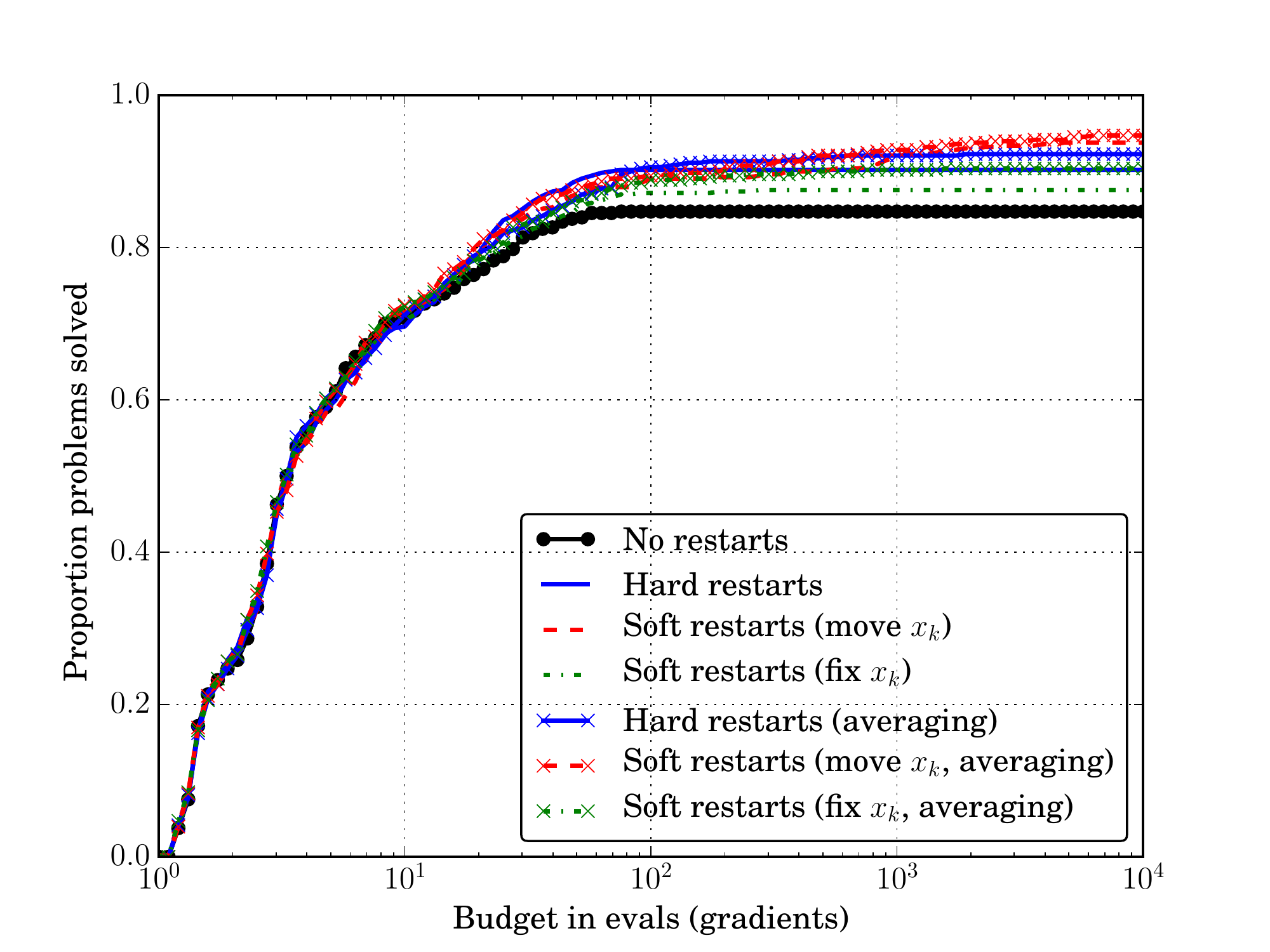}
		\caption{Additive Gaussian noise}
		\label{fig_restarts_avg_noise2_addgsn_noisyf}
	\end{subfigure}
	\\
	\begin{subfigure}[b]{0.48\textwidth}
		\includegraphics[width=\textwidth]{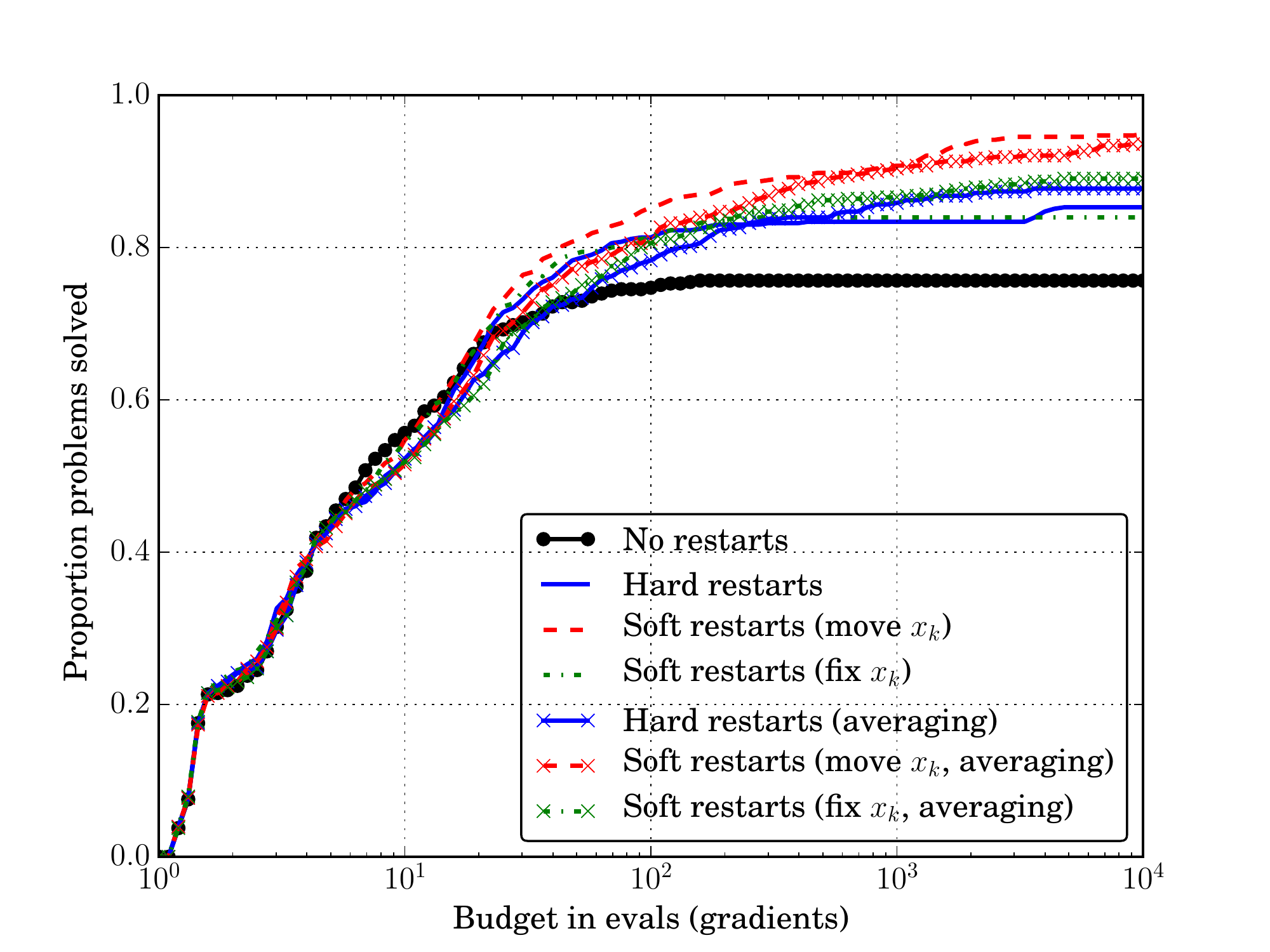}
		\caption{Additive $\chi^2$ noise}
		\label{fig_restarts_avg_noise2_addchisq_noisyf}
	\end{subfigure}
	\caption{Comparison of multiple restarts with and without sample averaging for DFO-LS (using $n+1$ interpolation points). We are using noisy objective evaluations with $\sigma=10^{-2}$, high accuracy $\tau=10^{-5}$, and an average of 10 runs for each solver. The problem collection is (MW).}
	\label{fig_restarts_avg_noise2}
\end{figure}

\section{Benchmark Comparisons of DFO-LS} \label{sec_dfols_benchmarking}

This section compares the performance of DFO-LS with other, state-of-the-art derivative-free solvers for nonlinear least-squares problems, namely, DFO-GN \cite{Cartis2017a}, and DFBOLS \cite{Zhang2010} with $2n+1$ and $(n+1)(n+2)/2$ points.
DFO-LS uses $p+1=n+1$ interpolation points and the default values for all other parameters, unless otherwise specified.
We use the computational budget, and initial and final trust region radii as in \secref{sec_testing_methodology}, with accuracy level $\tau=10^{-5}$.

\figref{fig_basic_smooth} shows results for smooth (noiseless) objective functions for both the (MW) and (CR) test sets.
Since DFO-LS uses randomized initial points, we show an average result over 10 runs.
For the (CR) set, we do not show DFBOLS with $(n+1)(n+2)/2$ points, as in most cases the initialization cost will use almost all of the available budget.
DFO-LS performs similarly to DFO-GN and DFBOLS, which is to be expected given the similarity of these algorithms.
Note that for (CR), initializing DFO-LS with only 2 evaluations yields only slightly worse performance, but gains the benefit of decreasing the objective at a very low cost (as shown in \secref{sec_growing_testing}).

Similarly, for noisy functions (from the (MW) set),
\figref{fig_basic_noise2} shows DFO-LS with and without restarts versus the same solvers as above. 
It is in this scenario that the flexibility of DFO-LS becomes evident --- its ability to adjust the default algorithm parameters in the presence of noisy evaluations allows it to solve a larger proportion of problems than both DFBOLS and DFO-GN, and this robustness is further improved, by a significant margin, by the use of multiple restarts.

\begin{figure}
	\centering
	\begin{subfigure}[b]{0.48\textwidth}
		\includegraphics[width=\textwidth]{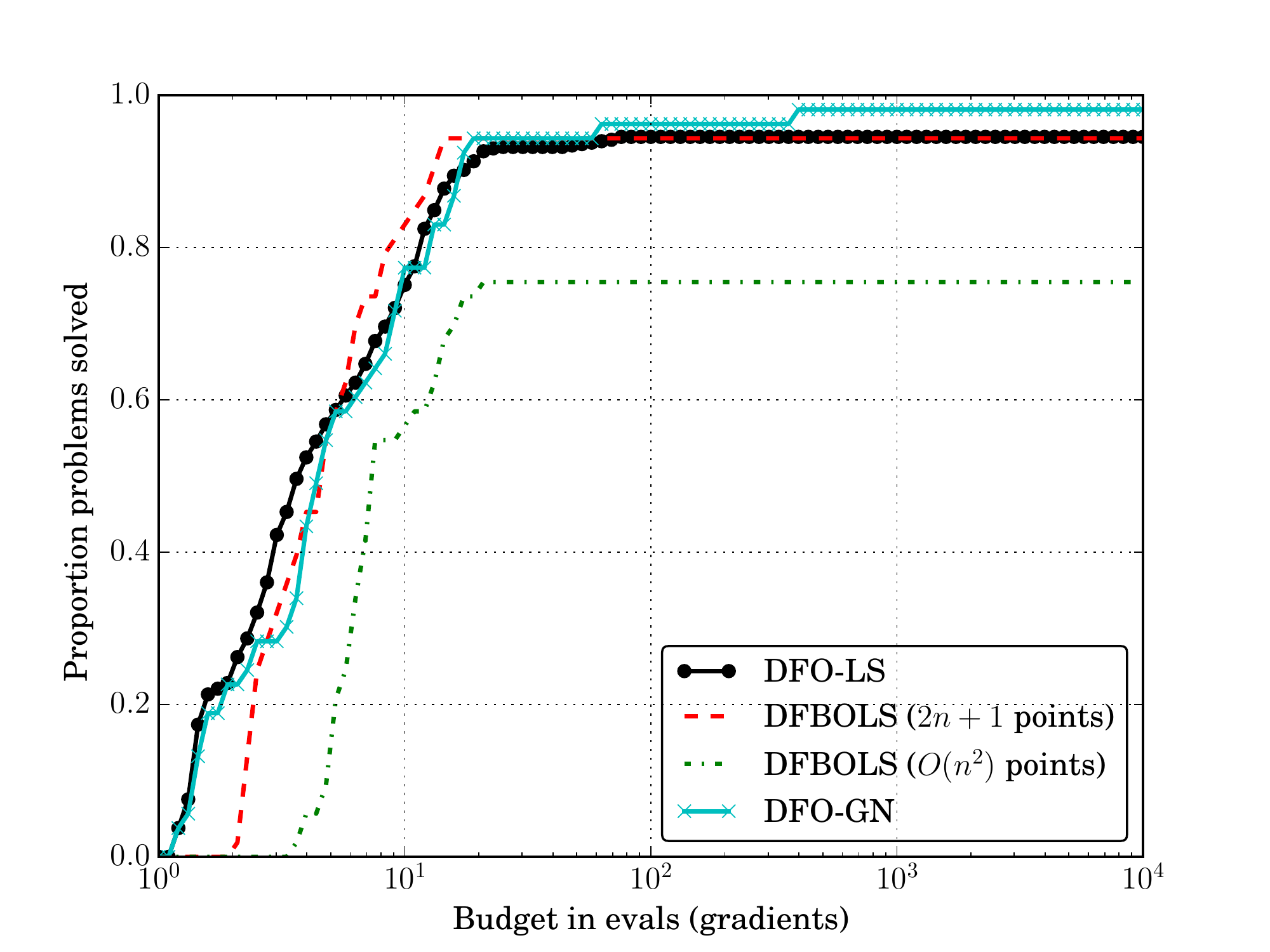}
		\caption{Problem collection (MW)}
		\label{fig_basic_smooth_data}
	\end{subfigure}
	~
	\begin{subfigure}[b]{0.48\textwidth}
		\includegraphics[width=\textwidth]{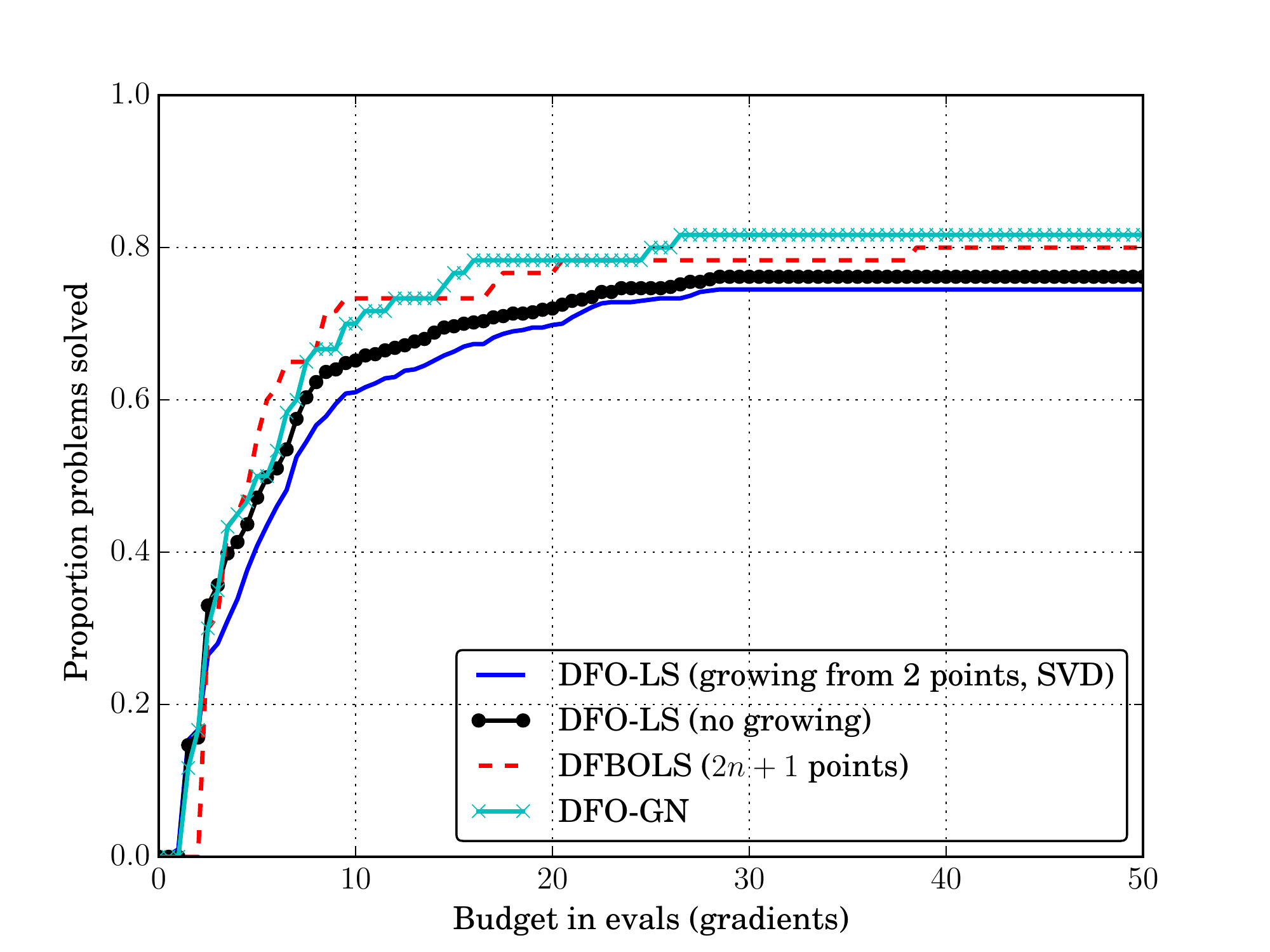}
		\caption{Problem collection (CR)}
		\label{fig_basic_cutest_smooth_data}
	\end{subfigure}
	\caption{Comparison of the basic implementation of DFO-LS (using $n+1$ interpolation points) with DFBOLS and DFO-GN for smooth objective evaluations and high accuracy $\tau=10^{-5}$. For DFBOLS, $2n+1$ and $\bigO(n^2)=(n+1)(n+2)/2$ are the number of interpolation points. For DFO-LS, results are an average of 10 runs. In (b), we show results using the full initialization cost of $n+1$ evaluations, and a reduced cost of 2 evaluations (using the SVD method).}
	\label{fig_basic_smooth}
\end{figure}

\begin{figure}
	\centering
	\begin{subfigure}[b]{0.48\textwidth}
		\includegraphics[width=\textwidth]{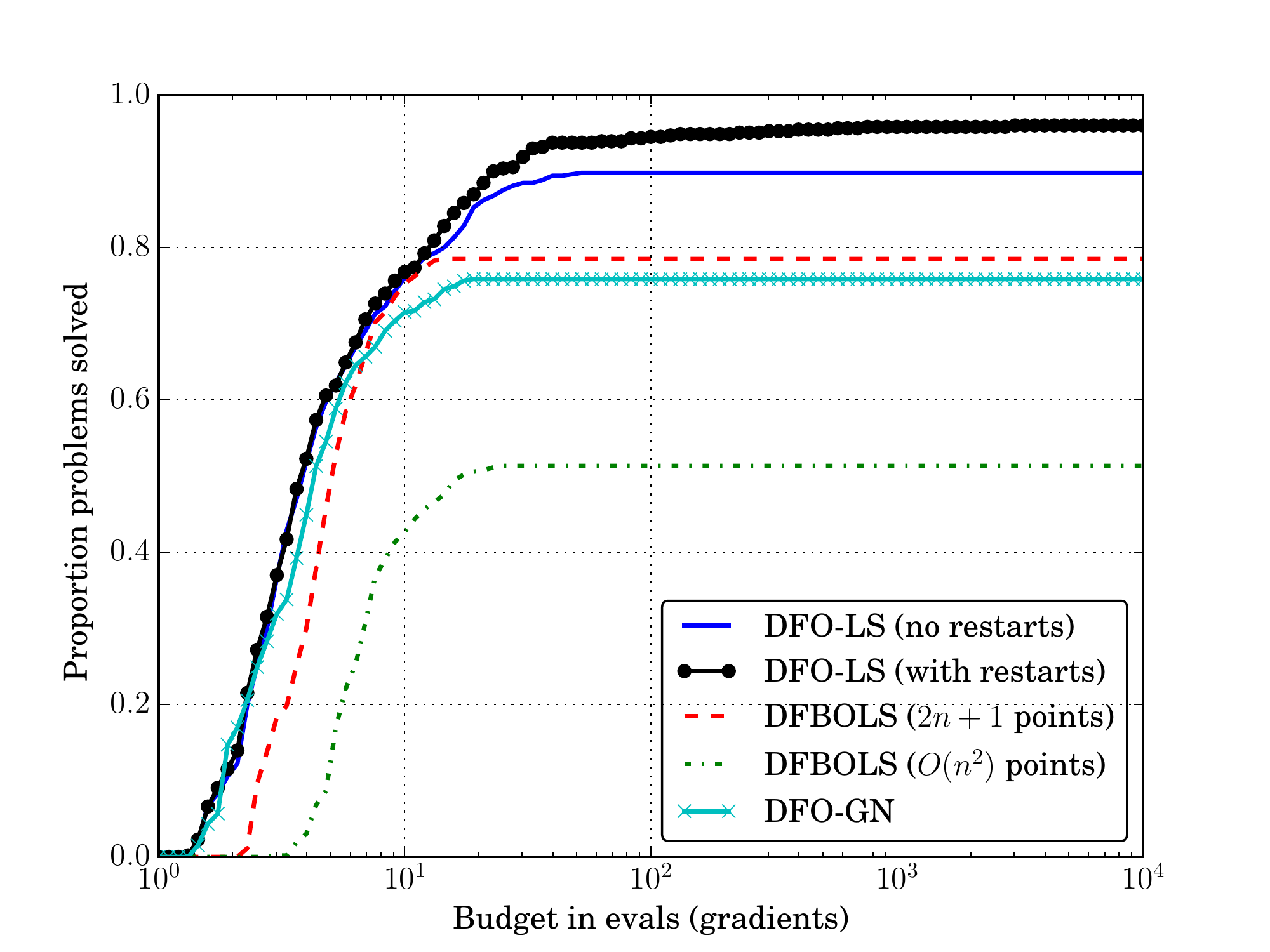}
		\caption{Multiplicative Gaussian noise}
		\label{fig_basic_noise2_ubgsn_noisyf}
	\end{subfigure}
	~
	\begin{subfigure}[b]{0.48\textwidth}
		\includegraphics[width=\textwidth]{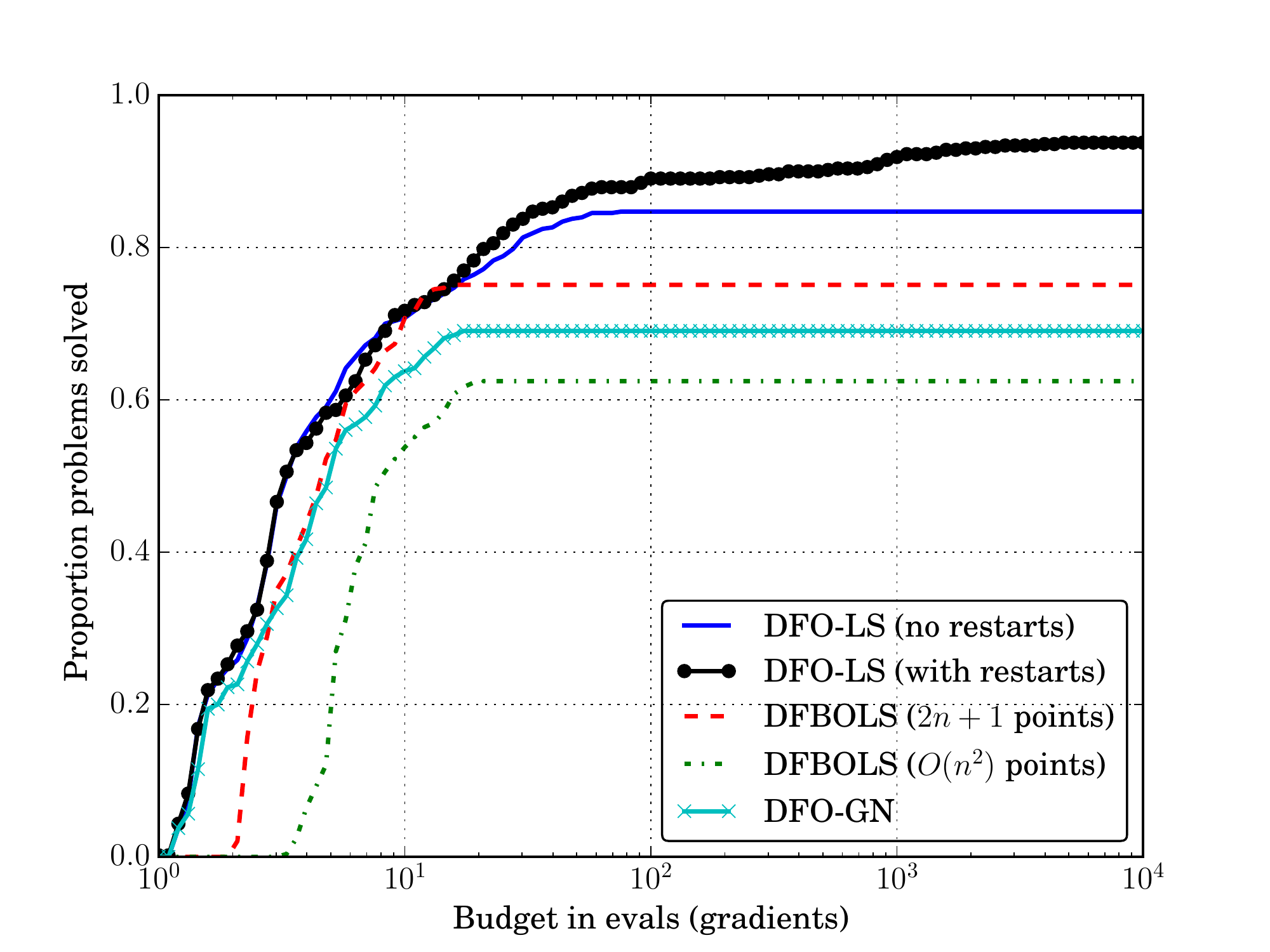}
		\caption{Additive Gaussian noise}
		\label{fig_basic_noise2_addgsn_noisyf}
	\end{subfigure}
	\\
	\begin{subfigure}[b]{0.48\textwidth}
		\includegraphics[width=\textwidth]{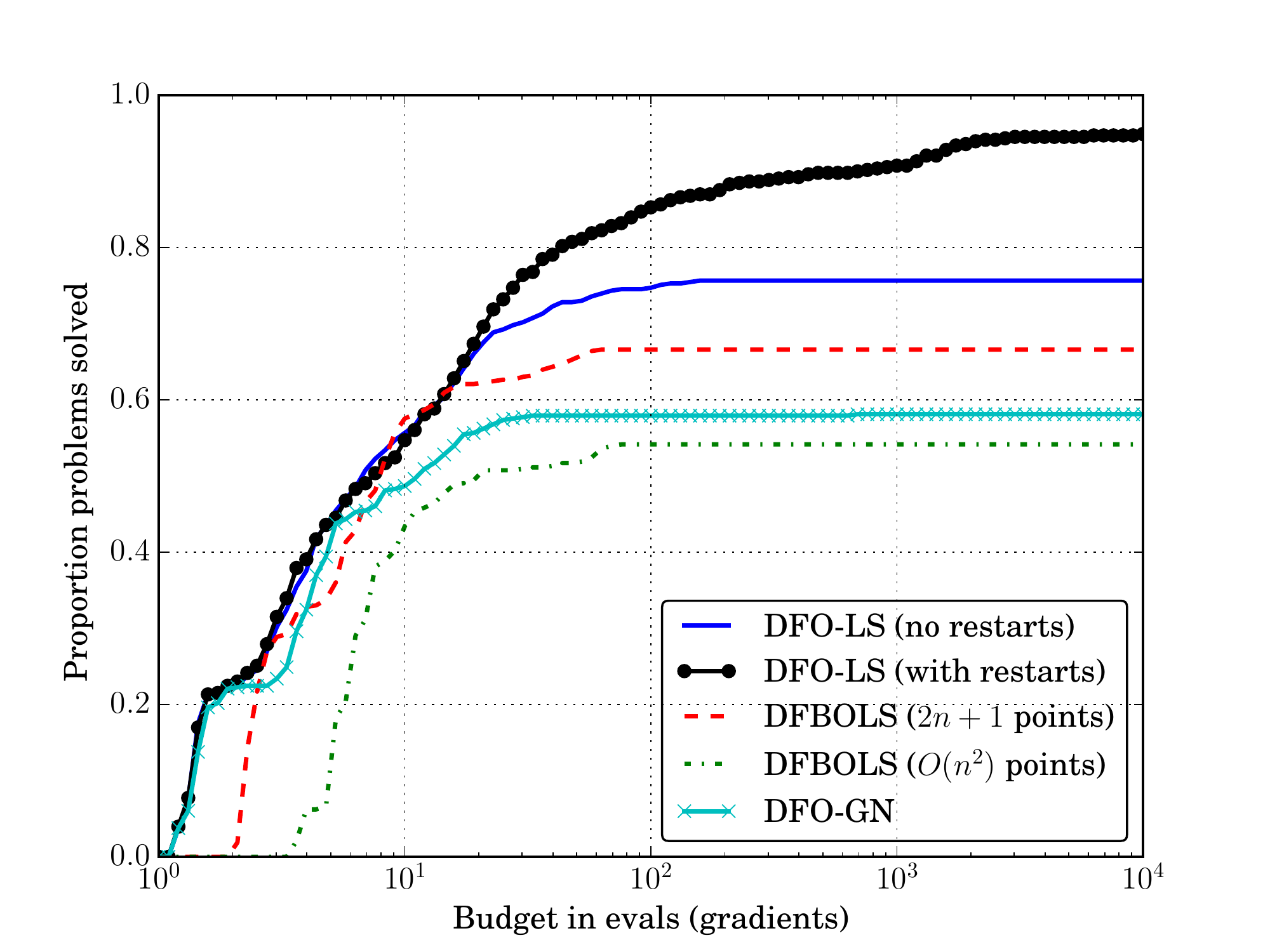}
		\caption{Additive $\chi^2$ noise}
		\label{fig_basic_noise2_addchisq_noisyf}
	\end{subfigure}
	\caption{Comparison of the basic implementation of DFO-LS (using $n+1$ interpolation points) with DFBOLS and DFO-GN for noisy objective evaluations with $\sigma=10^{-2}$ and high accuracy $\tau=10^{-5}$. For DFBOLS, $2n+1$ and $\bigO(n^2)=(n+1)(n+2)/2$ are the number of interpolation points. Results shown are an average of 10 runs for each solver. The problem collection is (MW).}
	\label{fig_basic_noise2}
\end{figure}

\paragraph{Expensive \& Noisy Problems}
Next, we illustrate that the two regimes --- `expensive' and `noisy' --- are not mutually exclusive.
In \figref{fig_cutest_noisy}, we run DFO-LS, DFBOLS and DFO-GN on the (CR) problem set with additive Gaussian noise.
The DFO-LS runs use the default settings for noisy problems (i.e.~slower trust region decrease parameters, multiple restarts). The results are very similar to the smooth case (see Figures \ref{fig_growing_smooth} and \ref{fig_basic_cutest_smooth_data}): the reduced initialization cost allows progress to be made within $n+1$ objective evaluations for some problems, at the cost of reduced small-budget performance, but achieves similar overall robustness, with performance for long budgets at high accuracy levels.

\begin{figure}
	\centering
	\begin{subfigure}[b]{0.48\textwidth}
		\includegraphics[width=\textwidth]{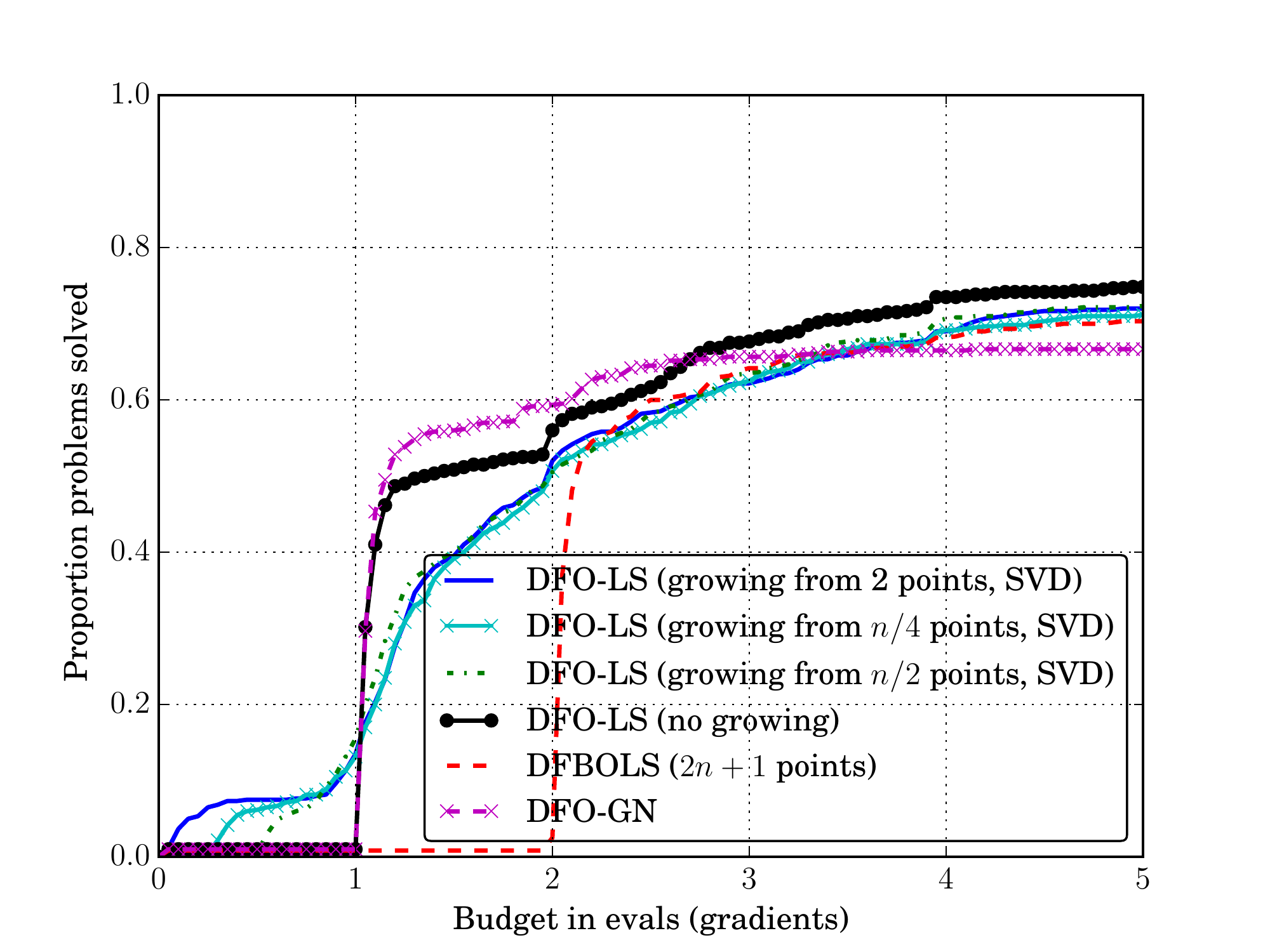}
		\caption{Short budget, $\tau=10^{-1}$}
		\label{fig_cutest_noisy_short}
	\end{subfigure}
	~
	\begin{subfigure}[b]{0.48\textwidth}
		\includegraphics[width=\textwidth]{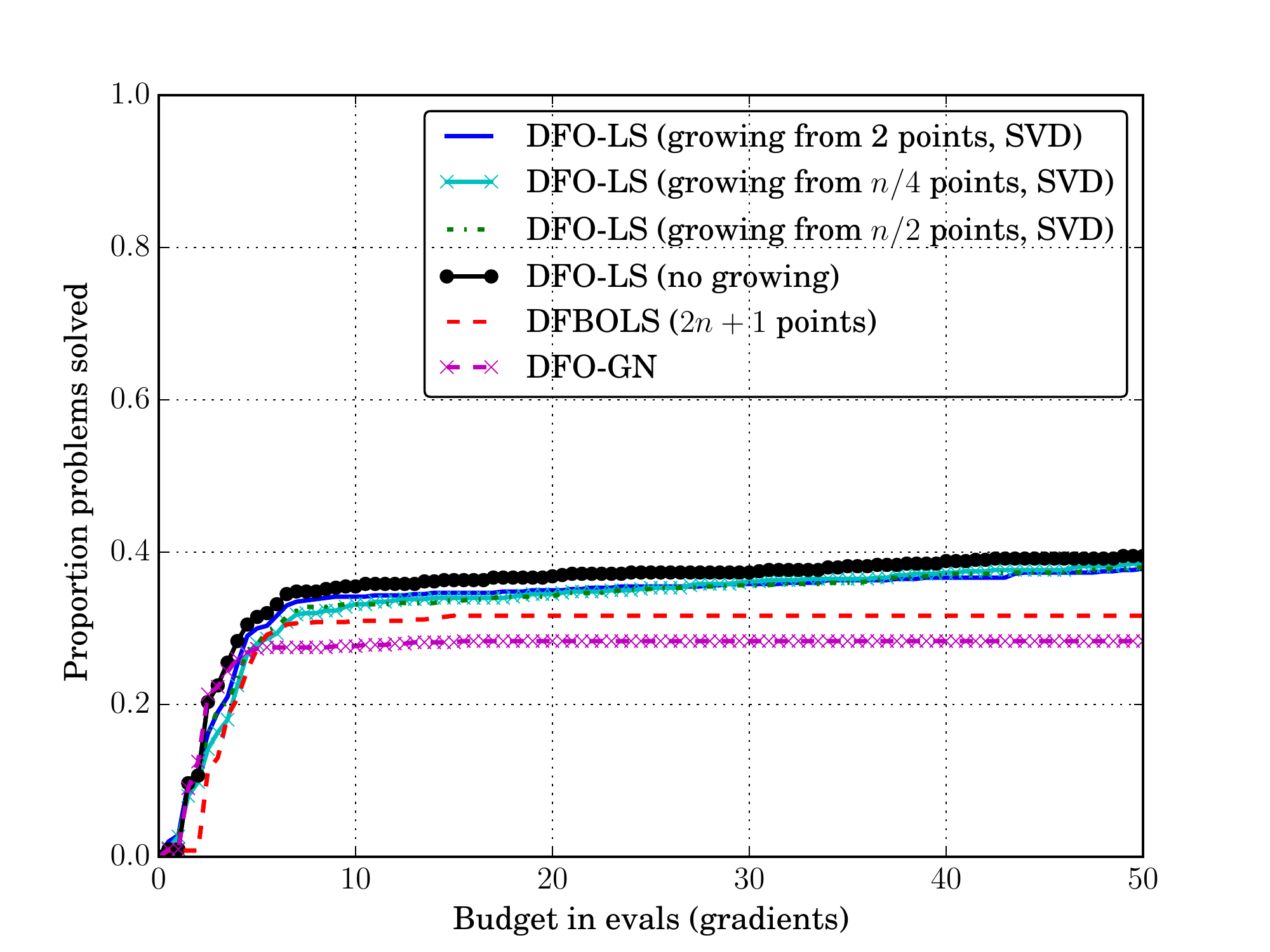}
		\caption{Long budget, $\tau=10^{-5}$}
		\label{fig_cutest_noisy_long}
	\end{subfigure}
	\caption{Comparison of the reduced initialization cost of DFO-LS (using $n+1$ interpolation points, SVD method) against using the full initial set, DFBOLS and DFO-GN, for objectives with additive Gaussian noise, $\sigma=10^{-2}$. For DFBOLS, $2n+1$ is the number of interpolation points. Results are an average of 10 runs in each case. The problem collection is (CR).}
	\label{fig_cutest_noisy}
\end{figure}

\paragraph{Multiple Restarts for Noiseless Problems}
Although the multiple restarts feature is designed for noisy problems, it can also be useful for smooth objectives.
In \figref{fig_smooth_restarts_dfols_profiles}, we show \figref{fig_basic_smooth_data}, but including results for DFO-LS with soft (moving $\bx_k$) and hard restarts\footnote{\:For noiseless problems, we do not use the autodetection of restarts feature from \secref{sec_restarts_description}.}. 
Both restart mechanisms provide a slight improvement --- for most problems, the restarts give similar performance (although using the full computational budget allowed), but in some cases they are beneficial.

The first possible benefit of multiple restarts is being able to escape local minima.
In \figref{fig_smooth_restarts_dfols_prob14}, we show the objective value $f(\bx_k)$ for one run of DFO-LS with soft restarts for problem 14 in (MW); the vertical lines show where restarts occurred.
This problem has two local minima, with $f(\bx^*)\approx 48.98$ and $f(\bx^*)=0$ \cite{More1981}.
We see that when the first restart occurs, we have found the local minimum with higher objective value --- this is when DFO-LS would usually terminate.
However, if we allow DFO-LS to perform three soft restarts, it manages to find the other local minimum (which is also the global minimum).

The other possible benefit is a faster convergence rate.
In \figref{fig_smooth_restarts_dfols_prob18}, we consider problem 18 in (MW), and we again show $f(\bx_k)$ for DFO-LS without restarts, and with hard restarts.
For this problem, DFO-LS with the default settings terminates on the `slow progress' termination criterion (see \appref{sec_general_features_appendix}; the solid circle in the plot), but we show how DFO-LS without restarts continues to make progress when this criterion is disabled.
We also show DFO-LS with hard restarts; in this case, we keep the `slow progress' termination criterion, and this triggers a restart.
We can see that eventually, the run with multiple restarts finds better objective values, and seems to be converging at a faster asymptotic rate than DFO-LS without restarts.

\begin{figure}[t]
	\centering
	\begin{subfigure}[b]{0.48\textwidth}
		\includegraphics[width=\textwidth]{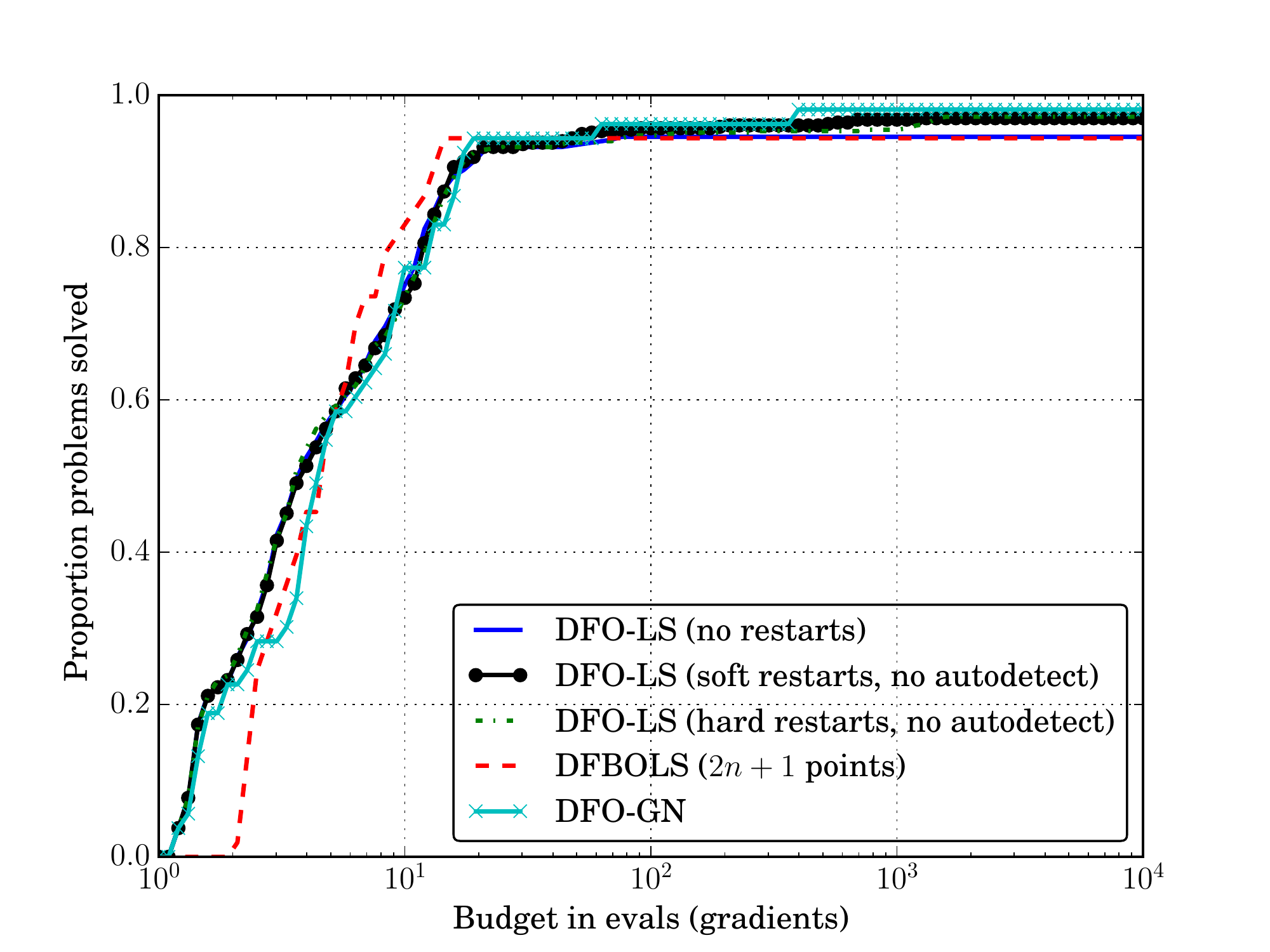}
		\caption{Data Profile, $\tau=10^{-5}$}
		\label{fig_smooth_restarts_dfols_profiles}
	\end{subfigure}
	~
	\begin{subfigure}[b]{0.48\textwidth}
		\includegraphics[width=\textwidth]{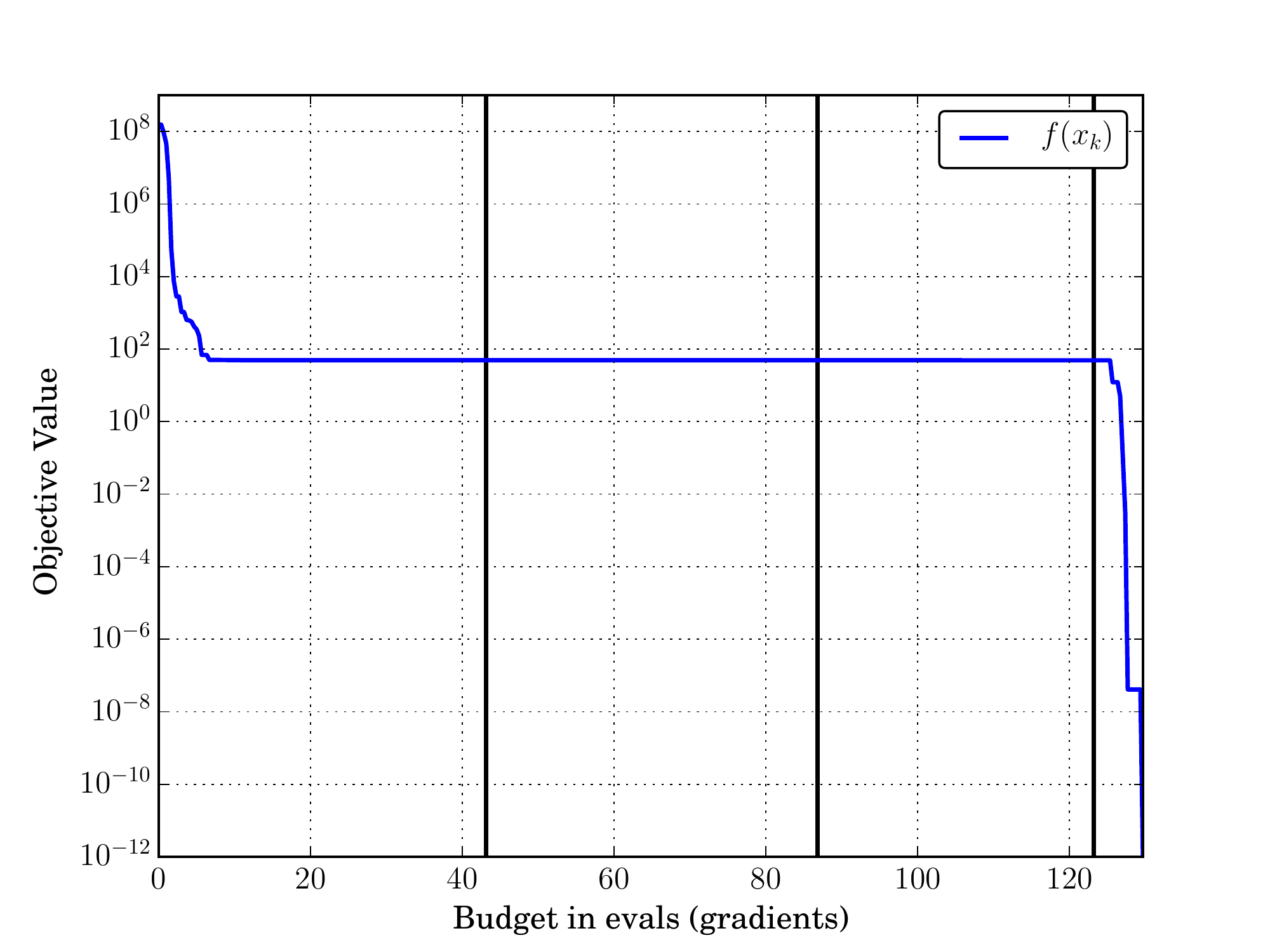}
		\caption{Objective reduction, problem 14 (soft restarts)}
		\label{fig_smooth_restarts_dfols_prob14}
	\end{subfigure}
	\\
	\begin{subfigure}[b]{0.48\textwidth}
		\includegraphics[width=\textwidth]{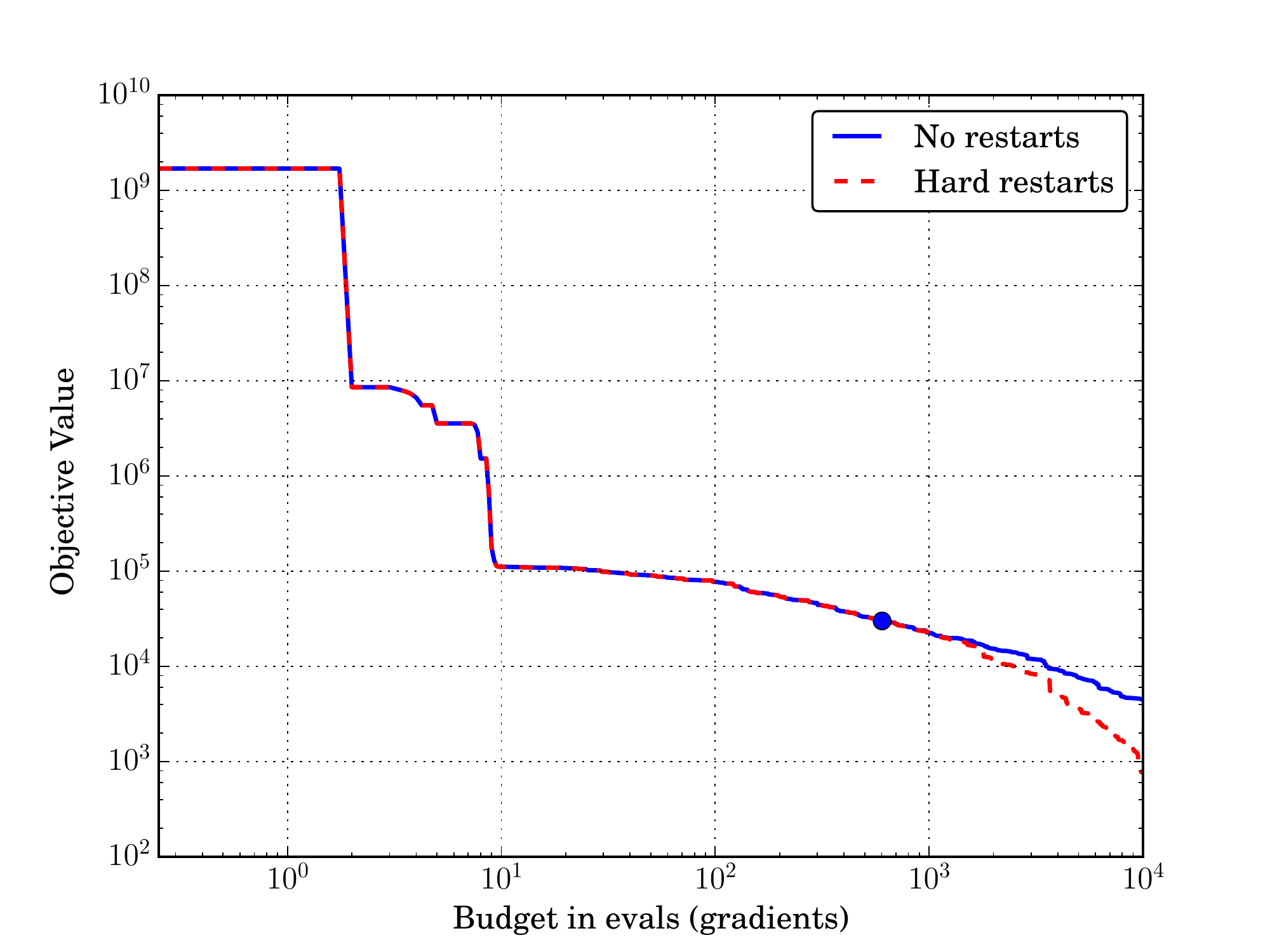}
		\caption{Objective reduction, problem 18}
		\label{fig_smooth_restarts_dfols_prob18}
	\end{subfigure}
	\caption{Illustration of the impacts of multiple restarts for noiseless problems. Figure (a) is the same as \figref{fig_basic_smooth_data}, but also showing DFO-LS with soft and hard restarts, without use of autodetection (problem collection (MW)). Figure (b) shows the objective value $f(\bx_k)$ using DFO-LS with soft restarts (moving $\bx_k$), for (MW) problem 14; the vertical lines indicate where restarts occurred. Figure (c) shows the objective value using DFO-LS with and without hard restarts, for (MW) problem 18. The dot indicates where the default `slow decrease' termination criterion is triggered; the rest of the results for the `no restarts' case are found by disabling this criterion.}
	\label{fig_smooth_restarts_dfols}
\end{figure}

\section{Py-BOBYQA: DFO for General Objective Problems} \label{sec_pybobyqa}
In this section we consider the case of general objective problems; that is,
\be \min_{\bx\in\R^n} f(\bx), \ee
for some sufficiently smooth $f$ with unknown structure.
We call our solver Py-BOBYQA, as it is a Python-based solver which is very similar to Powell's (Fortran) BOBYQA \cite{Powell2009}.

The overall algorithmic structure of (Py-)BOBYQA is the same as \algref{alg_dfols}: we construct an interpolation-based model for $f$, calculate a step to minimize this model inside a trust region, and perform one of several phases (safety, successful, model-improving, unsuccessful) depending on the outcome.
The most important difference is that the model $m_k(\bs)\approx f(\bx_k+\bs)$ is built by directly interpolating $f(\by_t)$ for $\by_t\in Y_k$ and is typically quadratic.
Specifically, for an interpolation set of size $|Y_k|\in\{n+1,\ldots,(n+1)(n+2)/2\}$, we construct
\be m_k(\bs) = c_k + \bg_k^{\top}\bs + \frac{1}{2}\bs^{\top}H_k \bs, \ee
satisfying the interpolation (not regression) conditions
\be m_k(\by_t-\bx_k) = f(\by_t), \quad \text{for all $\by_t\in Y_k$.} \label{eq_bobyqa_interp_conditions} \ee
If $|Y_k|<(n+1)(n+2)/2$, the solution to \eqref{eq_bobyqa_interp_conditions} is non-unique; following \cite{Powell2009} we use the remaining degrees of freedom by choosing $H_k=0$ if $|Y_k|=n+1$, and solving
\be \min_{c_k,\bg_k,H_k} \|H_k-H_{k-1}\|_F^2 \quad \text{subject to \eqref{eq_bobyqa_interp_conditions}}, \label{eq_bobyqa_interp_problem} \ee
otherwise.
The value of $|Y_k|$ is a user-specified input, which defaults to $2n+1$ for smooth problems and $(n+1)(n+2)/2$ for noisy problems.

\paragraph{Simplifications from original BOBYQA}
For the purposes of a simplified code, and to be more closely aligned with DFO-LS, we simplify the model construction process in Py-BOBYQA as compared to its original implementation in \cite{Powell2009}.
Specifically, in \cite{Powell2004a}, it was noted that changing a single interpolation point yielded a low-rank update to the linear system corresponding to \eqref{eq_bobyqa_interp_problem}.
This, together with a well-chosen system for building $Y_0$, meant that the linear system for \eqref{eq_bobyqa_interp_problem} was never solved directly; instead, a factorization of the corresponding matrix inverse was maintained at all iterations, and updated using the Sherman-Morrison-Woodbury formula.
By contrast, in Py-BOBYQA, as in DFO-LS, we use random directions to build $Y_0$, and construct the model by solving the linear system resulting from \eqref{eq_bobyqa_interp_problem} at every iteration.

\paragraph{Improvements from original BOBYQA}
The goal of implementing Py-BOBYQA was to endow it with some of the key features from DFO-LS in order to improve its robustness to noise.
Given the extra complexity of managing quadratic rather than linear models, we transferred the features from DFO-LS which did not require a large redesign of the model construction routines.
Specifically, Py-BOBYQA contains the following new features:
\begin{itemize}
	\item The user can specify $|Y_k|=n+1$, compared to $|Y_k|\geq n+2$ as required by BOBYQA;
	\item Larger range of termination conditions, as per \appref{sec_general_features_appendix}. The changes are that the `small objective value' threshold is just $f(\bx_k) \leq \epsilon_{abs}$, as we no longer have $f\geq 0$ guaranteed, and for the same reason the slow decrease condition \eqref{eq_slow_termination_defn} in Py-BOBYQA uses $f(\bx_{k_{(i-K)}})-f(\bx_{k_i})$ rather than log-decrease;
	\item Flexible choice of algorithm parameters, including setting different default values for noisy problems, as per \appref{sec_general_features_appendix};
	\item Sample averaging using \eqref{eq_nsamples}, as per \secref{sec_restarts_description}; and
	\item Multiple restarts as per \secref{sec_restarts_description} (both soft and hard restarts). 
	However, the automatic detection of restarts uses linear fits for both $\{(k, \log\|\bg_k-\bg_{k-1}\|)\}$ and $\{(k, \log\|H_k-H_{k-1}\|_F)\}$ instead of $\{(k, \log\|J_k-J_{k-1}\|_F)\}$ in DFO-LS.
\end{itemize}

\paragraph{Numerical Results for Smooth Problems}
 \figref{fig_bobyqa_basic_smooth} compares the basic implementation of Py-BOBYQA v1.0.1 (no sample averaging or restarts) with the original BOBYQA \cite{Powell2009} and NOWPAC \cite{Augustin2014} for smooth problems. 
We use the (MW) problem set, and a third collection of test problems: 
\begin{description}
	\item[\rm \textit{(CFMR)}] The set of 60 nonlinear least-squares problems (CR), and the 30 general-objective problems from CUTEst listed in \appref{sec_genobj_problems}. These extra problems are also medium-sized, with $50 \leq n \leq 110$.
\end{description}
We also use the same budget and trust region radii settings as in \secref{sec_dfols_benchmarking}
(as described in \secref{sec_test_problems}), and our budget is $50(n+1)$ evaluations for (CFMR), like for the (CR) test problems.

For (Py-)BOBYQA applied to (MW), we show results for the default choice $|Y_k|=2n+1$, as well as the maximum value $|Y_k|=(n+1)(n+2)/2$, which is Py-BOBYQA's default choice for noisy problems.
We do not show the $|Y_k|=(n+1)(n+2)/2$ results for (CFMR), because the small budget and high dimension means that almost all of the budget would be used by the initialization phase.
We see that Py-BOBYQA has comparable performance with BOBYQA and NOWPAC for smooth problems, which we expect given the similarity of the algorithms.
Due to the size of the (CFMR) problems, we allowed Py-BOBYQA and NOWPAC to run for a maximum of 12 hours per problem.

\paragraph{Numerical Results for Noisy Problems} In \figref{fig_bobyqa_basic_noise2}, we compare Py-BOBYQA with BOBYQA, STORM for unbiased noise \cite{Chen2016}\footnote{\:As mentioned in the introduction, there are several variants of STORM proposed in \cite{Chen2016}. We chose this version because it showed better performance than other variants.} and SNOWPAC \cite{Augustin2017}.
Here, in line with the rest of the paper, we show results for the (MW) set only, and use Py-BOBYQA's noise default of $|Y_k|=(n+1)(n+2)/2$ for both Py-BOBYQA and BOBYQA.
As the slowest solver to run, we allowed SNOWPAC to run for a maximum of 12 hours per problem.
Since SNOWPAC uses points from the full history of observations of the objective to construct a Gaussian Process surrogate model, its performance can slow down rapidly as the computational budget is increased; as a result, we only update the surrogate model every $5n$ iterations.
Similar to our results for DFO-LS, we see that using multiple restarts gives a substantial improvement in the robustness of Py-BOBYQA, and it performs substantially better than BOBYQA.
In our experiments, Py-BOBYQA can solve more problems than STORM within the computational budget, and can solve many problems much more efficiently.
We note that STORM relies on constructing models which are entirely independent at each iteration, so it takes many more evaluations to begin seeing the desired objective reductions.
Compared to SNOWPAC, which uses both objective values and noise standard errors from each evaluation, Py-BOBYQA performs either comparably or better, with the difference most noticeable for multiplicative noise.
The multiple restarts approach in Py-BOBYQA has the advantage of not requiring extra user input, and being cheap to implement compared to constructing a surrogate model.

\begin{figure}
	\centering
	\begin{subfigure}[b]{0.48\textwidth}
		\includegraphics[width=\textwidth]{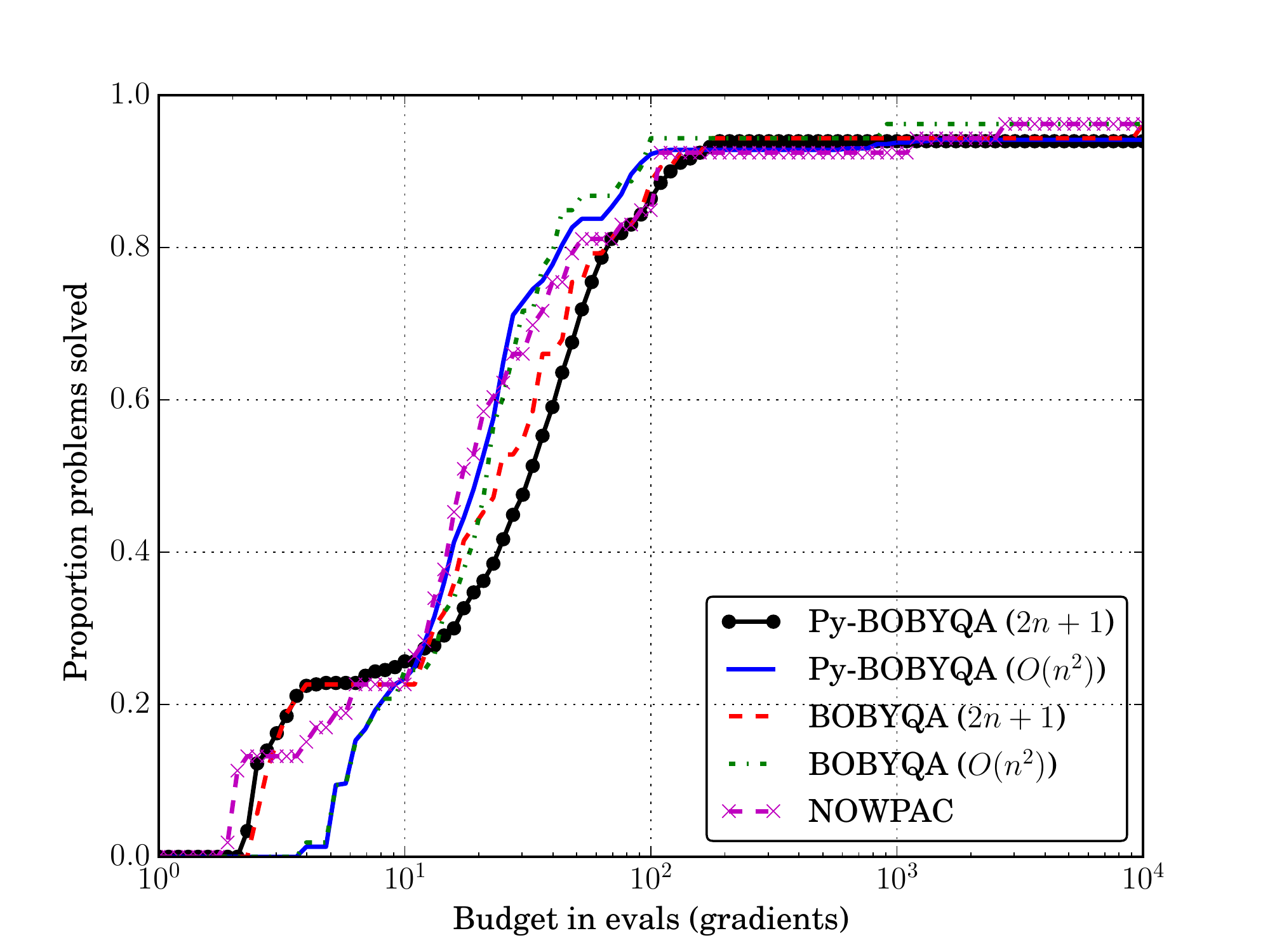}
		\caption{Problem collection (MW)}
		\label{fig_bobyqa_basic_smooth_mw}
	\end{subfigure}
	~
	\begin{subfigure}[b]{0.48\textwidth}
		\includegraphics[width=\textwidth]{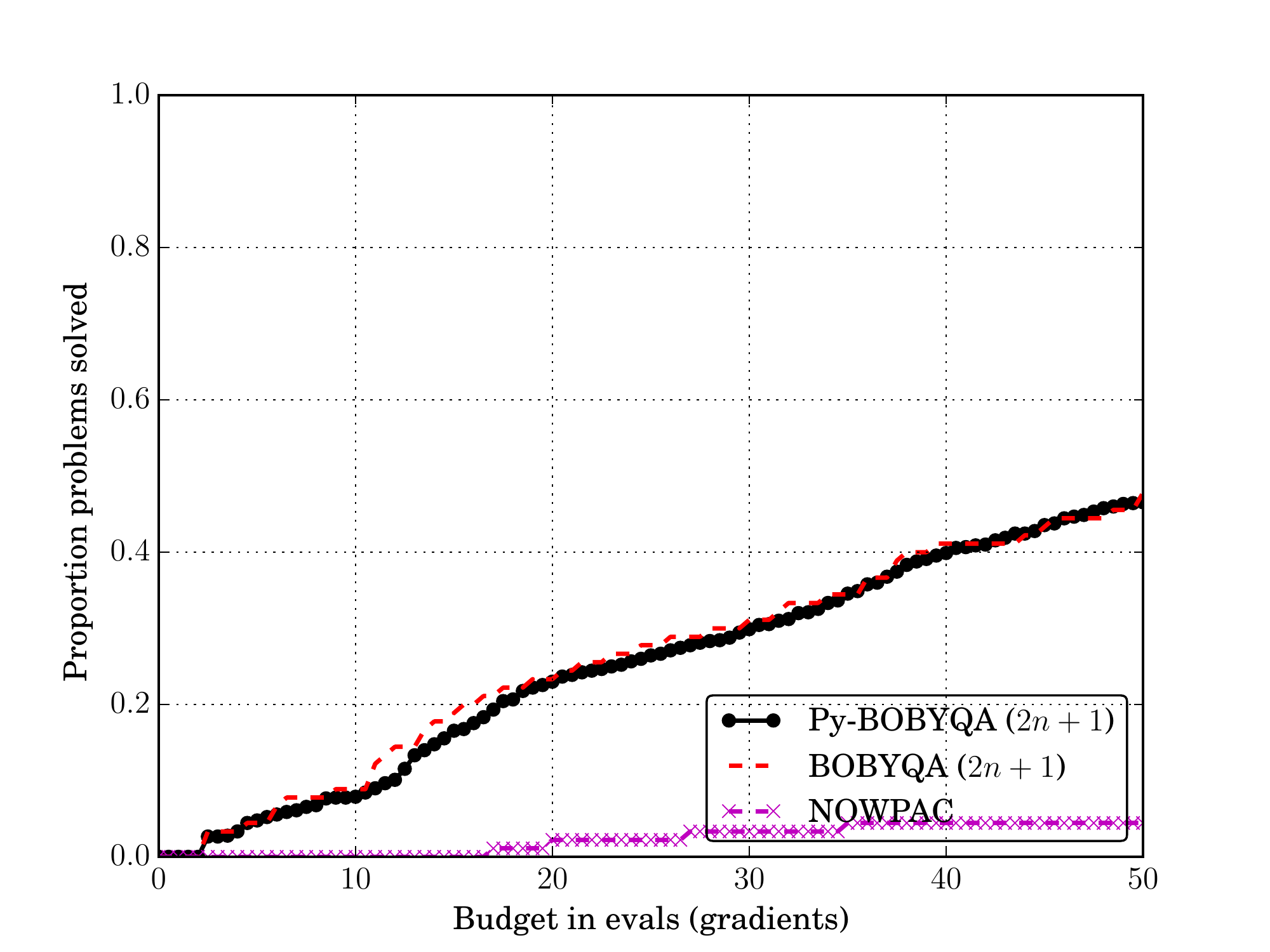}
		\caption{Problem collection (CFMR)}
		\label{fig_bobyqa_basic_smooth_cutest}
	\end{subfigure}
	\caption{Comparison of the basic implementation of Py-BOBYQA with the original Fortran BOBYQA and NOWPAC for smooth objective evaluations and high accuracy $\tau=10^{-5}$. For (Py-)BOBYQA, $2n+1$ and $\bigO(n^2)=(n+1)(n+2)/2$ are the number of interpolation points. For Py-BOBYQA, results are an average of 10 runs.}
	\label{fig_bobyqa_basic_smooth}
\end{figure}

\begin{figure}
	\centering
	\begin{subfigure}[b]{0.48\textwidth}
		\includegraphics[width=\textwidth]{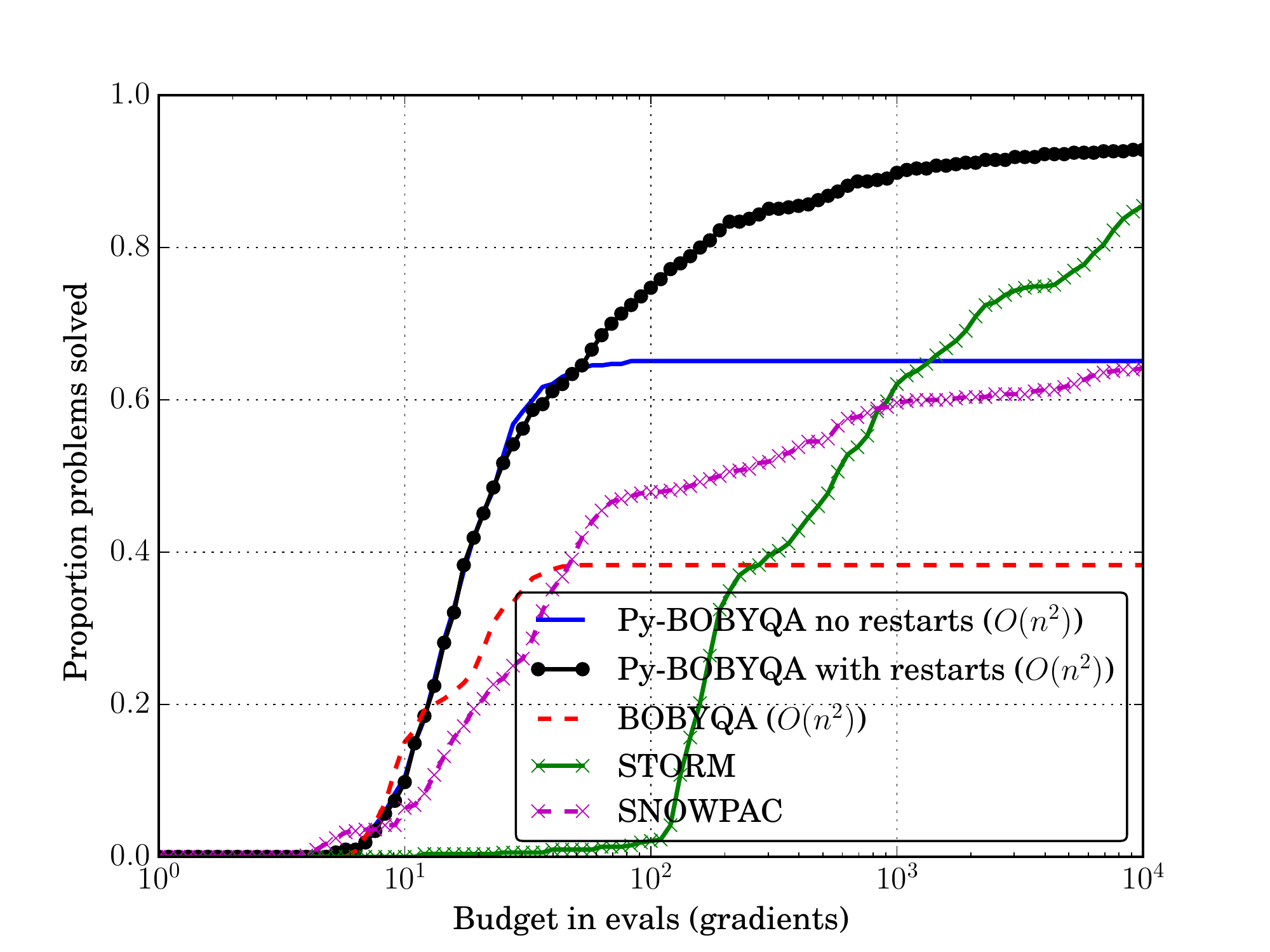}
		\caption{Multiplicative Gaussian noise}
		\label{fig_bobyqa_basic_noise2_ubgsn_noisyf}
	\end{subfigure}
	~
	\begin{subfigure}[b]{0.48\textwidth}
		\includegraphics[width=\textwidth]{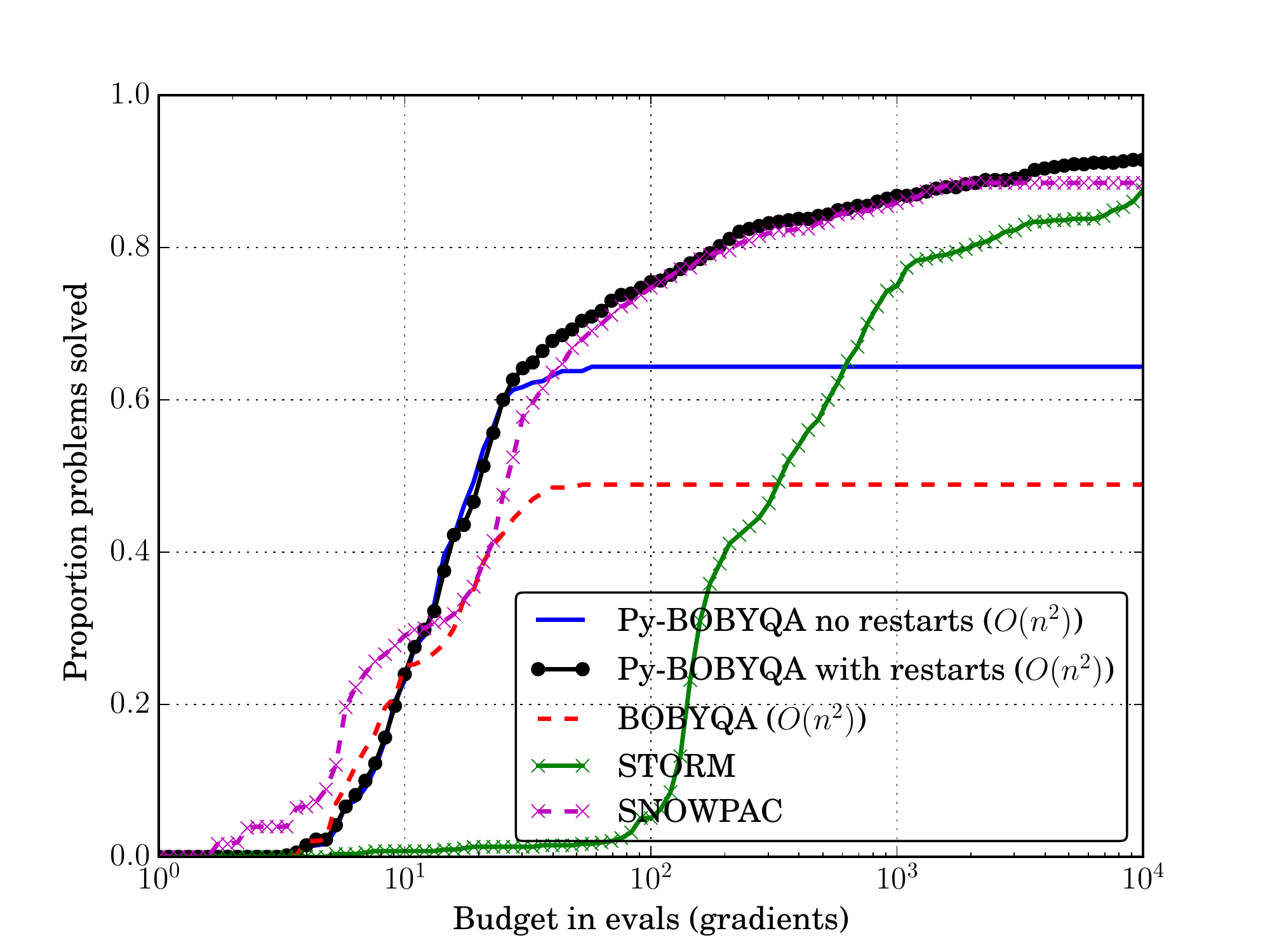}
		\caption{Additive Gaussian noise}
		\label{fig_bobyqa_basic_noise2_addgsn_noisyf}
	\end{subfigure}
	\\
	\begin{subfigure}[b]{0.48\textwidth}
		\includegraphics[width=\textwidth]{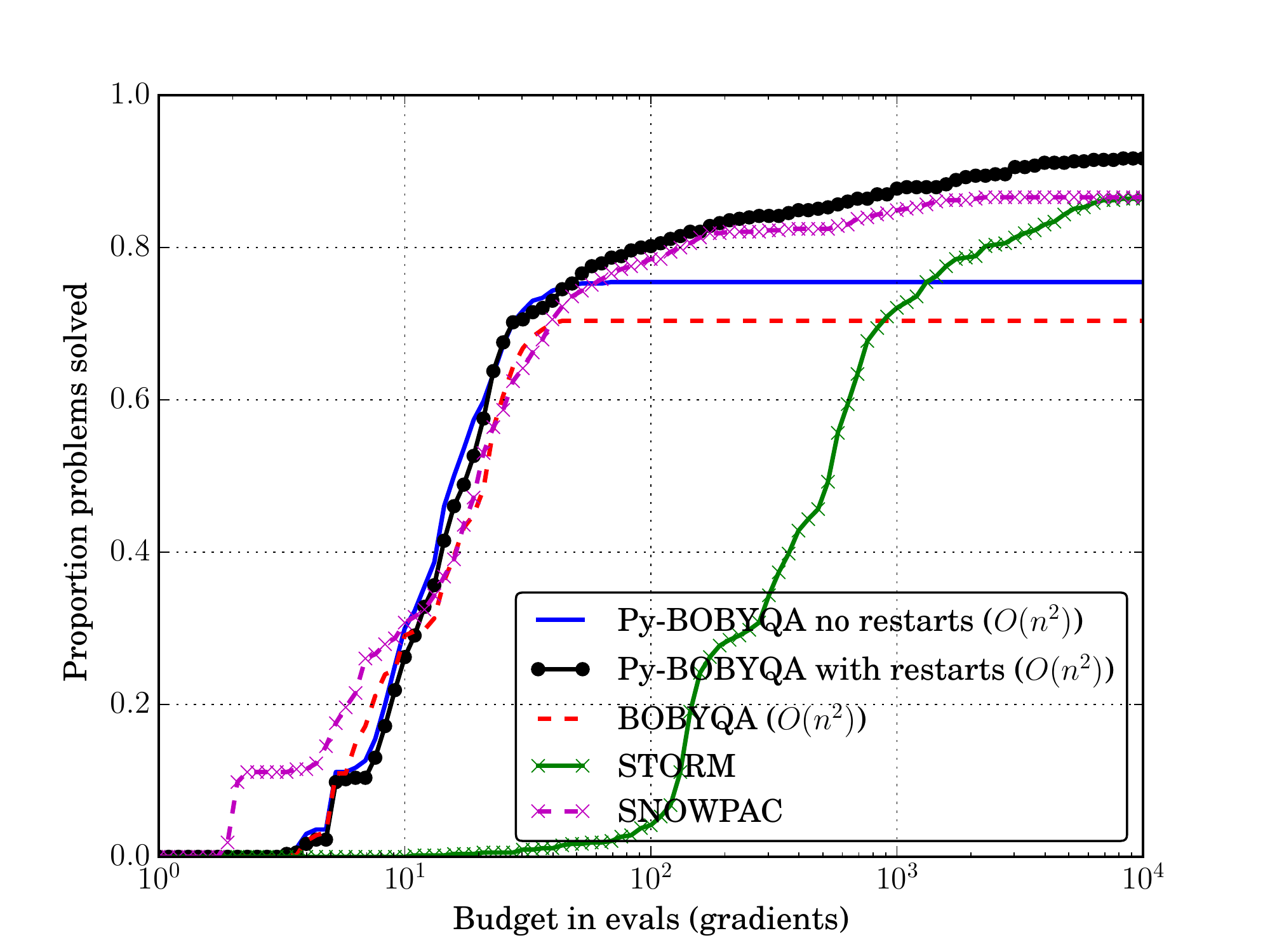}
		\caption{Additive $\chi^2$ noise}
		\label{fig_bobyqa_basic_noise2_addchisq_noisyf}
	\end{subfigure}
	\caption{Comparison of the basic implementation of Py-BOBYQA with the original Fortran BOBYQA, STORM and SNOWPAC for noisy objective evaluations with $\sigma=10^{-2}$ and high accuracy $\tau=10^{-5}$. For (Py-)BOBYQA, $\bigO(n^2)=(n+1)(n+2)/2$ is the number of interpolation points. Results shown are an average of 10 runs for each solver. The problem collection is (MW).}
	\label{fig_bobyqa_basic_noise2}
\end{figure}

To illustrate the relative cost of multiple restarts compared to building surrogate models, \figref{fig_snowpac_timings} shows the runtime\footnote{\:CPU time, measured on a Lenovo ThinkCentre M900 (with one 64-bit Intel i5 processor, 8GB of RAM).} for Py-BOBYQA and SNOWPAC for two different noisy problems from (MW), using the large budget of $10^4 (n+1)$ objective evaluations and additive Gaussian noise.
For each problem, we imposed a timeout on each solver after 12 hours, and mark when each solver achieved the particular objective reduction $\tau_{crit}(p)$; see \eqref{eq_tau_modification}.
For both problems, the runtime of Py-BOBYQA grows linearly with the number of objective evaluations, after the initial setup cost of $\bigO(n^2)$ evaluations.
However SNOWPAC's runtime starts to grow much more quickly for large budgets.
In SNOWPAC, the number of points used to build the surrogate model --- which drives the cost of surrogate model construction --- depends on the number of evaluated points in the entire run which are sufficiently close to $\bx_k$.
In many cases, this means the rapid increase in runtime occurs in the asymptotic regime, when a good solution has already been found (i.e.~accuracy $\tau_{crit}(p)$ has been achieved), and SNOWPAC is trying to improve the quality of the solution using a more accurate surrogate.
This occurs in problem 1, for instance, where $\tau_{crit}(p)=10^{-2}$, and SNOWPAC achieves this accuracy well before the runtime starts to grow quickly.
However, problem 53 is an example where the increase in runtime comes before this high accuracy regime: we have $\tau_{crit}(p)=10^{-13}$, and SNOWPAC terminates (from the timeout) without achieving accuracy $10^{-4}$.
By comparison, on this problem, Py-BOBYQA terminates on maximum budget after reaching the much higher accuracy level $\tau=10^{-11}$ before terminating (on budget).
Overall, the use of a surrogate model is beneficial for achieving robustness to noise, but may result in reduced performance in order to realise this benefit.


\begin{figure}
	\centering
	\begin{subfigure}[b]{0.48\textwidth}
		\includegraphics[width=\textwidth]{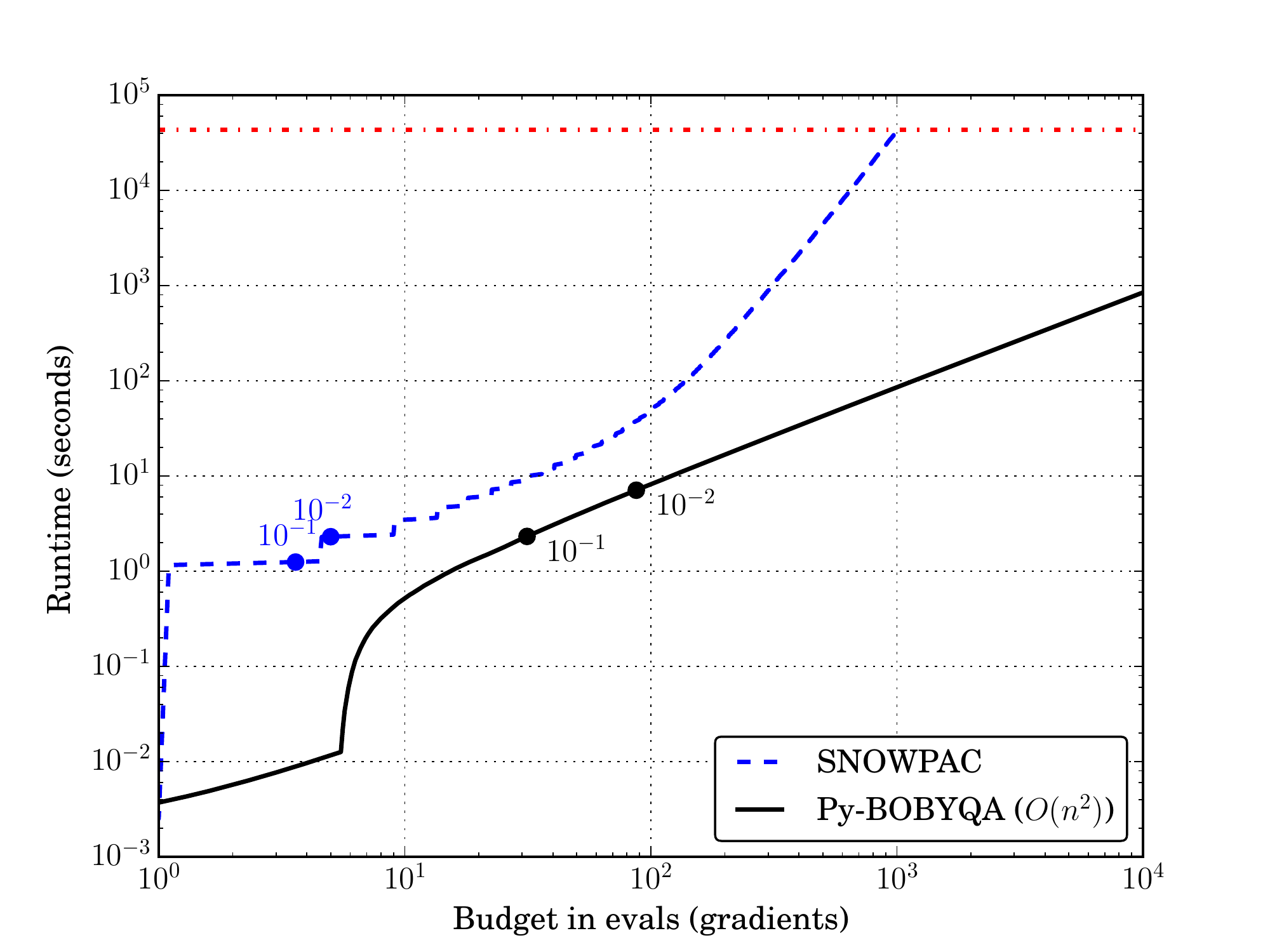}
		\caption{Problem 1, $\tau_{crit}(p)=10^{-2}$}
	\end{subfigure}
	~
	\begin{subfigure}[b]{0.48\textwidth}
		\includegraphics[width=\textwidth]{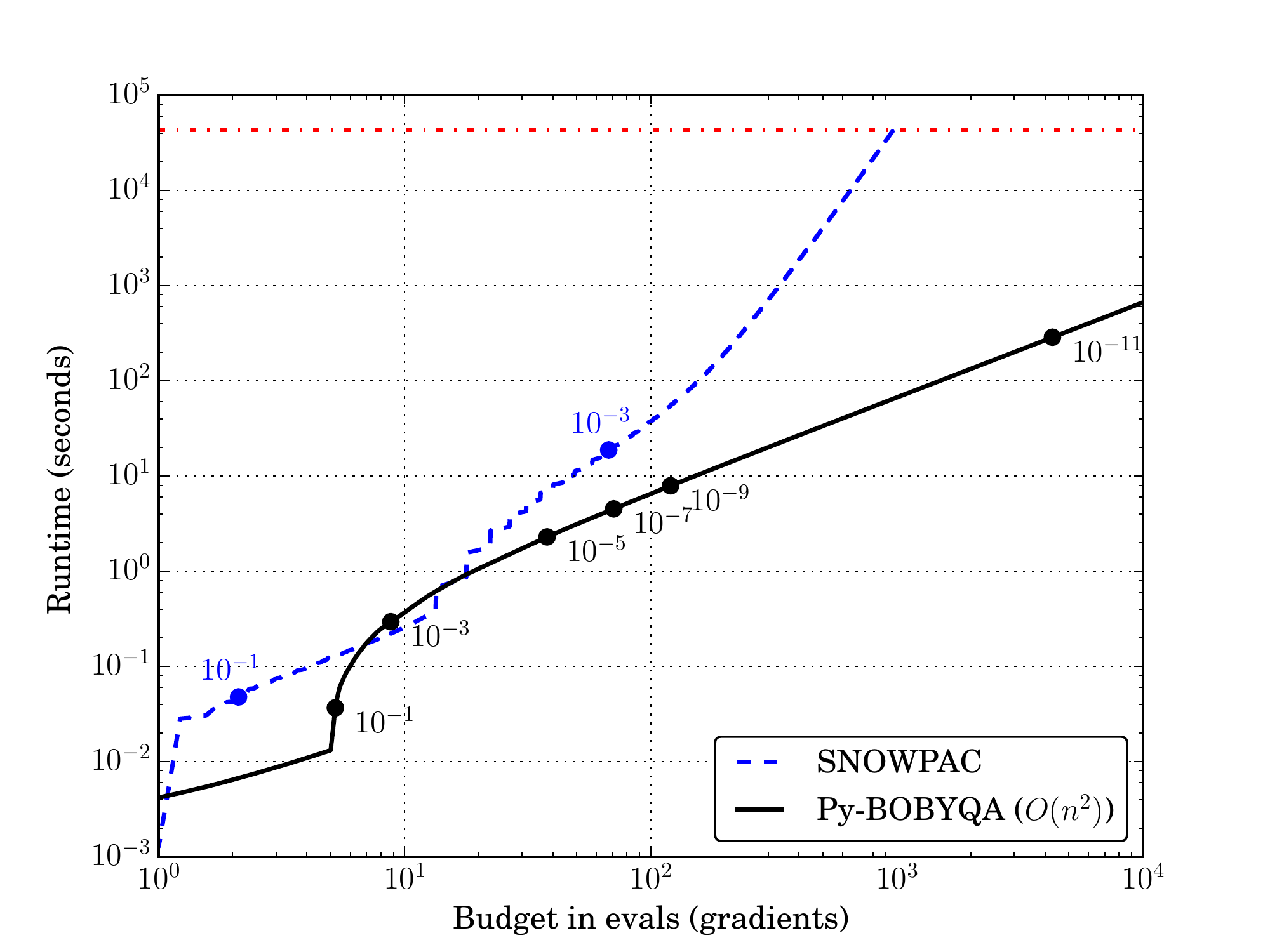}
		\caption{Problem 53, $\tau_{crit}(p)=10^{-13}$}
	\end{subfigure}
	\caption{Comparison of average runtimes --- up to a maximum of 12 hours (horizontal dot-dash line) --- for Py-BOBYQA (with $(n+1)(n+2)/2$ interpolation points and multiple restarts) and SNOWPAC, for two problems from (MW). The marked points are average budget/runtime when each solver achieved the labelled objective reduction $\tau$. Both problems had additive Gaussian noise with $\sigma=10^{-2}$. Results shown are an average of 10 runs for each solver.}
	\label{fig_snowpac_timings}
\end{figure}

\paragraph{Multiple Restarts for Noiseless Problems}
Similar to DFO-LS (see \secref{sec_dfols_benchmarking}), we conclude by illustrating that there may also be some benefit in using multiple restarts when running Py-BOBYQA on smooth problems\footnote{\:Unlike \secref{sec_dfols_benchmarking}, we do not consider reduced initialization cost for noisy problems, as Py-BOBYQA does not have this feature.}.
As before, since the first run of Py-BOBYQA with restarts is the same as the full solver run without restarts, there is no performance loss from using multiple restarts (although more of the computational budget is used).
In \figref{fig_smooth_restarts_pybobyqa_profiles}, we compare Py-BOBYQA without restarts against soft (moving $\bx_k$) and hard restarts for the (MW) collection.
As expected, at this accuracy level, multiple restarts either gives the same or slightly better robustness than no restarts --- the improvement is larger when using $(n+1)(n+2)/2$ interpolation points.

However, as for DFO-LS, we find that multiple restarts may help Py-BOBYQA to escape local minima.
In \figref{fig_smooth_restarts_pybobyqa_prob14}, we show the objective value $f(\bx_k)$ for one run of Py-BOBYQA with soft restarts and $2n+1$ interpolation points for problem 14 in (MW) --- this is the same as \figref{fig_smooth_restarts_dfols_prob14} for DFO-LS.
As before, we see that the first run of Py-BOBYQA finds the local minimum $f(\bx^*)\approx 48.98$, but after two restarts, it manages to escape and find the global minimum $f(\bx^*)=0$.

\begin{figure}
	\centering
	\begin{subfigure}[b]{0.48\textwidth}
		\includegraphics[width=\textwidth]{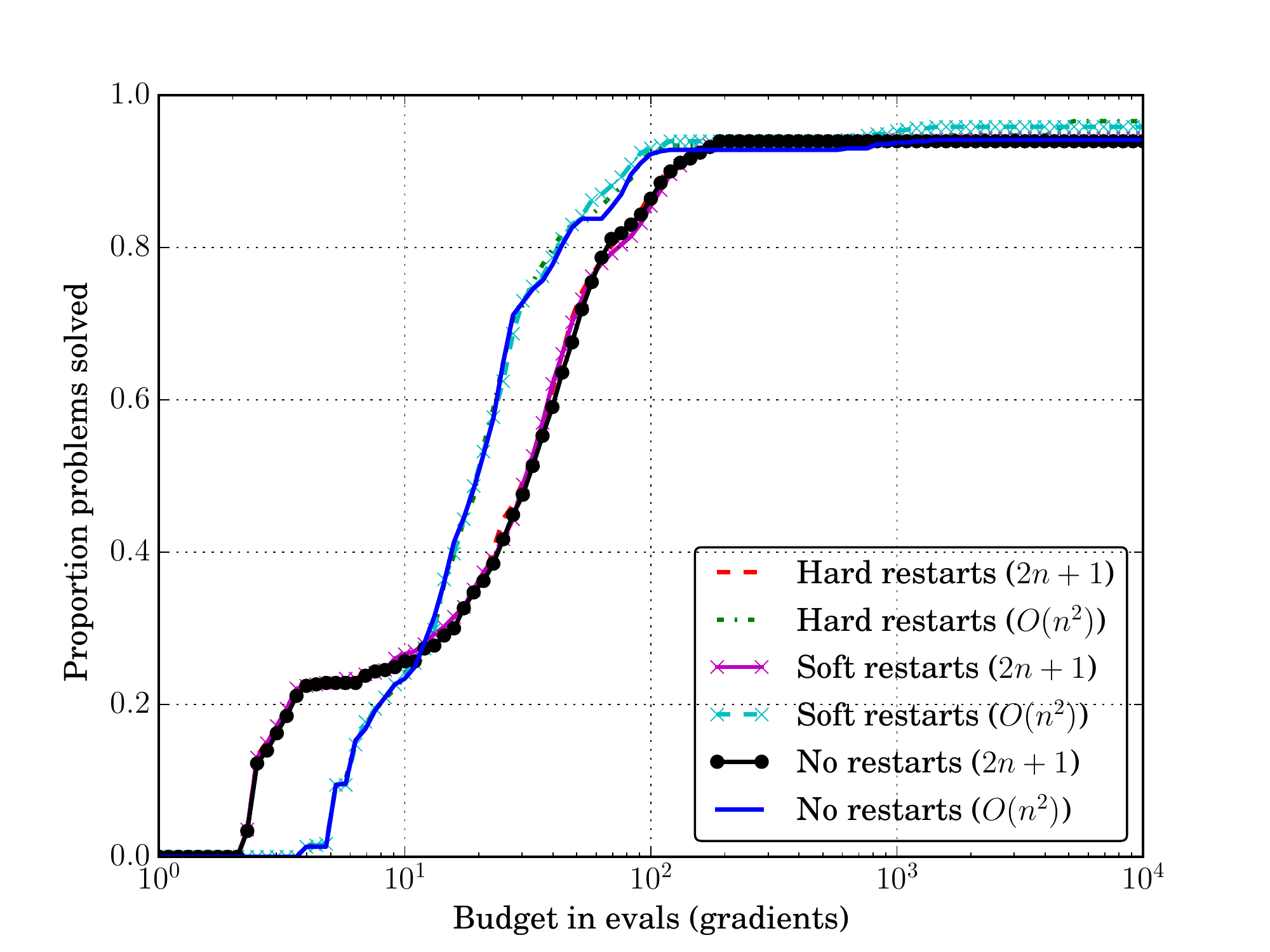}
		\caption{Data Profile, $\tau=10^{-5}$}
		\label{fig_smooth_restarts_pybobyqa_profiles}
	\end{subfigure}
	~
	\begin{subfigure}[b]{0.48\textwidth}
		\includegraphics[width=\textwidth]{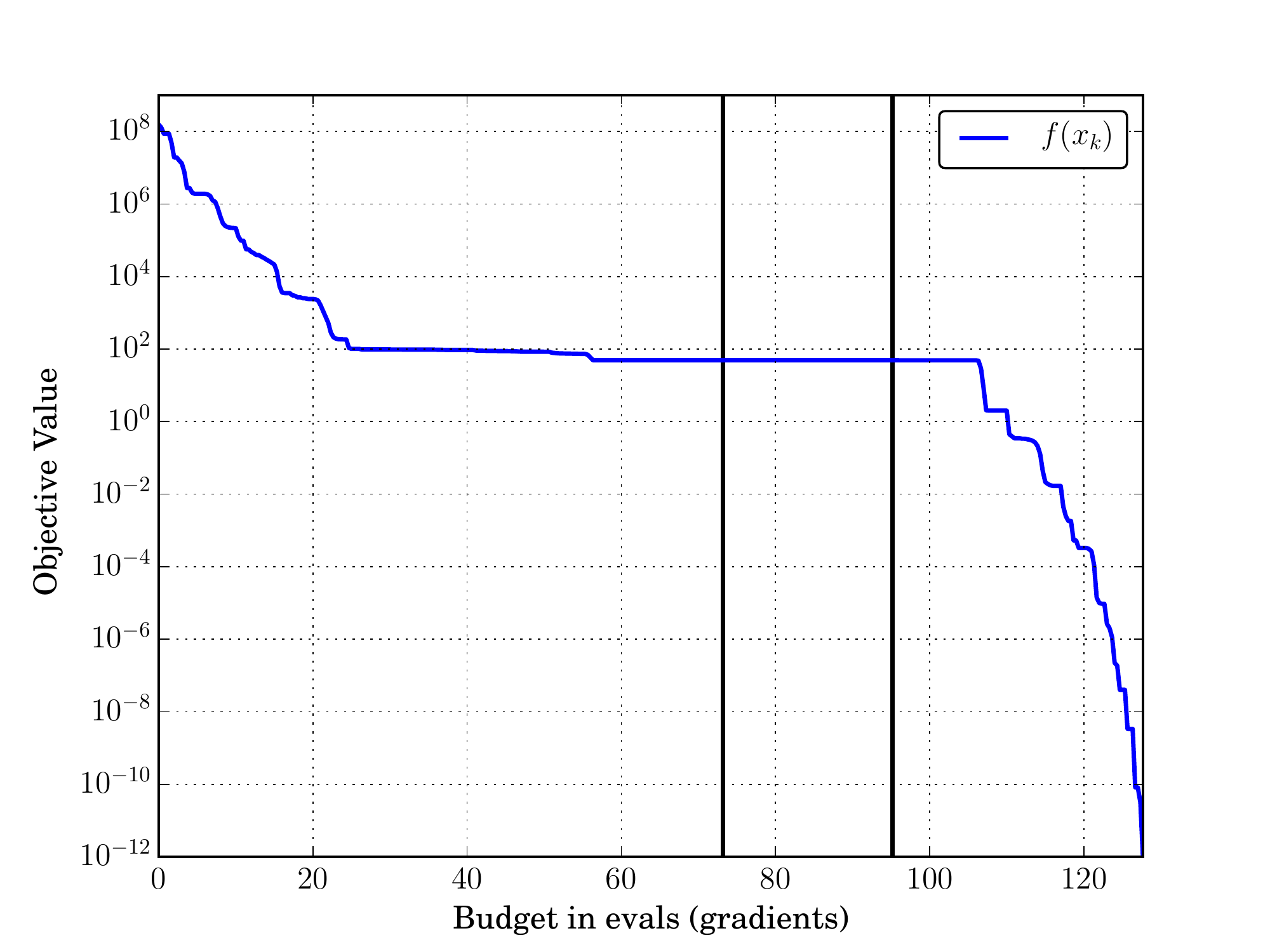}
		\caption{Objective reduction, problem 14 (soft restarts)}
		\label{fig_smooth_restarts_pybobyqa_prob14}
	\end{subfigure}
	\caption{Illustration of the impacts of multiple restarts for Py-BOBYQA on noiseless problems. Figure (a) is the same as \figref{fig_bobyqa_basic_smooth_mw}, but only showing Py-BOBYQA without restarts, and with soft (moving $\bx_k$) and hard restarts, without use of autodetection (problem collection (MW)). Figure (b) shows the objective value $f(\bx_k)$ using Py-BOBYQA with soft restarts and $2n+1$ interpolation points, for (MW) problem 14; the vertical lines indicate where restarts occurred.}
	\label{fig_smooth_restarts_pybobyqa}
\end{figure}

\section{Conclusion} \label{sec_conclusion}
We have presented two model-based DFO routines: DFO-LS for nonlinear least-squares problems, and Py-BOBYQA for general objective problems, both with optional bound constraints.
Both routines perform comparably to or better than state-of-the-art solvers on noisy problems with large, inexpensive budgets.
This is due to their ability to select different, more appropriate, algorithm parameters for noisy problems, and their use of multiple restarts.
Compared to other techniques for improving robustness to noise, such as sample averaging, regression models, and surrogate models, multiple restarts are cheap to implement and do not cause a deterioration in performance in the early phase of the algorithm.
However, both codes also allow the user to employ a wide family of sample averaging strategies, and DFO-LS additionally allows the use of regression models.
Although multiple restarts are designed for noisy problems, they do not disadvantage performance on smooth problems and can sometimes even improve it, such as when allowing the algorithm to escape local minima.

In addition, DFO-LS has the ability to start making progress using as few as 2 objective evaluations, rather than at least $n+1$ as in many model-based DFO codes (for an $n$-dimensional problem). 
This is a useful feature when objective evaluations are expensive, and can be used for noisy and noiseless objectives alike.
By reducing the initialization cost in this way, reasonable progress can be made on several problems even with fewer than $n$ objective evaluations (i.e.~less than the cost of evaluating the gradient of the objective at a single point).
This improvement has a tradeoff in performance for medium-sized budgets, but achieves the same long-term performance as having a full initialization cost.

Throughout, we have shown results for noisy problems using a problem- and noise-adjusted accuracy level.
This adjustment is chosen so that the progress defined by decreases in the noisy and underlying smooth objective produce similar results.
Therefore, this approach may be a useful way of benchmarking solvers for noisy problems, by focusing on a regime where progress as measured in the noisy objective (which is seen by the solver/user) corresponds to genuine optimization steps, and not luck in sampling errors.

\subsection{Acknowledgements}
This work was supported by the EPSRC Centre For Doctoral Training in Industrially Focused Mathematical Modelling (EP/L015803/1) in collaboration with the Numerical Algorithms Group Ltd.
We would like to thank Michael Ferris, Nick Gould, Raphael Hauser, Katya Scheinberg and Amy Willis for useful discussions regarding the DFO-LS algorithm, measuring solver performance, and comparing averaging and regression models.
We also acknowledge the use of the University of Oxford Advanced Research Computing (ARC) facility\footnote{\:\url{http://dx.doi.org/10.5281/zenodo.22558}} in carrying out this work. 

\addcontentsline{toc}{section}{References} 
\bibliographystyle{siam}
\bibliography{dfols_refs} 
\appendix

\section{Convergence Guarantees for DFO-LS} \label{sec_convergence}
In this section, we provide details of the convergence theory for the DFO-LS algorithm.
These results largely follow the arguments in \cite[Section 3]{Cartis2017a}.

\subsection{Accuracy of Regression Models} \label{sec_fully_linear}
In \secref{sec_main_algo}, we introduced $\Lambda$-poisedness as the key measure of the quality of the geometry of $Y_k$.
The $\Lambda$-poisedness of $Y_k$ guarantees accuracy of the regression models $\bem_k$ \eqref{eq_linear_models} and $m_k$ \eqref{eq_gn_full_model_dfo}, in the following sense \cite{Conn2009,Grapiglia2016}:

\begin{definition}[Fully linear, scalar model]
	A model $m_k\in C^1$ for a scalar function $f\in C^1$ is fully linear in $B(\bx_k,\Delta_k)$ if there exist positive constants $\kappa_{ef}$ and $\kappa_{eg}$, independent of $\bx_k$ and $\Delta_k$ such that
	\begin{align}
		|m_k(\bs) - f(\bx_k+\bs)| &\leq \kappa_{ef}\Delta_k^2, \\
		\|\grad m_k(\bs) - \grad f(\bx_k+\bs)\| &\leq \kappa_{eg}\Delta_k,
	\end{align}
	for all $\|\bs\|\leq\Delta_k$.
\end{definition}

\begin{definition}[Fully linear, vector model]
	A model $\bem_k\in C^1$ for a vector function $\br\in C^1$ is fully linear in $B(\bx_k,\Delta_k)$ if there exist positive constants $\kappa_{ef}^r$ and $\kappa_{eg}^r$, independent of $\bx_k$ and $\Delta_k$ such that
	\begin{align}
		\|\bem_k(\bs) - \br(\bx_k+\bs)\| &\leq \kappa_{ef}^r\Delta_k^2, \\
		\|J^m(\bs) - J(\bx_k+\bs)\| &\leq \kappa_{eg}^r\Delta_k,
	\end{align}
	for all $\|\bs\|\leq\Delta_k$, where $J^m$ and $J$ are the Jacobians of $\bem_k$ and $\br$ respectively.
\end{definition}

To establish the connection between $\Lambda$-poisedness of $Y_k$ and full linearity of our regression interpolation models $\bem_k$ and $m_k$, we require extra assumptions on the smoothness of the objective.

\begin{assumption} \label{ass_smoothness}
	The function $\br$ is $C^1$ and its Jacobian $J(\bx)$ is Lipschitz continuous in $\mathcal{B}$, the convex hull of $\cup_k B(\bx_k,\Delta_{max})$, with constant $L_J$. 
	We also assume that $\br(\bx)$ and $J(\bx)$ are uniformly bounded in the same region; i.e.~$\|\br(\bx)\|\leq r_{max}$ and $\|J(\bx)\| \leq J_{max}$ for all $\bx\in\mathcal{B}$.
\end{assumption}

If \assref{ass_smoothness} holds, then $\grad f$ is Lipschitz continuous in $\mathcal{B}$ with constant $L_{\grad f} \defeq r_{max}L_J + J_{max}^2$ \cite[Lemma 3.3]{Cartis2017a}.
The main result, analogous to \cite[Lemma 3.4]{Cartis2017a}, is the following.

\begin{lemma} \label{lem_fully_linear}
	Suppose \assref{ass_smoothness} holds, and $Y_k$ with $|Y_k|=p+1$ is $\Lambda$-poised in $B(\bx_k,\Delta_k)$ in the regression sense.
	Then $\bem_k$ and $m_k$ are fully linear models in $B(\bx_k,\Delta_k)$ for $\br$ and $f$ respectively, with constants
	\begin{align}
		\kappa_{ef}^r &= 2\kappa_{eg}^r, \\
		\kappa_{eg}^r &= \frac{1}{2}L_J\left(\sqrt{p}\:C+2\right), \\
		\kappa_{ef} &= \kappa_{eg} + L_{\grad f}/2 + \left(2r_{max}+\kappa_{eg}^r\Delta_{max}\right)\kappa_{eg}^r + (\kappa_{eg}^r\Delta_{max} + J_{max})^2, \\
		\kappa_{eg} &= L_{\grad f} + 2 J_{max}\kappa_{eg}^r \Delta_{max} + 2\kappa_{eg}^r r_{max} + 2(\kappa_{eg}^r\Delta_{max}+J_{max})^2,
	\end{align}
	where $C=\bigO(\Lambda)$.
\end{lemma}
\begin{proof}
This result extends the proof\footnote{\:This argument is not in the original text, but given in the errata (\url{http://www.mat.uc.pt/~lnv/idfo/}).} of \cite[Theorem 2.13]{Conn2009} to vector-valued functions, and their composition in a least-squares objective, in the style of \cite[Lemma 3.4]{Cartis2017a}.

Using $\by_0=\bx_k$, we may write
\be W_k \defeq \begin{bmatrix}1 & (\by_0-\bx_k)^{\top} \\ \vdots & \vdots \\ 1 & (\by_p-\bx_k)^{\top}\end{bmatrix} = \begin{bmatrix}1 & \b{0}^{\top} \\ \bee & L_k \end{bmatrix},\ee
where $\bee\in\R^p$ is the vector of ones and $L_k\in\R^{p\times n}$ has rows $(\by_t-\bx_k)^{\top}$ for $t=1,\ldots,p$.

We will also use scaled versions of these matrices:
\be \hat{W}_k \defeq \begin{bmatrix}1 & \b{0}^{\top} \\ \bee & \hat{L}_k \end{bmatrix}, \quad \text{where} \quad \hat{L}_k \defeq \begin{bmatrix}(\by_1-\bx_k)^{\top}/\Delta_k \\ \vdots \\ (\by_p-\bx_k)^{\top}/\Delta_k \end{bmatrix} = L/\Delta_k. \ee
Equivalently, we have
\be \hat{W}_k = W_k D_k, \quad \text{where} \quad D_k \defeq \begin{bmatrix}1 & \b{0}^{\top} \\ \b{0} & \frac{1}{\Delta_k}I_{n\times n}\end{bmatrix}. \label{eq_W_scaling} \ee
To begin, we recall that if $J(\bx)$ is continuous with Lipschitz constant $L_J$, then \cite[Appendix A]{Nocedal2006}
\be \|\br(\by) - \br(\bx_k) - J(\bx_k)(\by-\bx_k)\| \leq \frac{1}{2}L_J \|\by-\bx_k\|^2. \label{eq_lipschitz_bd} \ee
Our overdetermined interpolation system \eqref{eq_linear_interp_system} can be rewritten in matrix form as
\be W_k \begin{bmatrix}\br_k^{\top} \\ J_k^{\top}\end{bmatrix} = \begin{bmatrix}\br(\bx_k)^{\top} \\ \br(\by_1)^{\top} \\ \vdots \\ \br(\by_p)^{\top} \end{bmatrix}, \quad \text{and so} \quad \begin{bmatrix}\br_k^{\top} \\ J_k^{\top}\end{bmatrix} = W_k^{\dagger}\begin{bmatrix}\br(\bx_k)^{\top} \\ \br(\by_1)^{\top} \\ \vdots \\ \br(\by_p)^{\top} \end{bmatrix}. \label{eq_interp_system_matrix} \ee
Separately, we compute
\be W_k \begin{bmatrix}\br(\bx_k)^{\top} \\ J(\bx_k)^{\top}\end{bmatrix} - \begin{bmatrix}\br(\bx_k)^{\top} \\ \br(\by_1)^{\top} \\ \vdots \\ \br(\by_p)^{\top} \end{bmatrix} = \begin{bmatrix}\br(\bx_k)^{\top} \\ [\br(\bx_k)+J(\bx_k)(\by_1-\bx_k)]^{\top} \\ \vdots \\ [\br(\bx_k) + J(\bx_k)(\by_p-\bx_k)]^{\top} \end{bmatrix} - \begin{bmatrix}\br(\bx_k)^{\top} \\ \br(\by_1)^{\top} \\ \vdots \\ \br(\by_p)^{\top} \end{bmatrix} =: E. \label{eq_E_defn} \ee
Combining \eqref{eq_interp_system_matrix} and \eqref{eq_E_defn}, we get
\be \begin{bmatrix}[\br(\bx_k)-\br_k]^{\top} \\ [J(\bx_k)-J_k]^{\top}\end{bmatrix} = W_k^{\dagger} E, \ee
and so from \eqref{eq_W_scaling}, since $D_k$ is invertible, we have $\hat{W}_k^{\dagger}=D_k ^{-1} W_k^{\dagger}$ and hence conclude
\be \begin{bmatrix}[\br(\bx_k)-\br_k]^{\top} \\ \Delta_k [J(\bx_k)-J_k]^{\top}\end{bmatrix} = D_k^{-1} \begin{bmatrix}[\br(\bx_k)-\br_k]^{\top} \\ [J(\bx_k)-J_k]^{\top}\end{bmatrix} = \hat{W}_k^{\dagger} E. \ee
Since the first row of $E\in \R^{(p+1)\times m}$ is zero, and the norms of all other rows are bounded by \eqref{eq_lipschitz_bd}, we have
\be \|E\| \leq \|E\|_F = \left(0 + \sum_{t=1}^{p}\|\br(\bx_k)+J(\bx_k)(\by_t-\bx_k) - \br(\by_t)\|^2\right)^{1/2} \leq \frac{1}{2}L_J \sqrt{p}\Delta_k^2. \ee
This gives us the error bounds
\begin{align}
	\|\br(\bx_k)-\br_k\| &\leq \left\|\begin{bmatrix}[\br(\bx_k)-\br_k]^{\top} \\ \Delta_k [J(\bx_k)-J_k]^{\top}\end{bmatrix}\right\| \leq \frac{1}{2}L_J\sqrt{p}\|\hat{W}_k^{\dagger}\|\Delta_k^2, \label{eq_rk_error} \\
	\|J(\bx_k)-J_k\| &\leq \Delta_k^{-1} \left\|\begin{bmatrix}[\br(\bx_k)-\br_k]^{\top} \\ \Delta_k [J(\bx_k)-J_k]^{\top}\end{bmatrix}\right\| \leq \frac{1}{2}L_J\sqrt{p}\|\hat{W}_k^{\dagger}\|\Delta_k.
\end{align}
Thus we conclude that for any $\by\in B(\bx_k,\Delta_k)$
\be \|J_k-J(\by)\| \leq \|J_k-J(\bx_k)\| + \|J(\by)-J(\bx_k)\| \leq L_J \left(1 + \frac{1}{2}\sqrt{p}\|\hat{W}_k^{\dagger}\|\right)\Delta_k.  \label{eq_fully_linear_vector_g} \ee
For convenience, define $\kappa_{eg}^r \defeq L_J \left(1 + \frac{1}{2}\sqrt{p}\|\hat{W}_k^{\dagger}\|\right)$.
Next, we compute
\begin{align}
	\|\bem_k(\by-\bx_k) - \br(\by)\| &= \|\br(\by) - \br_k - J_k(\by-\bx_k)\|, \\
	&\leq \|\br(\bx_k)-\br_k\| + \|\br(\by) - \br(\bx_k) - J(\bx_k)(\by-\bx_k)\| \nonumber \\
	&\qquad\qquad + \|J(\bx_k)-J_k\|\cdot\|\by-\bx_k\|, \\
	&\leq \left(\frac{1}{2}L_J\sqrt{p}\|\hat{W}_k^{\dagger}\| + \frac{L_J}{2} + \kappa_{eg}^r\right)\Delta_k^2,
\end{align}
where we use \eqref{eq_fully_linear_vector_g} and \eqref{eq_lipschitz_bd}.
Hence $\bem_k$ is a fully linear model for $\br$ with constants $\kappa_{eg}^r$ defined above and $\kappa_{ef}^r = 2\kappa_{eg}^r$.

Since $\bem_k$ is fully linear, \eqref{eq_fully_linear_vector_g} gives us
$\|J_k\| \leq \|J(\bx_k) - J_k\| + \|J(\bx_k)\| \leq \kappa_{eg}^r\Delta_{max} + J_{max}$, 
so $\|J_k\|$ is uniformly bounded for all $k$. Since $H_k = 2 J_k^{\top}J_k$, we get that $\|H_k\| = 2 \|J_k\|^2$ is uniformly bounded for all $k$. 
Similarly, using \eqref{eq_rk_error}, we have the bounds $\|\br(\bx_k)-\br_k\| \leq \kappa_{eg}^r\Delta_k^2$ and $\|\br_k\| \leq r_{max}+\kappa_{eg}^r\Delta_{max}^2$.

To prove full linearity of $m_k$, we begin by computing
\begin{align}
	\|\grad m_k(\by-\bx_k) - \grad f(\by)\| &= \|\grad f(\by) - 2 J_k^{\top}\br_k - 2 J_k^{\top}J_k(\by-\bx_k)\|, \\
	&\leq \|\grad f(\by) - \grad f(\bx_k)\| + \|2J(\bx_k)^{\top}(\br(\bx_k)-\br_k)\| \nonumber \\ 
	&\qquad\qquad + \|2(J(\bx_k)-J_k)^{\top}\br_k\| + \|2 J_k^{\top}J_k\|\cdot\|\by-\bx_k\|, \\
	&\leq L_{\grad f} \Delta_k + 2J_{max}\kappa_{eg}^r\Delta_k^2 \nonumber \\
	&\qquad\qquad + 2\kappa_{eg}^r r_{max}\Delta_k + 2\left(\kappa_{eg}^r\Delta_{max}+J_{max}\right)^2\Delta_k, \\
	&\leq \kappa_{eg}\Delta_k,
\end{align}
where $\kappa_{eg}\defeq L_{\grad f} + 2 J_{max}\kappa_{eg}^r \Delta_{max} + 2\kappa_{eg}^r r_{max} + 2(\kappa_{eg}^r\Delta_{max}+J_{max})^2$, as required.
Then, we use \eqref{eq_lipschitz_bd} and the above to get
\begin{align}
	|m_k(\by-\bx_k) - f(\by)| &= \left|f(\by) - \|\br_k\|^2 - \bg_k^{\top}(\by-\bx_k) - \frac{1}{2}(\by-\bx_k)^{\top}H_k(\by-\bx_k)\right|, \\
	&\leq \left|f(\by) - f(\bx_k) - \grad f(\bx_k)^{\top}(\by-\bx_k)\right| + \|\br(\bx_k)-\br_k\|\left(\|\br(\bx_k)\|+\|\br_k\|\right) \nonumber \\ 
	&\qquad\qquad + \left\|\grad f(\bx_k) - \bg_k - \frac{1}{2}H_k(\by-\bx_k)\right\|\cdot\|\by-\bx_k\|, \\
	&\leq \frac{1}{2}L_{\grad f} \Delta_k^2 + \kappa_{eg}^r \Delta_k^2 \left(2r_{max}+\kappa_{eg}^r\Delta_{max}\right) \nonumber \\
	&\qquad\qquad + \left[\|\grad f(\bx_k) - \grad m_k(\bx_k)\| + \frac{1}{2}\|H_k\|\cdot\|\by-\bx_k\|\right]\cdot \Delta_k, \\
	&\leq \frac{1}{2}L_{\grad f} \Delta_k^2 + \kappa_{eg}^r \left(2r_{max}+\kappa_{eg}^r\Delta_{max}\right) \Delta_k^2 \nonumber \\
	&\qquad\qquad + \left[\kappa_{eg}\Delta_k + (\kappa_{eg}^r\Delta_{max}+J_{max})^2\Delta_k\right]\Delta_k, \\
	&\leq \kappa_{ef}\Delta_k^2,
\end{align}
where $\kappa_{ef} \defeq \kappa_{eg} + L_{\grad f}/2 + \kappa_{eg}^r \left(2r_{max}+\kappa_{eg}^r\Delta_{max}\right) + (\kappa_{eg}^r\Delta_{max} + J_{max})^2$.
Lastly, we have $C\defeq\|\hat{W}_k^{\dagger}\|\leq\sqrt{p+1}\:\Lambda=\bigO(\Lambda)$ from \cite[Theorem 2.9]{Conn2008}.
\end{proof}

For our convergence theory to hold, we need to be more specific about the geometry of $Y_k$ being `good' in \algref{alg_dfols}; for the purposes of convergence we take `good' to mean `$Y_k$ is $\Lambda$-poised'.
Note that in the $p=n$ case of exact interpolation, there are algorithms for changing $Y_k$ to make it $\Lambda$-poised.
For the regression case of $p>n$, it suffices to make a subset of $n+1$ points in $Y_k$ $\Lambda$-poised --- see \cite[Chapter 6]{Conn2009} for a discussion of these issues. The case of reduced initialization ($p<n$) is addressed at the end of next section.

\subsection{Global Convergence and Complexity}
To ensure global convergence of DFO-LS, we need to add one more phase in \algref{alg_dfols}.
This phase, known as the `criticality phase', is called when the interpolated model constructed in line \ref{ln_loop} has $\|\bg_k\|\leq\epsilon_C$.
In this situation, our model gradient is small, so we impose two requirements: shrink $\Delta_k$ to be of the same magnitude as $\|\bg_k\|$ (specifically, we achieve $\Delta_k\leq \mu\|\bg_k\|$), and ensure $m_k$ is fully linear.
Details of this phase can be found in \cite[Appendix B]{Cartis2017a}.
A version of DFO-LS including the criticality phase is given in \algref{alg_dfols_theory}.
We consider this version of DFO-LS only for case of noiseless objectives (so \texttt{NOISY=FALSE}), and where we do not use the reduced initialization cost (i.e.~$p_{init}=p$).

\begin{algorithm}
	\small{
	\begin{algorithmic}[1]
		\Require Starting point $\bx_0\in\R^n$, initial trust region radius $\Delta_0^{init}>0$ and interpolation set size $p\geq n$. 
		\Statex Parameters from \algref{alg_dfols} are the same, except $p_{init}=p$ and \texttt{NOISY=FALSE}.
		Additional parameters are criticality threshold $\epsilon_C>0$, criticality scaling $\mu>0$ and poisedness constant $\Lambda\geq1$. We also require $\gamma_S < 2c_1/(1+\sqrt{1+2c_1})$, with $c_1$ from \assref{ass_cauchy_decrease}.
		\State Build an initial interpolation set $Y_0\subset B(\bx_0,\Delta_0^{init})$ of size $p_{init}+1=p+1$, with $\bx_0\in Y_0$. Set $\rho_0^{init}=\Delta_0^{init}$.
		\For{$k=0,1,2,\ldots$} 
			\State Given $\bx_k$ and $Y_k$, solve the interpolation problem \eqref{eq_linear_interp_system} and form $m_k^{init}$ \eqref{eq_gn_full_model_dfo}.
			\If{$\|\bg_k^{init}\| \leq \epsilon_C$}
				\State \underline{Criticality Phase}: using \cite[Algorithm 2]{Cartis2017a}, modify $Y_k$ and find $\Delta_k\leq\Delta_k^{init}$ such that $Y_k$ is $\Lambda$-poised in $B(\bx_k,\Delta_k)$ and $\Delta_k\leq\mu\|\bg_k\|$, where $\bg_k$ is the gradient of the new $m_k$. Set $\rho_k= \min(\rho_k^{init}, \Delta_k)$.
			\Else
				\State Set $m_k=m_k^{init}$, $\Delta_k=\Delta_k^{init}$ and $\rho_k=\rho_k^{init}$.
			\EndIf
			\State Follow lines \ref{ln_trs} to \ref{ln_loop_end} of \algref{alg_dfols} to determine $\bx_{k+1}$, $\Delta_{k+1}^{init}$ and $\rho_{k+1}^{init}$, updating $Y_k$ as needed. 
			All references to `improving the geometry of $Y_k$' must be changed to `make $Y_k$ $\Lambda$-poised in $B(\bx_{k+1},\Delta_{k+1}^{init})$'.
			Similarly, checking `geometry of $Y_k$ is good' must be changed to `$Y_k$ is $\Lambda$-poised in $B(\bx_k,\Delta_k)$'.
			We do not terminate if $\rho_k \leq \rho_{end}$, or if objective decrease is too slow.
		\EndFor
	\end{algorithmic}
	} 
	\caption{DFO-LS with criticality phase.}
	\label{alg_dfols_theory}
\end{algorithm}

To guarantee convergence of our algorithm, we want our approximate solution to the trust region subproblem \eqref{eq_tr_subproblem} to provide a reasonable decrease in $m_k$, and so we require the following minimal assumption.

\begin{assumption} \label{ass_cauchy_decrease}
	The calculated step $\bs_k$ \eqref{eq_tr_subproblem} in line \ref{ln_trs} of \algref{alg_dfols} satisfies the `Cauchy decrease' condition
	\be m_k(\b{0}) - m_k(\bs_k) \geq c_1 \|\bg_k\| \min\left(\Delta_k, \frac{\|\bg_k\|}{\max(\|H_k\|, 1)}\right), \ee
	for some $c_1\in[1/2, 1]$ independent of $k$.
\end{assumption}

This assumption is easy to achieve, for instance by one iteration of steepest descent with exact linesearch (achieving $c_1=1/2$).
Lastly, we require one more assumption, which is very common for trust region methods.

\begin{assumption} \label{ass_bdd_hess}
	We assume that $\|H_k\|\leq\kappa_H$ for all $k$, for some $\kappa_H\geq1$.
\end{assumption}

We can now state the convergence result for DFO-LS; aside from the details in \secref{sec_fully_linear} the details of the proof are identical to \cite{Cartis2017a}.

\begin{theorem} \label{thm_lim}
	Suppose Assumptions \ref{ass_cauchy_decrease}, \ref{ass_smoothness} and \ref{ass_bdd_hess} hold.
	Then \algref{alg_dfols_theory} produces iterates $\bx_k$ such that $\lim_{k\to\infty}\Delta_k=0$ and
	$\lim_{k\to\infty}\|\grad f(\bx_k)\|=0$.
\end{theorem}

Again following the details from \cite{Cartis2017a}, we can also bound the number of iterations and objective evaluations required to achieve $\|\grad f(\bx_k)\|\leq \epsilon$. 

\begin{theorem} \label{thm_iters_bound}
	Suppose Assumptions \ref{ass_cauchy_decrease}, \ref{ass_smoothness} and \ref{ass_bdd_hess} hold, and that the criticality threshold $\epsilon_C \geq c_3\epsilon$ for some constant $c_3>0$.
	Then the number of iterations  $i_{\epsilon}$ (i.e.~the number of times a model $m_k$ \eqref{eq_gn_full_model_dfo} is built) needed by \algref{alg_dfols_theory} until $\|\grad f(\bx_{i_{\epsilon}+1})\|< \epsilon$ is at most
	\begin{align}
		\left\lfloor\frac{4f(\bx_0)}{\eta_1 c_1}\left(1 + \frac{\log \overline{\gamma}_{inc}}{|\log \alpha_3|}\right)\max\left(\kappa_H c_4^{-2}\epsilon^{-2}, c_4^{-1}c_5^{-1}\epsilon^{-2}, c_4^{-1}\Delta_0^{-1}\epsilon^{-1}\right)\right. \nonumber \\
\left.\quad + \frac{4}{|\log \alpha_3|}\max\left(0,\log \left(\Delta_0 c_5^{-1} \epsilon^{-1}\right)\right)\right\rfloor \label{eq_num_iters}
	\end{align}
	where $c_4 \defeq \min\left(c_3, (1 + \kappa_{eg}\mu)^{-1}\right)$, where $\mu$ is the criticality phase threshold on $\|\bg_k\|$, and
	\be c_5 \defeq \min\left(\frac{\omega_C}{\kappa_{eg}+1/\mu}, \frac{\alpha_1 c_4}{\kappa_H}, \alpha_1\left(\kappa_{eg} + \frac{2\kappa_{ef}}{c_1(1-\eta_2)}\right)^{-1}\right). \ee
\end{theorem}

For succinctness, we can look at the complexity bounds to leading order in $\epsilon$. 

\begin{corollary} \label{cor_complexity}
	Suppose the assumptions of \thmref{thm_iters_bound} hold.
	Then for $\epsilon\in(0,1]$, the number of iterations  $i_{\epsilon}$ needed by \algref{alg_dfols_theory} until $\|\grad f(\bx_{i_{\epsilon}+1})\|< \epsilon$ is at most $\bigO(\kappa_H \kappa_d^2 \epsilon^{-2})$, and the number of objective evaluations until $i_{\epsilon}$ is at most $\bigO(\kappa_H \kappa_d^2 p \epsilon^{-2})$, where $\kappa_d\defeq\max(\kappa_{ef},\kappa_{eg})=\bigO(p L_J^2)$.
\end{corollary}

If the reduced initialization phase with $p<n$ is appended at the start of \algref{alg_dfols_theory}, Theorem \ref{thm_lim} continues to hold, and the complexity bounds in Theorem \ref{thm_iters_bound} 
and Corollary \ref{cor_complexity} for the resulting algorithm increase by $n$ iterations and function evaluations. This is due to the growing set of directions until full-dimensionality that is being generated in the early phase, when no points get removed; the geometry of this set is automatically adjusted by the algorithm, if needed.


\section{General Features of DFO-LS} \label{sec_general_features_appendix}
In this section, we provide more details on the general features of DFO-LS summarized in \secref{sec_implementation_general}.

\paragraph{Geometry-Improving Steps}
The goal of the geometry-improving steps in \algref{alg_dfols} is to improve the quality of the model $m_k$; specifically, we wish to make $Y_k$ $\Lambda$-poised in $B(\bx_k,\Delta_k)$, so $m_k$ is a fully linear model for $f$ in the trust region.
However, as mentioned above, guaranteeing the $\Lambda$-poisedness of $Y_k$ if $|Y_k|>n+1$ is not straightforward.
If $|Y_k|=n+1$, then we can achieve $\Lambda$-poised via the iteration \cite[Algorithm 6.3]{Conn2009}
\begin{enumerate}
	\item Select the point $\by_t\in Y_k$ ($\by_t\neq\bx_k$) for which $\max_{\by\in B(\bx_k,\Delta_k)} |\Lambda_t(\by)|$ is maximized;
	\item Replace $\by_t$ in $Y_k$ with $\by^+$, where 
	\be \by^+ = \argmax_{\by\in B(\bx_k,\Delta_k)} |\Lambda_t(\by)|, \label{eq_geom_improvement_2} \ee
	and repeat until $Y_k$ is $\Lambda$-poised.
\end{enumerate}
As in DFO-GN, in practice we perform a simplified geometry-improving phase: we simply choose $\by_t\in Y_k$ to be the point furthest from $\bx_k$, and replace it with $\by^+$ as defined by \eqref{eq_geom_improvement_2}.
We do not repeat this process; only one point is moved per call of the geometry-improving phase.

Similarly, we use a simplified test to determine if the geometry of $Y_k$ needs improving at all.
In theory, we need to check if $Y_k$ is $\Lambda$-poised.
Instead, we say that the geometry of $Y_k$ needs improving if $\max_t \|\by_t-\bx_k\| > \epsilon$, for some threshold $\epsilon$, usually a constant multiple of $\Delta_k$ or $\rho_k$.

\paragraph{Model Updating}
In \algref{alg_dfols}, we only add $\bx_k+\bs_k$ to $Y_k$ in successful steps.
However in practice, like in DFO-GN, we always incorporate new information when it becomes available, and so we update $Y_{k+1}=Y_k\cup\{\bx_k+\bs_k\}\setminus\{\by_t\}$ for some $\by_t\neq\bx_{k+1}$ at every iteration, successful or otherwise.
Similarly, we always choose to centre our trust region at the best value found so far, so we ensure $\bx_k=\argmin_{\by_t\in Y_k} f(\by_t)$ at every iteration --- this optimal point (so far) can come from a trust region step, or even from a geometry-improving phase.

Given a point to add to $Y_k$ (to form $Y_{k+1}$), we use the method from DFO-GN for determining which point it should replace.
This method uses a criterion which chooses to remove points which are far from $\bx_k$ and for which the replacement would most improve the geometry of $Y_{k+1}$.

\paragraph{Inclusion of Bound Constraints and Variable Scaling}
The implementation of DFO-LS solves problems with optional bound constraints.
That is, it solves \eqref{eq_ls_definition} subject to $\b{a} \leq \bx \leq \b{b}$.
The only changes to \algref{alg_dfols} required for this are in the calculation of the trust region step \eqref{eq_tr_subproblem} and geometry-improving step \eqref{eq_geom_improvement_2}, which now also have bound constraints.

For the calculation of a trust region step subject to bound constraints, we use the routine \texttt{TRSBOX} from BOBYQA \cite{Powell2009}, as modified by Zhang et al.~in DFBOLS \cite{Zhang2010}.
Geometry-improving steps with bound constraints are calculated using \cite[Algorithm 3]{Cartis2017a}.

Because bound constraints can often provide information about the natural scaling of a problem, DFO-LS allows the optional internal scaling of variables based on the bound constraints, to reduce the likelihood of ill-conditioning.
If this is used, we internally shift and scale the inputs so that the new feasible region is $\bx\in[0,1]^n$.

\paragraph{Termination Criteria}
There are four ways in which DFO-LS can terminate.
The first three are inherited from DFO-GN:
\begin{itemize}
	\item Small objective value: since we have the lower bound $f\geq 0$ always, we allow termination when $f(\bx_k) \leq \max\{\epsilon_{abs}, \epsilon_{rel}f(\bx_0)\}$, for user-specified parameters $\epsilon_{abs}$ and $\epsilon_{rel}$.
	Having this feature is especially useful for DFO solvers, when often just achieving some desired decrease in the objective is the goal, rather than solving to full optimality (e.g.~when function evaluations are expensive);
	\item Small trust region: we know that $\rho_k\to 0$ as $k\to\infty$ \cite[Lemma 3.11]{Cartis2017a}, so we terminate when $\rho_k \leq \rho_{end}$; and
	\item Computational budget: we terminate after a given number of evaluations of the objective.
\end{itemize}
The last two of these are designed to cause termination after a sufficient number of unsuccessful steps.
The first criteria is triggered by successful steps, but is likely to be triggered only for zero-residual problems.

To ensure a timely termination based on successful steps, we introduce an extra criterion, similar to the ``$f_i^{*\prime}$ test'' of Larson and Wild \cite{Larson2013}.
We define a successful iteration as `slow' if the last $K$ successful iterations have produced an average reduction in $\log(f(\bx_k))$ below a given threshold.
That is, if $\{k_i : i\in\N\}$ are the successful iterations, then iteration $k_i$ is `slow' if
\be \frac{\log(f(\bx_{k_{(i-K)}})) - \log(f(\bx_{k_i}))}{K} < \epsilon, \label{eq_slow_termination_defn} \ee
for some value $\epsilon>0$.
Note that since we are only considering successful iterations, $f(\bx_{k_i})$ will be the best objective value found up to iteration $k_i$.
Our termination condition is then: quit after successful iteration $k_i$ if $\{k_{i-N+1},\ldots,k_{i-1},k_i\}$ were all `slow', for some $N\in\N$.

Lastly, DFO-LS also includes an optional noise-aware termination condition.
Specifically, we terminate if all function values $f(\by_t)$ are within some user-provided `noise level' of $f(\bx_k)$. That is, for all $t=1,\ldots,p$, either
\be |f(\by_t) - f(\bx_k)| \leq \mathrm{const}\cdot \frac{\epsilon}{\sqrt{N_t}} \qquad \text{or} \qquad \left|\frac{f(\by_t)}{f(\bx_k)}\right| \leq \mathrm{const}\cdot \frac{\epsilon}{\sqrt{N_t}}, \label{eq_termination_noise} \ee
where $N_t$ is the number of samples used to estimate the value $f(\by_t)$; see \secref{sec_restarts_description} for details.
Which of these criteria is used depends on whether the user has specified $\epsilon$ as an additive or multiplicative noise level in the evaluation of $f$, and the value of `const' is also user-provided (default is 1).

\paragraph{Default Parameters for Noisy Problems}
One of the main problem types that DFO-LS is designed to solve is where objective evaluations are noisy.
In this situation, the set of default parameters --- which are designed for smooth objectives --- are not necessarily good choices.

The most notable examples of this are the parameters which govern decreases of $\Delta_k$ and $\rho_k$, namely $\gamma_{dec}$, $\alpha_1$ and $\alpha_2$ (default values $0.5$, $0.1$ and $0.5$ respectively).
When we have noisy evaluations, it is common to get unsuccessful iterations even when a step $\bs_k$ is useful, because the noise in the objective evaluation leads to inaccuracies in the calculated $r_k$ \eqref{eq_tr_ratio}.
This then leads to unnecessary reductions in the trust region radius, causing the algorithm to progress more slowly, and potentially terminate too early.

In DFO-LS, we allow the user to specify if their objective evaluation is noisy, and consequently modify the default values for several algorithm parameters.
Note that the user can choose to override any parameter value by specifying it directly, even if the default has been modified.
For example, the `noisy problem' default values of $\gamma_{dec}$, $\alpha_1$ and $\alpha_2$ are $0.98$, $0.9$ and $0.95$ respectively.

\paragraph{Other Differences}
There are other small differences between the implementation of DFO-LS and \algref{alg_dfols}, which are inherited from DFO-GN; a list of these may be found in \cite[Section 4.4]{Cartis2017a}.

We also change the default method for constructing the initial set $Y_0$.
In DFO-GN, like DFBOLS \cite{Zhang2010} and BOBYQA \cite{Powell2009}, the initial set is typically taken to be $\bx_0\pm \Delta_0\bee_t$ for coordinate vectors $\bee_t$ (adjusted when for bound constraints and for more than $2n+1$ interpolation points).
In DFO-LS, the default mechanism is to use $\bx_0\pm\Delta_0\b{q}_t$ for random orthonormal vectors $\b{q}_t$ (again, adjusted in the case of bound constraints or $p>2n$).

\section{Comparison of Sample Averaging and Regression} \label{sec_regression_v_avg}
There are two places in \algref{alg_dfols} where noise in objective evaluations can have an impact: the construction of the model \eqref{eq_linear_models}, and the measurement of objective decrease \eqref{eq_tr_ratio}. 
We show below that the errors in model construction due to noise are likely comparable when using either sample averaging or regression.
However, for a fixed level of noise, sample averaging will produce a better estimate of objective decrease (compare \cite[Lemma 9.1]{Nocedal2006}, for instance).
Thus, overall, we would expect sample averaging to perform somewhat better than regression, when considering overall robustness. 
Since sample averaging may use more objective evaluations per iteration, this may not be the case when the computational budget is limited.

\paragraph{Error Bounds on Model Estimation}
Here, we give a short argument that sample averaging and regression models produce comparably good models.
For simplicity, suppose we wish to construct a model for a linear function $f(\bx)=c+\b{g}^{\top}\bx$ from noisy evaluations $\t{f}(\bx)=f(\bx)+\epsilon$, where $\epsilon\sim N(0,\sigma^2)$ is i.i.d.~stochastic noise.

If we perform sample averaging using $N$ samples at a given interpolation point $\by_t$, we get an unbiased estimate for $f$ with smaller variance:
\be \t{f}_N(\by_t) = \frac{1}{N}\sum_{i=1}^{N}(f(\by_t) + \epsilon_i) \sim N\left(f(\by_t), \frac{\sigma^2}{N}\right). \ee
Now suppose we construct a regression model as per \eqref{eq_interp_conditions} using points $Y\defeq \{\by_0,\ldots,\by_p\}$ for some $p\geq n$.
We assume that $Y$ is \emph{strongly} $\Lambda$-poised \cite[Definition 4.10]{Conn2009} in $B(\by_0,\Delta)$, which is a stronger condition\footnote{\:It can be achieved if, for instance, $Y$ is formed by concatenating several sets of size $n+1$ which are all $\Lambda$-poised in the \emph{interpolation} sense \cite[Definition 3.6]{Conn2009}.} than \defref{def_poised}, but better suited to comparing the geometry of sets with different sizes $p$.

Under these conditions, the Gauss-Markov Theorem implies that the regression model from \eqref{eq_interp_conditions} gives an optimal unbiased estimator $(\t{c}, \t{\bg})$ for $(c, \b{g})$ with error (co)variance 
\be \mathbb{E}\left[\left(\begin{bmatrix}\t{c} \\ \t{\bg}\end{bmatrix} - \begin{bmatrix}c \\ \bg\end{bmatrix}\right) \left(\begin{bmatrix}\t{c} \\ \t{\bg}\end{bmatrix} - \begin{bmatrix}c \\ \bg\end{bmatrix}\right)^{\top}\right] = \frac{\sigma^2}{N}(W_k^{\top}W_k)^{-1}. \ee
By shifting $Y$ to $\hat{Y}\defeq\{(\by_t-\by_0)/\Delta : t=0,\ldots,p\}$ as in the proof of Lemma \ref{lem_fully_linear},
 we have \eqref{eq_W_scaling}, and so the variance satisfies
\be \left\|(W_k^{\top}W_k)^{-1}\right\| = \left\|D_k\left(\hat{W}_k^{\top}\hat{W}_k\right)^{-1}D_k^{\top}\right\| \leq \max(1, \Delta^{-2})\left\|(\hat{W}_k^{\top}\hat{W}_k)^{-1}\right\| = \frac{\max(1, \Delta^{-2})}{\sigma_{min}(\hat{W}_k)^2}, \ee
where $\sigma_{min}(\hat{W}_k)$ is the smallest singular value of $\hat{W}_k$, since $\|D_k\|=\max(1,\Delta^{-1})$.
However, from \cite[Theorem 4.12]{Conn2009}, the strong $\Lambda$-poisedness of $Y$ gives
\be \frac{1}{\sigma_{min}(\hat{W}_k)} \leq \frac{\theta(n+1)}{\sqrt{p+1}}\Lambda, \ee
for some constant $\theta>0$.
All together, we get
\be \mathbb{E}\left[\left(\begin{bmatrix}\t{c} \\ \t{\bg}\end{bmatrix} - \begin{bmatrix}c \\ \bg\end{bmatrix}\right) \left(\begin{bmatrix}\t{c} \\ \t{\bg}\end{bmatrix} - \begin{bmatrix}c \\ \bg\end{bmatrix}\right)^{\top}\right] \leq \frac{\sigma^2 \max(1, \Delta^{-2}) \theta^2 (n+1)^2 \Lambda^2}{N(p+1)}, \ee
so the square error in the regression model $(\t{c},\t{\bg})$ is inversely proportional to $N(p+1)$, the total number of evaluations of $\t{f}$ used in building the right-hand side of \eqref{eq_linear_interp_system}.
We conclude that, all else being equal (including the strong $\Lambda$-poisedness of $Y$), we would get the same model error from using $|Y|=c(n+1)$ points with no sample averaging, or $|Y|=n+1$ points and using $c$ samples per point.
This provides further support for the similar results for sample averaging and regression models observed in \secref{sec_regression_results}.

\section{General Objective Test Problems} \label{sec_genobj_problems}
\begin{table}[H]
	\centering
	\small{
	\begin{tabular}{rlccccc} 
		\hline\noalign{\smallskip}
		\# & Problem & $n$ & $f(\bx_0)$ & $f(\bx^*)$ & Parameters \\ \noalign{\smallskip}\hline\noalign{\smallskip} 
		1 & ARWHEAD & 100 & 297 & 0 & $N = 100$ \\
		2 & BDEXP* & 100 & 26.52572 & 0 & $N = 100$ \\
		3 & BOX & 100 & 0 & $-11.24044$ & $N = 100$ \\
		4 & BOXPOWER & 100 & 866.2462 & 0 & $N = 100$ \\
		5 & BROYDN7D & 100 & 350.9842 & 40.12284 & $N/2 = 50$ \\
		6 & CHARDIS1 & 98 & 830.9353 & 0 & $NP1 = 50$ \\
		7 & COSINE & 100 & 86.88067 & $-99$ & $N = 100$ \\
		8 & CURLY10 & 100 & $-6.237221\times 10^{-3}$ & $-1.003163\times 10^{4}$ & $N = 100$ \\
		9 & CURLY20 & 100 & $-1.296535\times 10^{-2}$ & $-1.003163\times 10^{4}$ & $N = 100$ \\
		10 & DIXMAANA & 90 & 856 & 1 & $M = 30$ \\
		11 & DIXMAANF & 90 & $1.225292\times 10^{3}$ & 1 & $M = 30$ \\
		12 & DIXMAANP & 90 & $2.128648\times 10^{3}$ & 1 & $M = 30$ \\
		13 & ENGVAL1 & 100 & 5841 & 109.0881 & $N = 100$ \\
		14 & FMINSRF2 & 64 & 23.461408 & 1 & $P = 8$ \\
		15 & FMINSURF & 64 & 32.84031 & 1 & $P = 8$ \\
		16 & NCB20 & 110 & 202.002 & 179.7358 & $N = 100$ \\
		17 & NCB20B & 100 & 200 & 196.6801 & $N = 100$ \\
		18 & NONCVXU2 & 100 & $2.639748\times 10^{6}$ & 231.8274 & $N = 100$ \\
		19 & NONCVXUN & 100 & $2.727010\times 10^{6}$ & 231.6808 & $N = 100$ \\
		20 & NONDQUAR & 100 & 106 & 0 & $N = 100$ \\
		21 & ODC & 100 & 0 & $-1.098018\times 10^{-2}$ & $(NX, NY) = (10, 10)$ \\
		22 & PENALTY3 & 100 & $9.801798\times 10^{7}$ & 0.001 & $N/2 = 50$ \\
		23 & POWER & 100 & $2.550250\times 10^{7}$ & 0 & $N = 100$ \\
		24 & RAYBENDL & 62 & 98.03445 & 96.25168 & $NKNOTS = 32$ \\
		25 & SCHMVETT & 100 & $-280.2864$ & $-294$ & $N = 100$ \\
		26 & SINEALI* & 100 & $-0.8414710$ & $-9.900962\times 10^{3}$ & $N = 100$ \\
		27 & SINQUAD & 100 & 0.6561 & $-4.005585\times 10^{3}$ & $N = 100$ \\
		28 & TOINTGOR & 50 & $5.073786\times 10^{3}$ & $1.373905\times 10^{3}$ & --- \\
		29 & TOINTGSS & 100 & 892 & 10.10204 & $N = 100$ \\
		30 & TOINTPSP & 50 & $1.827709\times 10^{3}$ & 225.5604 & --- \\
		\noalign{\smallskip}\hline
	\end{tabular}}
	\caption{Details of medium-scale general objective test problems from the CUTEst test set, including the value of $f(\bx^*)$ used in \eqref{eq_solved_threshold} for each problem. Some problems are variable-dimensional; the relevant parameters yielding the given $n$ are provided. Problems marked * have bound constraints. The value of $n$ shown excludes fixed variables. Some of the problems were taken from \cite{Luksan2010}.}
	\label{tab_genobj_problems}
\end{table}

\end{document}